\documentclass[reqno]{amsart}
\usepackage{amsthm,amsmath,amssymb,amscd, mathrsfs,graphics, hyperref,enumitem,tikz}
\usepackage{diagrams}
\usepackage[all,cmtip]{xy} 
\usepackage{verbatim} 

\newarrow {Mapsto} |--->
\newarrow {Iso} ---->

\newcommand{\AGor}{{\check A}} 
 
\newcommand{\BG}{\B G}
\newcommand{\B}{B}

\newcommand{\Ch}{\mathrm{Ch}}

\newcommand{\cD}{\mathcal{D}}  
\newcommand{\cT}{\mathcal{T}}

\newcommand{\Dt}{\widetilde{\mathbf{D}}}

\newcommand{\IGX}{I_GX}

\newcommand{\inCh}{\ovCh} 
\newcommand{\Irrep}{\operatorname{Irrep}} 

\newcommand{\I}{\mathbf{I}} 
 
\newcommand{\IzeroX}[1]{I_G^{0,V}X}

\newcommand{\KGor}{{\check K}} 
\newcommand{\Kh}{\widehat{K}}

\newcommand{\Ks}{\ck}

\newcommand{\N}{\mathbf{N}}              
\newcommand{\ord}{\operatorname{ord}}
 
\newcommand{\oPZ}{\mathbf{P}} 
 
\newcommand{\PP}{\mathbf{P}}

\newcommand{\cPZ}{\mathbf{P}}

\newcommand{\Rep}{\operatorname{Rep}} 
\newcommand{\R}{\mathscr{R}}             
\newcommand{\rl}{\mathcal{R}} 
\newcommand{\rlb}{\overline{\rl}}
\newcommand{\Sm}{\mathscr{S}}             
\newcommand{\Smv}{^V\kern-.5em\mathscr{S}}

\newcommand{\T}{\mathbb{T}}    
\newcommand{\TV}{^V\kern-.4em\mathscr{T}}

\newcommand{\age}{\operatorname{age}}

\newcommand{\augi}{{\mathfrak a}}
\newcommand{\aug}{\epsilon}

\newcommand{\bE}{\mathbf{E}}

\newcommand{\bone}{\mathbf{1}} 
\newcommand{\bra}[1]{ {\{ {#1} \}} }

\newcommand{\cF}{{\cf}}

\newcommand{\cP}{\mathcal{P}}

\newcommand{\cf}{{\mathscr F}}
\newcommand{\cg}{\mathscr{G}}

\newcommand{\ck}{\mathscr{K}}

\newcommand{\cl}{{\mathscr L}}

\newcommand{\co}{{\mathscr O}}

\newcommand{\crr}{{\mathscr R}}

\newcommand{\cs}{{\mathscr S}}

\newcommand{\cv}{{\mathscr V}}
\newcommand{\cw}{{\mathscr W}}
\newcommand{\cx}{\mathscr{X}}

\newcommand{\dsand}{\quad \text{ and } \quad}

\newcommand{\euler}{\operatorname{eu}}

\newcommand{\ga}{\gamma}

\newcommand{\ix}{\stack{X}}

\newcommand{\kone}{1}

\newcommand{\lam}{\lambda}

\newcommand{\nc}{\mathbb{C}} 
\newcommand{\nk}{\mathbb{K}}

\newcommand{\nq}{\mathbb{Q}}

\newcommand{\nz}{\mathbb{Z}}
\newcommand{\oAZ}{\mathbf{A}} 
\newcommand{\oCh}{\mathscr{C}\kern-.25em{h}} 
\newcommand{\ovCh}{\widetilde{\oCh}}

\newcommand{\ovaug}{\widetilde{\epsilon}}

\newcommand{\oc}{\widetilde{c}}   
\newcommand{\oexp}{\widetilde{\operatorname{exp}}} 
\newcommand{\oga}{\widetilde{\ga}} 
\newcommand{\ocv}{\widetilde{c}}   
\newcommand{\ovc}{\ocv}   
\newcommand{\ovexp}{\widetilde{\operatorname{exp}}} 
\newcommand{\ovga}{\widetilde{\ga}}

\newcommand{\oKZ}{\mathbf{K}} 
 
\newcommand{\olam}{\widetilde{\lam}}

\newcommand{\one}{\mathbf{1}} 
\newcommand{\opsi}{\widetilde{\psi}} 
 
\newcommand{\pd}{\star}

\newcommand{\pro}{\mathbb{P}}

\newcommand{\psit}{\widetilde{\psi}}

\newcommand{\Res}{\operatorname{Res}} 
\newcommand{\res}[2]{\left.{#1}\right|_{#2}}

\newcommand{\rk}{\mathrm{rk}}

\newcommand{\sig}{\sigma}

\newcommand{\stack}[1]{\mathscr {#1}}

\newcommand{\taub}{\overline{\tau}}

\newcommand{\ti}{\star}

\newcommand{\A}{\mathbb A}

\newcommand{\Pro}{\mathbb P}
\newcommand{\Q}{\mathbb Q}

\newcommand{\Z}{\mathbb Z}

\newcommand{\Iner}{\mathfrak{I}} 

\newcommand{\IIX}{{\Iner \kern-.75em \Iner}_{\ix}}
\newcommand{\II}{\Itwo}

\DeclareMathOperator{\Pic}{Pic}

\DeclareMathOperator{\GL}{GL}

\DeclareMathOperator{\Td}{Td}
\DeclareMathOperator{\Ind}{Ind}

\DeclareMathOperator{\sgn}{sgn}

\newcommand{\Lie}{\operatorname{Lie}}

\newcommand{\Imult}[1]{{I}^{#1}}  
\newcommand{\Itwo}{\Imult{2}}

\newtheorem{quest}{Question} 
\newtheorem{result}{Main Results} 
\newtheorem{conseq}{Corollary}  

\newtheorem{thm}{Theorem}[section] 
\newtheorem{lm}[thm]{Lemma}
\newtheorem{prop}[thm]{Proposition} 
\newtheorem{crl}[thm]{Corollary}

\theoremstyle{definition}
\newtheorem{rem}[thm]{Remark} 

\newtheorem{df}[thm]{Definition} 
\newtheorem{ex}[thm]{Example}
\newtheorem{df-pr}[thm]{Definition-Proposition}

\theoremstyle{remark} 
\newtheorem{nota}[thm]{Notation} \renewcommand{\thenota}{\kern-1ex}

\begin{document}
\title[Inertial Chern Classes and Compatible Power Operations] {Chern Classes
  and Compatible Power Operations in Inertial K-theory}


\author [D. Edidin]{Dan Edidin} \address {Department of
Mathematics, University of Missouri, Columbia, MO 65211, USA}
\email{edidind@missouri.edu} 

\author [T. J. Jarvis]{Tyler J. Jarvis} \address {Department of
Mathematics, Brigham Young University, Provo, UT 84602, USA}
\email{jarvis@math.byu.edu} 

\author [T. Kimura]{Takashi Kimura} \address {Department of
Mathematics and Statistics; 111 Cummington Mall, Boston University; 
Boston, MA 02215, USA } \email{kimura@math.bu.edu} 
\thanks{Research of the first author was partially supported by a Simons Collaboration Grant. Research of the second author partially supported by NSA grant H98230-10-1-0181.
Research of the third author partially supported by NSA grant H98230-10-1-0179.}  

\date{\today}

\begin{abstract}
Let $\ix = [X/G]$ be a smooth Deligne-Mumford quotient stack. In a previous paper the authors constructed a class of exotic products called {\em 
inertial products}
on $K(I\ix)$, the Grothendieck group of vector bundles on the inertia stack
$I\ix$. In this paper we develop a theory of Chern classes and compatible power operations for inertial products. When $G$ is diagonalizable these give rise to an augmented $\lambda$-ring structure on inertial K-theory.

One well-known inertial product is the {\em virtual product}. Our
results show that for toric Deligne-Mumford stacks there is a
$\lambda$-ring structure on inertial K-theory.  As an example, we
compute the $\lambda$-ring structure on the virtual K-theory of the
weighted projective lines $\Pro(1,2)$ and $\pro(1,3)$.  We prove that
after tensoring with $\nc$, the augmentation completion of this
$\lam$-ring is isomorphic as a $\lambda$-ring to the classical
K-theory of the crepant resolutions of singularities of the coarse moduli
spaces of the cotangent bundles ${\mathbb T}^*\Pro(1,2)$ and ${\mathbb
  T}^*\Pro(1,3)$, respectively.  We interpret this as a manifestation
of mirror symmetry in the spirit of the Hyper-K\"ahler Resolution
Conjecture.
\end{abstract}
\maketitle \setcounter{tocdepth}{1}

\section{Introduction}

The work of Chen and Ruan \cite{CheRu:02}, Fantechi-G\"ottsche \cite{FaGo:03}, and Abramovich-Graber-Vistoli \cite{AGV:02, AGV:08} defined
orbifold products for the cohomology, Chow groups and K-theory
of the inertia stack $I\ix$ of a smooth
Deligne-Mumford stack $\ix$.  Moreover, there is an orbifold Chern
character $\oCh\colon K(I\ix) \to A^*(I\ix)_\nq$ which respects these
products \cite{JKK:07}.  In \cite{EJK:12a} we showed that the orbifold
product and Chern character fit into a more general formalism of {\em inertial products}
which are discussed later in the introduction.

In this paper, we are motivated by mirror symmetry to find examples of 
elements in orbifold and inertial algebraic K-theory 
that play a role analogous to 
classes of vector bundles in the ordinary algebraic K-theory.
Each such element should possess \emph{orbifold Euler
  classes} 
analogous to the classically defined classes $\lambda_{-1}({\mathcal E^*})$
and $c_{r}({\mathcal E})$ for vector bundles of rank $r$.
This leads us to
introduce the notions of an \emph{orbifold $\lam$-ring} and associated
\emph{Adams (or power) operations} which are suitably compatible with
\emph{orbifold Chern classes},  as we now explain.

Let $K(\cx)$ be the Grothendieck group of locally free sheaves on $\cx$ 
with multiplication given by the ordinary tensor product.
By
definition, $K(\cx)$ is generated by classes of vector bundles and 
each such class
possesses an Euler class.
In the context of mirror symmetry we may
be given a ring $K$ which is conjectured to be the ordinary K-theory of 
of some unknown variety.
From the ring structure alone there is no way to solve 
the problem of identifying the elements of $K$ which
correspond to Chern classes of vector bundles on this unknown variety.
However, a partial
solution arises from observing that ordinary K-theory
has the additional structure of a \emph{$\lam$-ring}. 
Every
$\lam$-ring has an associated invariant---the semigroup of
\emph{$\lam$-positive elements} (Definition \ref{df:LamPositive}) which share many of
the properties of classes of vector bundles in ordinary K-theory. In
particular, $\lam$-positive elements have Euler classes defined in terms 
of the $\lam$-ring structure.
In the case of ordinary K-theory of a scheme or stack,
classes of vector bundles are always $\lam$-positive, but there are other
$\lam$-positive classes as well.

Endowing the orbifold
K-theory ring with the structure of a $\lam$-ring with respect to its
orbifold product allows one to identify its semigroup of
$\lam$-positive elements. Furthermore, defining suitably
compatible orbifold Chern classes, should give these $\lam$-positive elements
orbifold Euler classes in orbifold K-theory, orbifold
Chow theory, and orbifold cohomology theory. These $\lam$-positive
elements can be regarded as building blocks of orbifold K-theory.

We prove the following results about smooth quotient
stacks $\ix = [X/G]$ where $G$ is a linear algebraic group acting with
finite stabilizer on a smooth variety
$X$.

\begin{result} \label{thm.orbifoldlam}

\ 
\begin{enumerate}
\item[(a)] If 
$\ix$ is Gorenstein, then there is
an orbifold Chern class homomorphism $c_t \colon K(I{\stack  X}) \to 
A^*(I{\stack  X})_\nq[[t]]$ (see Definition~\ref{df.ochernclass} and Theorem~\ref{thm:oChClassProps}).

\item[(b)] If 
$\ix$ is {\em strongly Gorenstein}
(see Definition~\ref{df:Gorenstein}), then there are Adams $\psi$-operations
and $\lambda$-operations defined on $K(I\ix)$
(resp. $K(I\ix)_\nq$) 
compatible with the Chern class homomorphism (see Definitions~\ref{df.opsi} and \ref{df.olam} as well as Theorem~\ref{thm:oChClassProps}).

\item[(c)] If $G$ 
is diagonalizable and $\ix$ is strongly Gorenstein,
then the Adams and $\lambda$-operations 
make $K(I\ix)_{\Q}:= K(I\ix) \otimes \Q$ with
its orbifold product into a 
rationally augmented 
$\lambda$-ring (see Theorem~\ref{thm:OrbifoldPsiRing}).

\item[(d)] If the orbifold $\ix$ is strongly Gorenstein, then there is an \emph{inertial dual} operation $\cf\to \cf^\dagger$ on $K(\ix)$ which is an involution and a ring homomorphism and which commutes with the orbifold Adams operations and the orbifold augmentation (see Theorem~\ref{prop.inertialdual}).

\end{enumerate}
\end{result}

Our method of proof is based on developing properties of {\em inertial
  pairs} defined in \cite{EJK:12a} . An inertial pair $(\R, {\cs})$
consists of a vector bundle $\R$ on the double inertia stack $\II\ix$
together with a class ${\cs} \in K(\ix)_\nq$, where $\R$ and $\cs$ satisfy certain
compatibility conditions.  The bundle ${\R}$ determines associative
{\em inertial products} on $K(I\ix)$ and $A^*(I\ix)$, and the class
${\cs}$ determines 
a Chern character homomorphism of inertial rings $\oCh \colon K(I\ix) \to
A^*(I\ix)_{\nq}$.

The basic example of an inertial pair $({\R}, {\cs})$ is the orbifold
obstruction bundle ${\R}$ and the class ${\cs}$ defined in
\cite{JKK:07}.  This pair corresponds to the usual orbifold product.
However, this is far from
being the only example.  Each 
vector bundle $V$ on $\ix$ determines
two inertial pairs, $({\R}^+V, {\cs}^+V)$ and $({\R}^-V,
{\cs}^-V)$. 
For example, if we denote the tangent bundle of $\ix$ by ${\mathbb T}$, then the inertial pair $({\R}^-{\mathbb T},
{\cs}^-{\mathbb T})$ produces the virtual orbifold product of
\cite{GLSUX:07}.

We prove that the main results listed above hold for many
inertial pairs.  
As a corollary, we obtain the following: 
\begin{conseq}\hfill
\begin{enumerate}
\item[(a)] The virtual orbifold product on $K(I\ix)$ admits a Chern 
series homomorphism $\ocv_t:K(I\ix)\to A^*(I\ix)_\nq[[t]]$ 
as well as compatible Adams $\psi$-operations and $\lambda$-operations on 
$K(I\ix)_\nq$.

\item[(b)] If $\ix = [X/G]$ with $G$ diagonalizable, 
then the virtual orbifold $\lambda$-operations 
make $K(I\ix)_{\nq}$ with
its orbifold product into a 
rationally augmented 
$\lambda$-ring with a compatible inertial dual.
\end{enumerate}
\end{conseq}

Whenever an inertial K-theory ring has a $\lam$-ring
structure compatible with its inertial Chern classes and inertial
Chern character, then its semigroup of $\lam$-positive elements will
have an \emph{inertial Euler class} in K-, Chow, and cohomology theory
(see Equation (\ref{eq:ovcvectorbundle})),
but where all products, rank,
Chern classes, and the Chern character are the inertial
ones. Furthermore, in many cases, the semigroup of $\lam$-positive
elements in inertial K-theory 
can be used to 
give a nice presentation of both the inertial K-theory ring and
inertial Chow ring.

A major motivation for the work in this paper is mirror symmetry.
Beginning with the work of Ruan, a series of conjectures have been
made that relate the orbifold quantum cohomology and Gromov-Witten
theory of a Gorenstein orbifold to the corresponding quantum
cohomology and Gromov-Witten theory of a crepant resolution of
singularities of the orbifold \cite{CoRu:13}.  When the orbifold also
has a holomorphic symplectic structure, these conjectures predict
that the orbifold cohomology ring 
should be isomorphic 
to the usual cohomology of a crepant resolution. In the
literature this conjecture is often referred to as
Ruan's Hyper-K\"ahler resolution conjecture (HKRC), because in many examples
the holomorphic symplectic structure is in fact Hyper-K\"ahler.

In view of Ruan's HKRC conjecture, it is natural to investigate whether there is an orbifold $\lambda$-ring structure on orbifold K-theory that is isomorphic to the usual $\lambda$-ring structure on $K(Z)$.
One place to look 
is on the cotangent bundles
of complex manifolds and orbifolds. These naturally carry a holomorphic symplectic
structure, and in many cases these are hyper-K\"ahler.
In \cite{EJK:12a} we prove that if $\ix = [X/G]$, then the virtual orbifold
Chow ring of $I\ix$ (as defined in \cite{GLSUX:07}) is isomorphic
to the orbifold Chow ring of $T^*I\ix$. Since
the inertial pair defining the virtual orbifold product is strongly
Gorenstein, we expect that the $\lambda$-ring structure 
on $K(I\ix)$ should  be
related to the usual
$\lambda$-ring structure on $K(Z)$.

When $\ix$ is an orbifold, $K(I\ix)$ typically has larger rank as
an Abelian group than the corresponding Chow group $A^*(I\ix)$, while
$K(Z)$ and $A^*(Z)$ have the same rank by the Riemann-Roch theorem for
varieties. Thus, it is not reasonable to expect an isomorphism of
$\lambda$-rings between $K(I\ix)$ with the virtual product and
$K(Z)$ with the tensor product.

But the Riemann-Roch theorem for Deligne-Mumford stacks implies
that a summand $\Kh(I\ix)_\nq$, corresponding to the completion at the classical
augmentation ideal in $K(I\ix)_{\nq}$, is isomorphic as an Abelian
group to $A^*(I\ix)_{\nq}$.  We prove the remarkable
result (Theorem \ref{thm.completions}) that if
$(\R,\cs)$ is any inertial pair, then the classical augmentation ideal
in $K(I\ix)_{\nq}$ and inertial augmentation ideal generate the same topology
on the Abelian group $K(I\ix)$. It follows that the summand $\Kh(I\ix)$
inherits any inertial $\lambda$-ring structure from $K(I\ix)$.

This allows us to formulate 
a  
$\lambda$-ring variant of the HKRC for
orbifolds $\ix = [X/G]$ with $G$ diagonalizable. Precisely, we expect
there to be an isomorphism of $\lambda$-rings
(after tensoring with $\nc$)
 between $\Kh(I\ix)$ with
its virtual orbifold product and $K(Z)$, where $Z$ is a hyper-K\"ahler
resolution of the cotangent bundle $\T^*\ix$.

We conclude 
by proving this conjecture
for the weighted projective line $\pro(1,n)$ for $n=2,3$.
  We also obtain an isomorphism of Chow rings
$(A^*(I\pro(1,n))_{\nc},\star_{virt}) \cong A^*(Z)_{\nc}$ 
commuting
with the corresponding Chern characters.
Furthermore, we show that the semigroup of inertial $\lam$-positive elements 
induces an exotic integral lattice structure on $(K(I\pro(1,n))_\nc,\star_{virt})$ 
and 
$(A^*(I\pro(1,n))_\nc,\star_{virt})$ 
which corresponds to the ordinary integral
lattice in $K(Z)_\nc$ and $A^*(Z)_\nc$, respectively.

Finally, our analysis suggests the following interesting question.
\begin{quest} 
Is there
category associated to the crepant resolution $Z$ whose
Grothendieck group (with $\nc$-coefficients) is isomorphic as a $\lam$-ring 
to the virtual orbifold K-theory 
$(K(I\cx)_\nc,\star_{virt})$ before completion at the augmentation ideal?
\end{quest}
\begin{rem}
It has subsequently been shown \cite{KS:12}  that the results (namely Propositions \ref{prop:POneTwoVirtualK}, \ref{prop:POneThreeVirtualK}, and 
Theorem \ref{prop:HKRC}) 
 in this paper for the virtual K-theory of $\pro(1,n)$ for $n=2,3$ generalize to all $n$. 
 This verifies the conjectured relationship between the virtual
 K-theory ring and the K-theory of the crepant resolution $Z_n$ of
 $T^*\pro(1,n)$ for all $n$.
\end{rem}
\subsection{Outline of the paper}  
We begin by briefly reviewing the results of \cite{EJK:10, EJK:12a} on inertial pairs, inertial products, and inertial Chern characters.

We then briefly recall the classical  $\lam$-ring and $\psi$-ring structures in ordinary equivariant K-theory, including the Adams (power) operations, Bott classes, Grothendieck's $\gamma$-classes,
and some relations among these and the Chern classes.

For Gorenstein inertial pairs we define a theory of Chern
classes and, for strongly Gorenstein inertial pairs, power (Adams) operations
on inertial K-theory. Since the inertial pair associated to the
virtual product of \cite{GLSUX:07} is always strongly Gorenstein, this produces Chern
classes and power operations in that theory.

We show that for strongly Gorenstein inertial pairs, the inertial
Chern classes satisfy a relation like that for usual Chern classes,
expressing the Chern classes in terms of the orbifold
$\psi$-operations and $\lam$-operations. Finally we prove that if $G$
is diagonalizable, the orbifold Adams operations are homomorphisms
relative to the inertial product. This shows that the virtual K-theory
of a toric Deligne-Mumford stack has $\psi$-ring and $\lambda$-ring
structures.  We also give an example to show that the
diagonalizability condition is necessary for obtaining a
$\lambda$-ring structure.

We then develop the theory of $\lambda$-positive elements for a $\lam$-ring
and show that $\lam$-positive elements of degree $d$ share many of the same
properties as classes of rank-$d$ vector bundles; for example, they have a
top Chern class in Chow theory and an Euler class in K-theory.  
We also introduce the notion of an inertial dual which 
is needed to define
the Euler class in inertial K-theory.  

We conclude by working through some examples, including that of 
$\B\mu_2$, 
and the virtual K-theory of the weighted projective lines $\pro(1,2)$ and $\pro(1,3)$.  

The $\lambda$-positive elements, and especially the $\lambda$-line
elements in the virtual theory, allow us to give a simple presentation
of the K-theory ring with the virtual product and a simple description
of the virtual first Chern classes. This allows us to prove that the
completion of this ring with respect to the augmentation ideal is
isomorphic as a $\lam$-ring to the usual K-theory of the resolution of
singularities of the cotangent orbifolds $T^*\pro(1,2)$ and
$T^*\pro(1,3)$, respectively.

\subsection*{Acknowledgments}
We wish to thank the \emph{Algebraic Stacks: Progress and Prospects}
Workshop at BIRS for their support where part of this work was
done. TJ wishes to thank Dale Husem\"oller for helpful
conversations and both the Max Planck Institut f\"ur Mathematik in
Bonn and the Institut Henri Poincar\'e for their generous support of
this research.  TK wishes to thank Yunfeng Jiang and
Jonathan Wise for helpful conversations and the Institut Henri
Poincar\'e, where part of this work was done, for their generous
support.  We also would like to thank Ross Sweet for help with
proofreading the manuscript.
Finally, we thank the referee for many helpful suggestions that have substantially improved the paper.

\section{Background material}
To make this paper self contained, we recall some background material from the papers \cite{EJK:10,EJK:12a}, but first we establish some notation and conventions. 

\subsection{Notation}

We work entirely in the complex algebraic category. We will work
exclusively 
with a smooth Deligne-Mumford stack $\ix$ with \emph{finite
stabilizer}, by which we mean the inertia
map $I\ix \to \ix$ is finite (see Definition~\ref{def:inertia} for the formal definition and more detail). We will also assume that every stack
$\ix$ has the \emph{resolution property}. This means that every
coherent sheaf is the quotient of a locally free sheaf. This
assumption has two consequences. The first is that the natural map
$K(\ix) \to G(\ix)$ is an isomorphism, where $K(\ix)$ is the
Grothendieck ring of vector bundles, and $G(\ix)$ is the Grothendieck
group of coherent sheaves. The second consequence  is
that $\ix$ is a \emph{quotient stack} \cite{Tot:04}. This means that $\ix = [X/G]$, 
where $G$ is a linear algebraic group acting on an affine scheme $X$.

If $\ix$ is a smooth Deligne-Mumford stack, we will 
explicitly  
choose a presentation $\ix = [X/G]$. This allows us to identify the
Grothendieck ring $K(\ix)$ with the equivariant Grothendieck ring
$K_G(X)$, and the Chow ring $A^*(\ix)$ with the equivariant
Chow ring $A^*_G(X)$. We will use the notation $K(\ix)$ and 
$K_G(X)$ (respectively $A^*(\ix)$ and  $A^*_G(X)$)  interchangeably.

\begin{df}\label{def:inertia}
  Let $G$ be an algebraic group acting on a scheme 
  $X$. We define the \emph{inertia 
scheme}  
$$\IGX:= \{(g,x)| gx
  = x\} \subseteq G\times X.$$  There is an induced action of $G$ on
  $\IGX $ given by $g \cdot (m,x) = (gmg^{-1}, gx)$. The quotient
  stack $I\ix = [\IGX /G] $ is the {\em inertia 
stack}
  of the quotient  ${\stack X} := [X/G]$.

  More generally, we define the higher inertia spaces to be the
  $k$-fold fiber products $$\Imult{k}_G X = \IGX  \times_X \ldots \times_X
  \IGX .$$
The quotient stack $\Imult{k}\ix := \left[\Imult{k}_G X /G\right]$ 
is the corresponding higher inertia stack.
\end{df}
The composition $\mu \colon G \times G \to G$ induces a composition
$\mu \colon     \Itwo_G X  \to \IGX$. 
This composition makes $\IGX $ into an $X$-group
with identity section $X \to \IGX $ given by $x \mapsto (1,x)$. 
Furthermore, for $i=1,2$, the projection map $e_i:\Itwo_G X\to \IGX$ is called the \emph{$i$th evaluation map}, since it corresponds to the evaluation morphism in Gromov-Witten theory.

\begin{df}
Let $\Psi \subset G$ be a conjugacy class.  We define $I(\Psi) = \{(g,x)|gx=x, g \in \Psi\} \subset G \times X$. 
More generally, let $\Phi \subset  G^{\ell}$ be a diagonal conjugacy class.  We define $\Imult{\ell}(\Phi) = \{(m_1, \ldots m_\ell, x)| (m_1, \ldots , m_\ell) \in \Phi \text{ and } 
m_i x = x \text{ for all } i = 1,\dots,\ell
\}$.
\end{df}

By definition, $I(\Psi)$ and $\Imult{\ell}(\Phi)$ are $G$-invariant subsets of $\IGX $
and $\Imult{\ell}_G(X)$, respectively. 
Since $G$ acts 
with finite stabilizer  on $X$, the conjugacy class $I(\Psi)$ is empty unless $\Psi$
consists of elements of finite order. Likewise, $\Imult{\ell}(\Phi)$ is empty unless every $\ell$-tuple $(m_1, \ldots , m_\ell) \in \Phi$ generates a finite group. Since conjugacy classes of elements of finite order are closed, $I(\Psi)$ and $\Imult{\ell}(\Phi)$
are closed.
\begin{prop} \label{prop.inertiadecomp}
(\cite[Prop.~2.11,~2.17]{EJK:10}) 
The conjugacy class $I(\Psi)$ is empty  for all but
  finitely many $\Psi$, and each $I(\Psi)$ is a union of connected components
  of $\IGX $. Likewise, $\Imult{\ell}(\Phi)$ is empty for all but finitely
  many 
  diagonal conjugacy classes 
  $\Phi \subset G^{\ell}$,
  and each $\Imult{\ell}(\Phi)$ is a union of
  connected components of $\Imult{\ell}_G(X)$.
\end{prop}

\begin{df}
In the special case that $\Psi = (1)$ is the class of the identity element $1\in G$, the locus  $I((1)) = \{(1,x) | x\in X\} \subset \IGX$, often written $X^1$, is canonically identified with $X$.  It is an open and closed subset of $\IGX$, but is not necessarily connected.  We often call $X^1$ the \emph{untwisted sector of $\IGX$} and the other loci $I(\Psi) $ for $\Psi \neq (1)$ the \emph{twisted sectors}.

Similarly the groups $A^*_G(X^1)$  and $K_G(X^1)$ are summands of $A^*_G(\IGX )$ and $K_G(\IGX)$, respectively, and each is called the \emph{untwisted sector} of $A^*_G(\IGX )$ or $K_G(\IGX)$, respectively.  The summands of $A^*_G(\IGX)$ and $K_G(\IGX)$ corresponding to the twisted sectors of $\IGX$ are also called \emph{twisted sectors}.
\end{df}

\begin{df}\label{df:euler}
If $E$ is a $G$ equivariant vector bundle on $X$,
the element
$\lambda_{-1}(E^*) = \sum_{i=0}^{\infty} (-1)^i  [\Lambda^i E^*] \in K_G(X)$ 
 is called the \emph{K-theoretic Euler
    class} of $E$. 
(Note that this sum is finite.)

Likewise, we define the \emph{Chow-theoretic
  Euler class} of $E$ to be the element $c_{\textbf{top}}( E) \in A^*_G(X)$,
corresponding to the sum of the top Chern classes of $E$ on
each connected component of $[X/G]$ (See \cite{EdGr:98} for the definition and properties of equivariant Chern classes).  These definitions can be extended to
any nonnegative element by multiplicativity.  It will be convenient
to use the symbol $\euler(\cf)$ to denote both of these Euler classes for a
nonnegative element $\cf \in K_G(X)$.
\end{df}

\paragraph{{\bf Rank and augmentation homomorphisms.}}
If $[X/G]$ is connected, then then the rank of a vector bundle defines
an augmentation homomorphism $\epsilon \colon K_G(X) \to \Z$. If we
denote by $\one$ the class of the trivial bundle on $X$, then
the decomposition of an element $x = \epsilon(x)\one + (x- \epsilon(x)\one)$ 
gives a decomposition of $K_G(X)$ into a sum of $K_G(X)$-modules
$K_G(X) = \Z + I$, where $I = \ker (\epsilon)$ is the augmentation ideal.
From this point of view, we can equivalently define the augmentation
as the projection 
endomorphism $K_G(X) \to K_G(X)$ given by $x \mapsto \rk(x) \one$,
where $\rk$ is the usual notion of rank for classes in equivariant K-theory.

Since we frequently work with a group $G$ acting on a space $X$ where
the quotient stack $[X/G]$ is not connected, 
some care is required in the definition of the rank 
of a vector
bundle.  Note that for any $X$, the group $A^0_G(X)$ satisfies $A^0_G(X) =\Z^{\ell}$, 
where $\ell$
is the number of connected components of the quotient stack $\ix = [X/G]$.  Since 
$\ix$ has finite type, $\ell$ is finite.

\begin{df}
  Any $\alpha \in K_G(X)$ uniquely determines an element
$\alpha_U$ 
of $K(U)$ on each connected component $U$ of
  $[X/G]$.  
If we fix an ordering of the components, then we define the \emph{rank} of $\alpha$ to be the
  $\ell$-tuple in $\Z^{\ell} = A^0_G(X)$ whose component in the factor
  corresponding to a connected component $U$ is the usual rank of
$\alpha_U$.  This agrees with the degree-zero part of the
  Chern character:
\[
\rk(\alpha) := \Ch^0({\alpha})
\in A^0_G(X) = \Z^{\ell}.
\]
\end{df}

In this paper, where we study exotic $\lambda$ and $\psi$-ring
structures on equivariant K-theory of $K_G(I_GX)$, we will need
to define corresponding
exotic augmentations. To facilitate 
their definitions we introduce the more general notion of an augmented ring. 
\begin{df}\label{df:AugRing} (cf. \cite[p.143]{CaEi:56}).
An
 \emph{augmentation homomorphism}  of a ring $R$ 
is
  an endomorphism $\aug$ of  $R$ that is 
a
projection, 
i.e.,
$\aug \circ \aug = \aug$.
The kernel of $\aug$ is called the \emph{augmentation ideal} of $R$.
The ring $R$ is said to be a \emph{ring with augmentation}. 
\end{df}

\begin{rem}
  In the language of \cite[p.~143]{CaEi:56}, the image of $\aug$ is
  called the augmentation module. Our definition is more
  restrictive than that of \emph{loc.~cit.}, since it
  requires that $R$ split as $R = \epsilon(R) + I$ where $\epsilon(R)$
  is the augmentation module and $I$ is the augmentation ideal.

Note that all rings have two trivial augmentations coming from the identity and zero homomorphisms. 
However, in our applications, $\aug$ will preserve unity in $R$. 
\end{rem}

We illustrate the use of this terminology by defining an augmentation
homomorphism on $K_G(Y)$ when $[Y/G]$ is not necessarily connected.

\begin{df} \label{df.augequivkt}
In equivariant K-theory we define the  \emph{augmentation homomorphism} $\aug:K_G(Y)\to K_G(Y)$ to be
the map which, for each connected component $[U/G]$ of $[Y/G]$, sends each $\cf$ in $K_G(Y)$
supported on $U$ to the 
rank of $\cf$ times the structure sheaf $\co_U$
$$\aug(\cf|_{U}) :=\Ch^0(\cf|_U)\co_U.$$   
\end{df}

Thus, 
for equivariant K-theory, the image of $\aug$ is isomorphic as a ring to 
$\nz^{\oplus \ell}$, 
where $\ell$ is the number of connected components of $[Y/G]$. 
However,
we will see that this property need not hold for inertial K-theory.

\subsection{Inertial products, Chern characters, and inertial pairs}
We review here the results from \cite{EJK:12a}, defining a
generalization of orbifold cohomology, obstruction bundles, age
grading, and stringy Chern character, by defining \emph{inertial
  products} on $K_G(\IGX )$ and $A^*_G(\IGX )$ using \emph{inertial
  pairs} $(\R,\cs)$, where 
$\R$ is a $G$-equivariant vector bundle on $\II_G X$
and $\cs\in K_G(\IGX)_\nq$ is a nonnegative class satisfying certain compatibility properties.

For each such pair, there is also a rational grading 
on the total Chow group, and a
Chern character ring homomorphism. There are many inertial pairs, and
hence there are many associative inertial products on $K_G(\IGX )$ and
$A^*_G(\IGX )$ with rational gradings and Chern character ring
homomorphisms.  The orbifold products on $K(I\ix)$ and $A^*(I\ix)$ and
the Chern character homomorphism of \cite{JKK:07} are a special case,
as is the virtual product of \cite{GLSUX:07}.

\begin{df}
If $\R$ is a vector bundle on $\II_G X$, we define
products on $A^*_G(\IGX )$ and $K_G(\IGX )$ via the following 
formula: 
\begin{equation} \label{eq.inertialprod}
x \star_\R y := \mu_*\left(e_1^*x \cdot e_2^*y \cdot \euler(\R)\right),
\end{equation}
where $x,y \in A^*_G(\IGX )$ (respectively $K_G(\IGX )$),
where $\mu \colon \II_GX 
\to \IGX$ is the composition map,
and $e_1, e_2 \colon 
\II_GX 
\to \IGX$ are the evaluation maps.
\end{df}

To define an inertial pair requires a little more notation from
\cite{EJK:10}, which we recall here.
Consider $(m_1, m_2,m_3) \in G^3$ such that $m_1m_2m_3 = 1$, and let 
$\Phi_{1,2,3} $ be the conjugacy class of $(m_1,m_2,m_3)$. 
Let $\Phi_{12,3}$
be the conjugacy class of $(m_1m_2,m_3)$ and $\Phi_{1,23}$ the
conjugacy class of $(m_1,m_2m_3)$.  Let $\Phi_{i,j}$ be the conjugacy
class of the pair $(m_i,m_j)$ with $i < j$. Finally, 
let $\Phi_{ij}$
be the conjugacy class of $m_im_j$, and 
let $\Phi_{i}$
be the conjugacy class of $m_i$.  There are composition maps $\mu_{12,3} \colon {\Imult3}(\Phi_{1,2,3}) \to
\II(\Phi_{12,3})$, and $\mu_{1,23} \colon {\Imult3}(\Phi_{1,2,3}) \to
\II(\Phi_{1,23})$.
The various maps we have defined are related by the
following Cartesian diagrams where all maps are l.c.i.~morphisms.
\begin{equation} \label{diag.excess12}
\begin{diagram}
{\Imult3}(\Phi_{1,2,3}) &\rTo^{e_{1,2}} & \II(\Phi_{1,2}) \\
\dTo^{\mu_{12,3}}  & &  \dTo_{\mu}\\
\II(\Phi_{12,3}) & \rTo^{e_1} & 
I(\Phi_{12})
\end{diagram}
\qquad 
\begin{diagram}
{\Imult3}(\Phi_{1,2,3}) &\rTo^{e_{2,3}}  & \II(\Phi_{2,3}) \\
\dTo^{\mu_{1,23}} & & \dTo_{\mu}\\
\II(\Phi_{1,23}) & \rTo^{e_1} & 
I(\Phi_{23})
\end{diagram}
\end{equation}
Let $E_{1,2}$ and $E_{2,3}$ be the respective excess normal bundles of the two diagrams \eqref{diag.excess12}.

\begin{df}\label{df.inertialpair}
Given a nonnegative element $\cs \in K_G(\IGX)_\nq$ and 
$G$-equivariant vector bundle $\R$ on $\II_G X$
we say that $(\R,\cs)$ is an \emph{inertial pair} if the following conditions hold:
\begin{enumerate}[label=(\alph*)]
\item 
The identity
\begin{equation}\label{eq:DefOfR}
{\R} = e_1^* {\cs} + e_2^* {\cs} - \mu^*{\cs} + T_{\mu}
\end{equation}
holds in $K_G(\II_G X)$, 
where $T_{\mu} = T\Itwo_G X  - \mu^*(T\IGX)$ is the relative tangent bundle of $\mu$.
\item  
${\R}|_{\II(\Phi)} = 0$
for every conjugacy class
$\Phi\subset G \times G$ such that $e_1(\Phi) = 1$ or $e_2(\Phi) =1$.
\item 
$i^* {\R} = {\R},$
where $i \colon \II_G X \to \II_G X$ is the 
isomorphism
 $i(m_1,m_2,x) = (m_1m_2m_1^{-1},m_1,x)$.

\item 
$e_{1,2}^*{\R} + \mu_{12,3}^*{\R} + E_{1,2} = 
e_{2,3}^*{\R} + \mu_{1,23}^*{\R} + E_{2,3}$
for each triple $m_1,m_2, m_3$ with $m_1m_2m_3=1$.
\end{enumerate}
\end{df}

\begin{prop}[\protect{\cite[\S3]{EJK:10}}] 
If $(\R,\cs)$ is an inertial pair, then the $\star_{\R}$ product is commutative and associative with identity 
${\one}_X$, where
${\one}_X$ is the identity class in the untwisted sector $A^*_G(X^1)$ (respectively $K_G(X^1)$).
\end{prop}

\begin{prop}\cite[Prop 3.8]{EJK:12a}  
If $(\R,\cs)$ is an inertial pair, then the map 
$$\inCh \colon K_G(\IGX )_\Q \to A^*_G(\IGX )_\Q,$$ defined by 
$\inCh(V) = \Ch(V) \cdot \Td(-{\cs})$,  
is a ring homomorphism with respect to the 
$\star_{\R}$-inertial products on $K_G(\IGX )$ and $A^*_G(\IGX )$. 
\end{prop}

It is shown in \cite{EJK:12a} that there are two inertial pairs for every $G$-equivariant vector bundle on $X$.  Most of our results in this paper apply to general inertial pairs, but we have a special interest in the inertial pair associated to the \emph{orbifold product} of \cite{ChRu:04, AGV:02, FaGo:03, JKK:07, 
    EJK:10} and in the inertial pair associated to the \emph{virtual product} of \cite{GLSUX:07}.

\begin{df}\label{df:TT}  Let $p:X \to \ix$ be the quotient map, $T_\ix$ be the tangent bundle of $\ix$,
and
 $\T = p^*T_\ix$ in $K_G(X)$. 
In \cite[Lemma 6.6]{EJK:10} we proved that $\T = T_X -\mathfrak{g}$, where $\mathfrak{g}$ is the Lie algebra of $G$ and $T_X$ is the tangent bundle on $X$.
\end{df}

\begin{df}\label{df:orbifold-in-pair}
The inertial pair associated to the \emph{orbifold product} is given by the element $\cs = \cs(\T) \in K_G(\IGX)_\nq$, defined as follows.  For any $m\in G$ of finite order $r$, the element $\cs$, when restricted to $X^m 
=\{(x,m) | mx = x\} \subset \IGX$, 
is  
\begin{equation}
\cs_m := \sum_{k=1}^{r-1}\frac{k}{r}\T_{m,k},
\end{equation}
where $\T_{m,k}$ is the eigenbundle of $\T$ on which $m$ acts as $e^{2\pi i k/r}$.
The first property of inertial pairs (see Definition~\ref{df.inertialpair}.(a)) then gives the explicit formula for $\R$:
\[
{\R} = e_1^* {\cs} + e_2^* {\cs} - \mu^*{\cs} + T_{\mu}.
\]
\end{df}

\begin{df}\label{df:virtual-in-pair}
The inertial pair associated to the \emph{virtual product} is given by 
$
\cs = \N$,
where $\N$ is the quotient $q^*T_X/T_{\IGX}$ where 
$q\colon 
\IGX \to X$ is the canonical morphism, and 
\begin{equation}\label{eq:VirtualR}
\R = \T|_{\II_G X} + \T_{\II_G X} - e_1^*\T_{\IGX } - e_2^*\T_{\IGX },
\end{equation}
where $\res{\T}{\II_G X}$ refers to the pullback of the bundle $\T$ to $\II_G X$
via 
the natural map $\II_G X \to X$, where $\T_{\IGX }$ denotes the
pullback to $\IGX$ of the tangent bundle to
$I\ix= [\IGX /G]$, and where $\T_{\II_G X}$ denotes the pullback to $\II_G X$ of the
tangent bundle to the stack $I^2{\stack X}= [\II_G X/G]$. 
\end{df}
\begin{rem}
By abuse of notation we will refer to the bundle $\N$ defined above
as the normal bundle to the morphism $\IGX \to IX$.
\end{rem}
\begin{rem}
In \cite{EJK:12a} we showed that the pairs for both the orbifold product and the virtual orbifold product are indeed inertial pairs. 
\end{rem}

\begin{df}
Given any nonnegative element $\cs \in K_G(\IGX)_\nq$,
we define the ${\cs}$-\emph{age} on a 
component $U$ of $\IGX $ 
corresponding to a connected component $[U/G]$ of 
$[I_GX/G]$
to be the rational rank of ${\cs}$ on the component $U$:
  $$\age_{\cs}(U) = \rk(\cs)_{U}.$$
   We define the \emph{${\cs}$-degree} of an element $x \in
  A^*_G(\IGX )$ 
  on such a component
$U$ of $\IGX$
   to be 
           $$\res{\deg_\cs x}{U}  = \res{\deg x}{U} + \age_{\cs}(U),$$
   where $\deg x$ is
  the degree with respect to the usual grading by codimension on
  $A^*_G(\IGX )$.
Similarly, if $\cf\in K_G(\IGX)$ 
is supported on $U$, 
then its $\cs$-degree 
is
\[
\deg_\cs \cf = \age_{\cs}(U)\bmod\nz.
\]
This yields a $\nq/\nz$-grading of the group $K_G(\IGX)$.
\end{df}

\begin{prop}\cite[Prop 3.11]{EJK:12a} 
  If $({\R},\cs)$ is an inertial pair, then the ${\R}$-inertial
  products on $A^*_G(\IGX)$ and $K_G(\IGX)$ respect the $\cs$-degrees. Furthermore, the inertial Chern character homomorphism $\inCh:K_G(\IGX)\to A_G^*(\IGX)$ preserves the $\cs$-degree modulo $\nz$.
  \end{prop}

\begin{df}
Let $A_G^{\bra{q}}(\IGX)$ be the subspace in $A_G^*(\IGX)$ of elements with an $\cs$-degree of $q\in \nq^\ell$, where $\ell$ is the number of connected components of $I\ix$.
\end{df}

\begin{df} \label{df.inertialaugkt}
Given a nonnegative ${\cs}\in K_G(\IGX )_\nq$, 
the 
homomorphism  
$\inCh^0:K_G(\IGX)\to A_G^{\bra{0}}(\IGX)$ is called the
\emph{inertial rank 
for $\cs$} or just the \emph{$\cs$-rank}. 

The \emph{inertial augmentation homomorphism} $\ovaug:K_G(\IGX)\to
K_G(\IGX)$ is the map which for each connected component $[U/G]$ of
$[(\IGX)/G]$ sends each $\cf$ in $K_G(\IGX)$ supported on $U$ to
\[
\ovaug(\res{\cf}{U}) = \inCh^0(\res{\cf}{U})
\co_{U}.
\]
\end{df}
Hence, if $\star$ is an inertial product associated to an inertial pair
$({\R}, {\cs})$, then 
$(K_G(\IGX),\star,1,\ovaug)$ is 
a ring with augmentation.
\begin{rem}
Note that
the restriction $\res{\inCh^0(\cf)}{U}$ of the inertial rank to a 
component is equal to the classical rank if the 
${\cs}$-age 
of that component is zero, and $\res{\inCh^0(\cf)}{U}$ vanishes if the age is nonzero. 
Hence the product $\inCh^0(\res{\cf}{U})
\co_{U}$ makes sense.
\end{rem}
\begin{df}\label{df:Gorenstein}
    An inertial pair $({\R}, {\cs})$ is
  called \emph{Gorenstein} if ${\cs}$ has integral 
   rank
    and \emph{strongly Gorenstein} if ${\cs}$ is represented by a vector
  bundle.

The Deligne-Mumford stack ${\stack X} = [X/G]$ is \emph{strongly Gorenstein}
if the inertial pair associated to the orbifold product (as in Definition~\ref{df:orbifold-in-pair}) is strongly Gorenstein.
\end{df}
Note that the inertial pair for the virtual product is always strongly Gorenstein.

\section{Review of $\lam$-ring and $\psi$-ring structures in equivariant K-theory}

In this section, we review the $\lam$-ring and $\psi$-ring structures in
equivariant K-theory and describe the Bott cannibalistic classes $\theta^j$, as well as the Grothendieck $\gamma$-classes.  The main theorems about these classes are the Adams-Riemann-Roch Theorem (Theorem~\ref{thm:ARR}) and Theorem~\ref{thm:chpsigamma}, which describes relations among the Chern Character, the $\psi$-classes, the Chern classes, and the $\gamma$-classes.  

Recall that a $\lam$-ring is a commutative ring $R$ with unity $1$ and with a map 
$\lam_t:R\to R[[t]]$, where 
\begin{equation}\label{eq:Lambdai}
\lam_t(a) =: \sum_{i\geq 0} \lam^i(a) t^i,
\end{equation}
such that 
the following are satisfied for all $x,y$ in $R$ and
for all integers $m, n\geq 0$: 
\[
\lam^0(x) = 1, \quad
\lam_t(1) = 1 + t, \quad
\lam^1(x) = x, \quad
\lam_t(x+y) = \lam_t(x)\lam_t(y),
\]
\begin{equation}\label{lamring.Pn}
\lam^n(x y) =  \PP_n(\lam^1(x),\ldots,\lam^n(x),\lam^1(y),\ldots,\lam^n(y)),
\end{equation}
\begin{equation}\label{lamring.Pmn}
\lam^m(\lam^n(x)) =  \PP_{m,n}(\lam^1(x),\ldots,\lam^{mn}(x)),
\end{equation}  
where $\PP_{n}$, and $\PP_{m,n}$ are certain universal polynomials, independent of $x$ and $y$ (see \cite[\S I.1]{FuLa:85}).

\begin{df} \label{df.lamalg}
If a $\lam$-ring $R$ is a $\nk$-algebra,
where $\nk$ is a field of characteristic 0, 
then we call $(R,\cdot,1,\lam)$ a \emph{$\lam$-algebra over $\nk$} if, 
for all $\alpha$ in $\nk$ and all $a$ in $R$, we have
\begin{equation}\label{eq:kAlgLambda}
\lambda_t(\alpha a) = \lambda_t(a)^\alpha
: = \exp(\alpha \log \lambda_t(a)).
\end{equation}
Note that $\log \lambda_t$ makes sense because any series for $\lambda_t$ starts with $1$.
\end{df}

\begin{rem}\label{rem:UniversalPoly}
The significance of the universal polynomials in the definition   of a
$\lam$-ring is that one can calculate $\lam^n(x y)$ and   $\lam^m(\lam^n(x))$
in terms of $\lam^i(x)$ and $\lam^j(y)$ by applying a formal splitting principle.

For example, suppose we wish to express $\lam_t(x\cdot y)$ in terms of $\lam_t(x)$ and
$\lam_t(y)$. First, replace $x$ by the formal sum $x\mapsto \sum_{i=1}^\infty x_i$, 
where we assume that $\lam_t(x_i) = 1 + t x_i$ for all $i$,
and similarly replace $y$ by the formal sum $y \mapsto \sum_{i=1}^\infty y_i$ in $\lam_t(x\cdot y)$,
where we assume that $\lam_t(y_i) = 1 + t y_i$ for all $i$. 
The fact that $\lam_t(x_i) = 1 + t x_i$ and $\lam_t(y_j) = 1 + t y_j$ means that 
$\lam_t(x_i y_j) = 1 + t x_i y_j$, and multiplicativity gives us
\[
\lam_t(x\cdot y) = \prod_{i,j=1}^\infty (1 + t x_i y_j).
\]
Therefore, $\lam^n(x\cdot y)$ corresponds to 
the $n$th elementary symmetric function $e_n(x y)$ in the 
variables $\{ x_i y_j\}_{i,j=1}^\infty$, but $e_n(x y)$ 
can be uniquely expressed as a polynomial
$\PP_n$ in the variables $\{ e_1(x),\ldots,e_n(x), e_1(y),\ldots,e_n(y)\}$,
where $e_q(x)$ denotes the $q$th elementary symmetric function in the $\{
x_i \}_{i=1}^\infty$ variables and  $e_r(y)$ denotes the $r$th elementary
symmetric function in the $\{ y_i \}_{i=1}^\infty$ variables.  Replacing
$e_q(x)$ by $\lam^q(x)$ and $e_r(y)$ by $\lam^r(y)$ in $\PP_n$ for all
$q,r\in\{1,\ldots,n\}$ yields the universal polynomial
$\PP_n(\lam^1(x),\ldots,\lam^n(x),\lam^1(y),\ldots,\lam^n(y))$ appearing in
the definition of a $\lam$-ring. 

A  similar analysis holds for $\PP_{m,n}$. 
\end{rem}

A closely related structure is that of a $\psi$-ring.
\begin{df}\label{df:PsiOps}
A commutative ring $R$ with unity $1$, together with a
collection of ring homomorphisms $\psi^n:R\to R$ for each $n\geq 1$, is called a $\psi$-ring if, 
for all $x,y$ in $R$ and for all integers $n\geq 1$, we have 
\[\psi^1(x) = x, \qquad \text{ and } \qquad  \psi^m(\psi^n(x)) = \psi^{m n}(x).
\]
The map $\psi^i:R\to R$ is called the\emph{ $i$th Adams operation (or power operation)}.

If the $\psi$-ring $(R,\cdot,1,\psi)$ is a $\nk$-algebra, 
then $(R,\cdot,1,\psi)$ is said to be a \emph{$\psi$-algebra over $\nk$} if, in addition, $\psi^n$ is a $\nk$-linear map.
\end{df}

\begin{thm}[cf. \cite{Knu:73} p.49]\label{thm:LambdaFromPsi}
Let $(R,\cdot,1,\lam)$ be a commutative $\lam$-ring, and let $\psi_t:R\to R[[t]]$ be given by
\begin{equation}\label{eq:PsiFromLambda}
\psi_{t} = -t\frac{d\log\lam_{-t}}{dt}.
\end{equation} 
Expanding $\psi_t$ as 
$\psi_t:=\sum_{n\geq 1}\psi^n t^n$ 
defines $\psi^n:R\to R$ 
for all $n\ge1$, and the resulting ring $(R,\cdot,1,\psi)$ is a $\psi$-ring.  

Conversely, if $(R,\cdot,1,\psi)$ is a $\psi$-ring and if $\lambda_t:R_\nq\to
R_\nq[[t]]$ is defined by 
\begin{equation}\label{eq:LambdaFromPsi}
\lam_t = \exp\left(\sum_{r\geq 1} (-1)^{r-1} \psi^r \frac{t^r}{r}\right),
\end{equation} 
then
$(R_\nq,\cdot,1,\lam)$ is a $\lam$-algebra over $\nq$.
\end{thm}

It follows from the definition of
the $\psi$-operations in terms of $\lambda$-operations, from Equation (\ref{eq:PsiFromLambda}), and from the identity of Equation \eqref{lamring.Pmn}
that 
\begin{equation} 
\lam^i\circ\psi^j = \psi^j\circ\lam^i
\end{equation}
for all $i\geq 0$ and $j\geq 1$ as maps from $R\to R$.

\begin{rem} 
As in Remark 
\ref{rem:UniversalPoly},
the $k$th $\lam$-operation
$\lam^k$ corresponds to the $k$th elementary symmetric function. Equation
(\ref{eq:LambdaFromPsi}) implies that the $k$th power operation, $\psi^k$,
corresponds to the $k$th power sum symmetric function, since this equation is
nothing more than the well known
relationship between the elementary symmetric functions
and the power sums. 
\end{rem}

Let $G$ be an algebraic group acting on an algebraic space $X$. The Grothendieck ring $(K_G(X),\cdot,\kone)$ of
$G$-equivariant vector bundles on $X$ is a unital commutative ring,
where 
$\cdot$ is
the tensor product and $\kone$ 
is
 the structure sheaf $\co_X$  of $X$.

It is well known that (non-equivariant) K-theory with exterior powers is a $\lambda$-ring, and the associated 
$\psi$-ring satisfies $\psi^k(\cl) = \cl^{\otimes k}$ for all line bundles $\cl$.  A lengthy but straightforward argument shows that an equivariant version of the splitting principle holds.  One can then use the splitting principle with the fact that exterior powers (and the associated $\psi$-operations) respect $G$-equivariance to prove the following proposition.
\begin{prop}
[cf. \cite{Koc:98}, Lemma 2.4]
For any $G$-equivariant vector bundle $V$ on $X$, define $\lam^k([V])$ to be
the class  $[\Lambda^k(V)]$  of the $k$th exterior power.  This
defines a $\lam$-ring structure $(K_G(X),\cdot,\kone,\lam)$ on 
$K_G(X)$.  For any line bundle $\cl$ and any integer $k\geq 1$, the corresponding homomorphisms $\psi$ on $(K_G(X),\cdot,\kone)$
satisfy  
\begin{equation}\label{eq:AdamsOp}
\psi^k(\cl) = \cl^{\otimes k}.
\end{equation}
\end{prop}

\begin{rem}
The $\lam$-ring $K_G(X)$ has still more structure,  since any element can be
represented as a difference of vector bundles. The collection $\bE$ of classes of
vector bundles in $K_G(X)$ endows the $\lam$-ring $K_G(X)$ with a
\emph{positive structure} 
\cite{FuLa:85}. Roughly speaking, this means
that $\bE$ is a subset of the $\lam$-ring consisting of elements of
nonnegative 
rank such that any element in the ring
can be written as differences of elements in $\bE$, and for any $\cf$ of rank-$d$ in $\bE$, $\lam_t(\cf)$ is a degree $d$ polynomial in $t$, and
$\lam^d(\cf)$ is invertible (i.e.,  $\lam^d(\cf)$ is a line bundle).
Furthermore, $\bE$ is closed under addition (but not subtraction) and
multiplication; $\bE$ contains the nonnegative integers; and there are
special rank-one elements in $\bE$, namely the line bundles; and 
various other properties also hold. A positive structure on a $\lam$-ring, if it
exists, need not be uniquely determined by the $\lam$-ring structure, nor
does a general $\lam$-ring possess a positive structure.

For example, if $G = \GL_n$, then the representation ring $R(G)$ can be
identified as a subring of Weyl-group-invariant elements in the
representation ring $R(T)$, where $T$ is a maximal torus and the
$\lambda$-ring structure on $R(T)$ restricts to the usual
$\lambda$-ring structure on $R(G)$.  However, the natural set of
positive elements in $R(T)$ is generated by the characters of $T$, and
this restricts to the set of positive symmetric linear combinations of
characters which contains, but does not equal, the set of irreducible
representations of $G$.

In Section \ref{sec:positive} we will introduce a different but
related notion called a \emph{$\lam$-positive structure}, which is a natural
invariant of a $\lam$-ring. This notion which will play a central role in our
analysis of inertial K-theory.
\end{rem}

The $\lambda$- and $\psi$-ring structures behave nicely with respect to the
augmentation on equivariant K-theory (Definition \ref{df.augequivkt}).
\begin{prop} For all   $\cf$ in $K_G(X)$ and integers $n\geq 1$, we have
\begin{equation}\label{eq:AugPsi}
\aug(\psi^n(\cf)) = \psi^n(\aug(\cf)) =\aug(\cf),
\end{equation}
and
\begin{equation}\label{eq:AugLambda}
\aug(\lambda_t(\cf)) = \lambda_t(\aug(\cf)) = (1+t)^{\aug(\cf)}.
\end{equation}
\end{prop}
\begin{proof}
Assume that $[X/G]$ is connected. Equation (\ref{eq:AugLambda}) holds if
$\cf$ is a rank $d$ $G$-equivariant vector bundle on $X$ since
$\lambda^i(\cf)$ has rank ${d \choose i}$. Since $K_G(X)$ is generated 
under addition by isomorphism classes of vector bundles, the same
equation holds for all $\cf$ in $K_G(X)$ by multiplicativity of $\lambda_t$.

If $[X/G]$ is not connected,  we have the ring isomorphism $K_G(X) =
\bigoplus_{\alpha} K_G(X_\alpha)$, where the sum is over $\alpha$ such that
$[X_\alpha/G]$ is a connected component of $[X/G]$. Equation
(\ref{eq:AugLambda}) follows from multiplicativity of $\lambda_t$.
Equation (\ref{eq:AugPsi}) follows 
from (\ref{eq:AugLambda}) and
(\ref{eq:PsiFromLambda}).
\end{proof}

This motivates the following definition.
\begin{df} \label{df.augpsiring} Let $(R,\cdot,1,\aug)$ be 
a ring with augmentation.
 $(R, \cdot, 1, \psi, \aug)$ is said to be an \emph{augmented $\psi$-ring} if
 $(R,\cdot,1,\psi)$ is a $\psi$-ring, and for all integers $n> 0$ we have 
$\epsilon \circ \psi^n = \psi^n \circ \epsilon = \epsilon$ as endomorphisms
of $R$.
If $R$ is an augmented $\psi$-ring, we define $\psi^0 :=  \aug$. 
\end{df}
\begin{rem}
The definition $\psi^0 = \aug$ is
consistent with all the conditions 
in the definition of a $\psi$-ring (Definition~\ref{df:PsiOps}).  
\end{rem}
\begin{df} \label{df.auglamring}
Let $(R, \cdot, 1, \lam)$ be a $\lam$-algebra (Definition \ref{df.lamalg}) 
over $\nq$ (respectively $\nc$). Let $\epsilon \colon R \to R$ be an
augmentation which is also a $\nq$-algebra (respectively $\nc$-algebra)
homomorphism.
We say that $(R,\cdot,1,\lam,\aug)$ is an 
\emph{augmented $\lam$-algebra over $\nq$ (respectively $\nc$)}
if $ \aug(\lambda_t(\cf)) = \lambda_t(\aug(\cf)) = (1+t)^{\aug(\cf)}$ for every $\cf \in R$.
Here the expression $(1+t)^x$ for an element $x$ of the
 $\nq$-algebra $R$ means that 
 \[
 (1+t)^x := \sum_{n=0}^\infty {x \choose n} t^n,\quad \text{ where } \quad {x \choose n} := \frac{\prod_{i=0}^{n-1}(x-i)}{n!}.
\]   
\end{df}

The previous proposition implies that ordinary equivariant K-theory is an
augmented $\psi$-ring. In fact, the equivariant Chow ring is also an 
augmented $\psi$-ring.
\begin{df}
For all $n\geq 1$, the map $\psi^n:A_G^*(X)\to A_G^*(X)$ defined by
\begin{equation}\label{eq:ChowPsi}
\psi^n(v) = n^d v
\end{equation}
for all $v$ in $A_G^d(X)$ endows $A_G^*(X)$ with the structure of a
$\psi$-ring and, therefore, $A_G^*(X)_\nq$ with the structure of a
$\lam$-ring.  The \emph{augmentation} $\aug:A_G^*(X)\to A_G^0(X)$ is the
canonical projection.
\end{df}

Associated to any $\lambda$-ring there is another (pre-$\lam$-ring) structure usually denoted by $\ga$.
These are the \emph{Grothendieck $\gamma$-classes}  $\gamma_t : R\to R[[t]]$ given by the formula
\begin{equation}\label{eq:gamma}
\gamma_t := \sum_{i = 0}^\infty  \gamma^i t^i :=\lambda_{t/(1-t)}.
\end{equation}

\begin{thm}[cf. \cite{FuLa:85}]\label{thm:chpsigamma} 
If $Y$ is a connected algebraic space with a proper action of a linear algebraic group $G$, and if, for each non-negative integer $i$, 
$\Ch^i$ is the degree-$i$ part of the Chern character and $c^i$ is the $i$th Chern class, 
then the following 
equations
  hold for all integers $n\geq 1$, $i\geq 0$
 and for all $\cf$ in $K_G(Y)$: 
\begin{equation}\label{eq:ChPsi}
\Ch^i\circ\psi^n = n^i \Ch^i,
\end{equation}
\begin{equation}\label{eq:cfromch}
c_t (\cf) = \exp\left(\sum_{n\geq 1} (-1)^{n-1} (n-1)! \Ch^n(\cf) t^n \right),
\end{equation}
and 
\begin{equation}\label{eq:ChernGamma}
c^i(\cf) = \Ch^i(\gamma^i(\cf - \aug(\cf))).
\end{equation}
\end{thm}

\begin{rem}
Equation (\ref{eq:ChPsi})  is precisely the statement that the Chern
character $\Ch:K_G(X)_\nq\to A_G^*(X)_\nq$ is a homomorphism of $\psi$-rings
and
therefore $\lam$-rings.  
\end{rem}
In order to define inertial Chern classes and the inertial $\lam$-ring and $\psi$-ring structures, we will need the so-called \emph{Bott cannibalistic classes}.
 
\begin{df}
Let $Y$ be an algebraic space with a proper action of a linear algebraic group
$G$. Denote by $K^+_G(Y)$ the semigroup of 
classes of $G$-equivariant
vector bundles on $Y$.

For each $j\geq 1$, the \emph{$j$th Bott (cannibalistic) class
$\theta^j:K^+_G(Y)\to K_G(Y)$
} is the multiplicative class defined for any line bundle $\cl$ by 
\begin{equation}
\theta^j(\cl) = \frac{1-\cl^j}{1-\cl} = \sum_{i=0}^{j-1} \cl^i.
\end{equation}
By the splitting principle, we can extend the
definition of  $\theta^j(\cf)$ to all $\cf$ in $K^+_G(Y)$. 
\end{df}

\begin{df}
Let $\augi_{Y}$ denote the kernel of the augmentation $\aug:K_G(Y) \to K_G(Y)$.  
It is an ideal in the ring $(K_G(Y), \cdot)$, where $\cdot$ denotes the usual tensor product; and $\augi$ defines a topology on $K_G(Y)$.  We denote the completion 
of $K_G(Y)_\nq$
with respect to that topology by $\Kh_G(Y)_\nq$.
\end{df}
\begin{rem} \label{rem.bottclass} 
  We will need to define Bott classes on elements of integral rank in
  rational K-theory.  This can be done in a straightforward manner, but the
  resulting class will live in the augmentation completion of rational
  K-theory. Stated precisely, if $\cl$ is a line bundle, then we can expand
  the power sum for $\psi^{j}(\cl)$
as $\psi^{j}(\cl) = j(1 + a_1(\cl-1) + \ldots +
  a_{j-1}(\cl-1)^{j-1})$ for some rational numbers $a_1, \ldots ,
  a_{j-1}$. Since $(\cl -1)$ lies in the augmentation ideal, any
  fractional power of the expression $1 + a_1(\cl-1) + \ldots +
  a_{j-1}(\cl-1)^{j-1}$ can be expanded using the binomial formula as
  an element of $\Kh_G(Y)_{\nq}$. It follows that if $\alpha =
  \sum_{i}q_i\cl_i$ with $\sum_i q_i \in \nz$, then the binomial
  expansion of the expression $j^{\sum_i q_i} \prod_i (1 + a_1(\cl_i
  -1) + \ldots + a_{j-1}(\cl-1)^{j-1})^{q_i}$ defines $\theta^j(\alpha)$ as
  an element of $\Kh_G(Y)_\nq$.
\end{rem}

We will also need the following result.

\begin{thm}[The Adams-Riemann-Roch Theorem for Equivariant Regular Embeddings \cite{Koc:91,Koc:98}]\label{thm:ARR}
Let $\iota :Y\rInto X$ be a $G$-equivariant closed regular embedding of smooth manifolds. The following  commutes for all integers $n \geq 1$:
\begin{equation}\label{eq:ARR}
\begin{diagram}
K_G(Y) & \rTo^{\theta^n(N_\iota^*) \psi^n}& K_G(Y) \\
\dTo^{\iota_*} && \dTo^{\iota_*} \\
K_G(X) & \rTo^{\psi^n} & K_G(X),
\end{diagram}
\end{equation}
where $N_\iota^*$ is the conormal bundle of the embedding $\iota$.
\end{thm}

\section{Augmentation ideals and completions of inertial K-theory}
We will use the
Bott classes of $\cs$ to define inertial $\lam$- and
$\psi$-ring structures as well as
 inertial Chern classes.  Since $\cs$ is
generally not integral, we will often need to work in the augmentation
completion 
$\Kh_G(\IGX)_\nq$ of $K_G(\IGX)_{\nq}$.
However, it is not \emph{a priori} clear that the
inertial product behaves well with respect to this completion, since the
topology involved is constructed by taking classical powers of the classical
augmentation ideal instead of inertial powers of the inertial augmentation
ideal.  The surprising result of this section is that 
when $G$ is diagonalizable, 
these two completions
are the same. 

\begin{df}
  Given any inertial pair $(\R,\cs)$, define $\augi_{\cs}$ to be the
  kernel of the inertial augmentation $\ovaug: K_G(\IGX) \to
  K_G(\IGX)$. It is an ideal with respect to the inertial product
  $\star := \star_{\R}$.  Define 
$\augi_{I\ix}$
to be the kernel of the
  classical augmentation $\aug:K_G(\IGX) \to K_G(\IGX)$.  It is an
  ideal of $K_G(\IGX)$
with respect to  the 
usual tensor product  
instead of the inertial product.
\end{df}

Each of these two ideals induces a topology on $K_G(\IGX)$,
and we also consider a third topology induced by the augmentation
ideal $\augi_{\BG}$ of $R(G)$.  By \cite[Theorem 6.1a]{EdGr:00} the
$\augi_{\BG}$-adic and $\augi_{I\ix}$-adic topologies on
$K_G(\IGX)$ 
are the same.  In this section we will show that the
$\augi_{\cs}$-adic topology agrees with the other two.

\begin{lm}\label{lm:RG-alg}
If $(\R, \cs)$ is an inertial pair, then $(K_G(\IGX), \star_{\R})$ is an $R(G)$-algebra.  
Moreover,  for any $x \in R(G)$, if $\beta_{\Psi} \in K_G(I(\Psi))$, we have
$x \beta_{\Psi} = x \one \star_\R \beta_{\Psi}$. 
\end{lm}
\begin{proof}
By definition of an inertial pair, if $\alpha_{1} \in K_G(X)$ is supported in 
the untwisted sector, then $\alpha_{1} \star_{\R}
 \beta_{\Psi} = f_\Psi^* \alpha \cdot \beta$, where $f_\Psi \colon I(\Psi) \to X$ is the projection.
The lemma now follows from the projection formula for equivariant K-theory.
\end{proof}

\begin{thm} \label{thm.completions} 
When $G$ is diagonalizable
the $\augi_{\BG}$-adic,
  $\augi_{\IGX}$-adic, and $\augi_{\cs}$-adic topologies on
  $K_G(\IGX)$
are all equivalent. In particular, the
$\augi_{\BG}$-adic, the $\augi_{\IGX}$-adic, and the
$\augi_{\cs}$-adic completions of $K_G(\IGX)_\nq$ are equal.
\end{thm}
\begin{proof}
To prove that the topologies are equivalent we must show the following:
\begin{enumerate}
\item \label{cond.ClassInInert} For each positive integer $n$ there is a positive integer $r$, such that\\ $\augi^{\otimes r}_{\BG} K_G(\IGX) \subseteq 
(\augi_{\cs})^{\star n}$.

\item \label{cond.InertInClass} For each positive integer $n$ there is a positive integer $r$, such that\\ $(\augi_{\cs})^{\star r} \subseteq
\augi^{\otimes n}_{\BG}K_G(\IGX)_\nq$.
\end{enumerate}

Condition~(\ref{cond.ClassInInert}) follows Lemma~\ref{lm:RG-alg}  and the observation that $\augi_{\BG} K_G(\IGX)
\subset \augi_{\cs}$. In particular, we may take $r = n$.

Condition~(\ref{cond.InertInClass}) is more difficult to check. 
Given a $G$-space $Y$, we denote by $\augi_Y$ the subgroup of $K_G(Y)$
of elements of rank 0. This is an ideal with respect to the tensor product.

For each connected component $[U/G]$ of $[\IGX/G]$, the inertial augmentation satisfies $\res{\inCh_0(\alpha)}{U} = 0$ if $\age_{{\cs}}(U) >0$ and  $\res{\inCh_0(\alpha)}{U} = \res{\Ch_0(\alpha)}{U}$ if $\age_{{\cs}}(U) =0$ 
\cite[Thm 2.3.9]{EJK:12a}.
So $\augi_{\cs}$ has the following decomposition as an Abelian group 
$$\augi_{\cs} = \bigoplus_{\substack{U: \\ \age_{{\cs}}(U) = 0}} \augi_{U} 
\oplus \bigoplus_{\substack{U: \\ \age_{{\cs}}(U)>0}} K_G(U).$$

\begin{lm} \label{lem.mm-1}
If $m \in G$ with $\alpha \in K_G(X^m) \cap \augi_{{\cs}}$, 
and $\beta \in K_G(X^{m^{-1}}) \cap \augi_{{\cs}}$, then
$\alpha \star \beta \in \augi_{I\ix}$.
\end{lm}
\begin{proof}
Since $m m^{-1} = 1$, we have $\alpha \star \beta \in K_G(X^1) \subset K_G(\IGX)$,
so we must show $\alpha \star \beta \in \augi_X$.
If $\age_{{\cs}}(X^m) = 0$, then 
$\alpha_{m} \in \augi_{X^m}$, so the inertial product
$$\mu_*( e_1^*\alpha \cdot e_2^* \beta \cdot \euler(\R))$$
would automatically be in $\augi_{X}$ because the finite pushforward
$\mu_*$ preserves the classical augmentation ideal.
Thus we may assume that $\age_{\cs}(X^m)$ and $\age_{\cs}(X^{m^{-1}})$
are both nonzero and that $\alpha$ and $\beta$ have nonzero rank as elements of
$K_G(X^m)$ and $K_G(X^{m^{-1}})$, respectively. If the fixed locus
$X^{m,m^{-1}}$ has positive codimension, then 
$\mu_*\left(K_G(X^{m,m^{-1}})\right) \subset  K_G(X^1)$ is also in the classical
 augmentation ideal, since it consists of classes supported on subspaces of positive codimension. On the other hand, if $X^{m,m^{-1}} = X$, then $T_\mu|_{X^{m,m^{-1}}} = 0$.
By
definition of an inertial pair, ${\cs}|_{X^1} = 0$, so 
$\R|_{X^{m,m^{-1}}}  =  (e^*_1{\cs} + e_2^*{\cs})|_{X^{m,m^{-1}}}$ is a nonzero vector bundle.
It follows that $\euler(\R|_{X^{m,m^{-1}}}) \in \augi_{X^{m,m^{-1}}}$,
and once again
$\alpha \star \beta \in \augi_{X}$.
\end{proof}

Since $G$ is diagonalizable and acts with finite stabilizer on $X$,
 there is a finite Abelian subgroup $H \subset G$
such that $X^g = \emptyset$ for all $g \notin H$. 
Let $s = \sum_{h \in H} \left(\ord(h)-1\right)$.
\begin{lm}
The $(s+1)$-fold inertial product
$(\augi_{\cs})^{\star (s+1)}$ is contained in $\augi_{\IGX}$.
\end{lm}
\begin{proof}
By the definition of $s$, any list $m_1, \ldots , m_{s+1}$
of non-identity elements of $H$ contains at least one $h$
with multiplicity at least $\ord(h)$.
It follows that such a list 
contains subsets $m_1, \ldots , m_k$ and $m_{k+1}, \ldots , m_{l}$
with $m_1 \ldots m_k = (m_{k+1} \ldots m_l)^{-1}$.  

Since the inertial product is commutative, we may write any product of the form
$\alpha_{m_1} \star \cdots \star \alpha_{m_{s+1}}$ 
with $\alpha_{m_i} \in K_G(X^{m_i})$ 
as 
$\tilde{\alpha}_{m} \star \tilde{\beta}_{m^{-1}} \star \tilde{\gamma}_{m'} $ for some
$\tilde{\alpha}_m \in K_G(X^m)$, $\tilde{\beta}_m \in K_G(X^{m^{-1}})$, and $
\tilde{\gamma}_{m'} \in K_G(X^{m'})$. Lemma~\ref{lem.mm-1} now gives the result.
\end{proof}

To complete the proof of Theorem~\ref{thm.completions}, observe first that we may use the equivalence of the $\augi_{\BG}$-adic and the $\augi_{\Itwo_\ix}$-adic topologies in the ring $(K_G(\Itwo_G X), \otimes)$ to see that for any $n$ there is an $r$ such that 
$\augi_{\Itwo_GX}^{\otimes r} \subset \augi_{\BG}^{\otimes n} K_G(\Itwo_GX)$. This implies that $\mu_*(\augi_{\Itwo_GX}^{\otimes r}) \subset \augi_{\BG}^{\otimes n} K_G(\IGX)$. It follows that $\augi_{\cs}^{\star(r(s+1))} \subset \augi_{\BG}^{\otimes n} K_G(\IGX)$.
\end{proof}

Since the three topologies are the same
we will not distinguish between them from now on, and will use the term 
\emph{augmentation completion} to denote the completion with respect any one of these augmentation ideals. 
The completion of $K_G(\IGX)_\nq$ will be denoted by $\Kh_G(\IGX)_\nq$. Note
that this completion is a summand in 
$K_G(\IGX)_\nq$ \cite[Proposition 3.6]{EdGr:05}.

\section{Inertial Chern classes and power operations}\label{sec.inertialpsi}

In this section we show that for each 
Gorenstein inertial pair
$({\R}, {\cs})$ and corresponding Chern Character
$\ovCh$, we can define inertial Chern classes. 
When $({\R},{\cs})$ is strongly Gorenstein, there 
are also $\psi$-operations, $\lam$-operations, and $\gamma$-operations 
on the corresponding inertial K-theory 
$K_G(\IGX)$. These operations 
behave nicely with respect to the inertial Chern
character and satisfy many relations, including an analog of
Theorem~\ref{thm:chpsigamma}.
When $G$ is diagonalizable these operations make the inertial 
K-theory ring  $K_G(\IGX)$ into a $\psi$-ring and 
$K_G(\IGX)\otimes\nq$ into 
a 
$\lambda$-ring.

\subsection{Inertial Adams (power) operations and inertial Chern classes}
\label{sec:oAdams}  

We begin by defining inertial Chern classes.
We then define inertial  Adams operations associated to a 
strongly Gorenstein pair $({\R}, {\cs})$ and show that, for a diagonalizable group $G$, the corresponding rings are
$\psi$-rings with many other nice properties.

\begin{df}\label{df.ochernclass}
For any Gorenstein inertial pair $({\R}, {\cs})$ 
the \emph{${\cs}$-inertial Chern series} $\ocv_t:K_G(\IGX)\to
A^*_G(\IGX)_\nq[[t]]$ is defined, for all $\cf$ in $K_G(\IGX)$, by
\begin{equation}\label{eq:OChernSeries} 
  \ocv_t(\cf) = \ovexp\left( \sum_{ n\geq 1} (-1)^{n-1} (n-1)! \ovCh^n(\cf) t^n \right),
\end{equation}
where the power series $\ovexp$ is defined with respect to the
$\star_{\R}$ product, and $\ovCh^n(\cf)$ is the component of
$\ovCh(\cf)$ in $A^*(\IGX)$ with $\cs$-age equal to $n$. For all $i\geq 0$,
the \emph{$i$th ${\cs}$-inertial Chern class $\ovc^i(\cf)$ of $\cf$} is
the coefficient of $t^i$ 
in $\ocv_t(\cf)$.
\end{df}

\begin{rem}
The definition of inertial Chern classes could be extended to the
non-Gorenstein case by introducing fractionally graded $\cs$-inertial Chern
classes, but the latter does not behave nicely with respect to the inertial
$\psi$-structures.
\end{rem}

\begin{df}\label{df.opsi}
Let $({\R}, {\cs})$ be a strongly Gorenstein inertial pair.
For all integers $j\geq 1$, we define the 
\emph{$j$th inertial Adams (or power) 
 operation}
$\opsi^j:K_G(\IGX)\to     K_G(\IGX)$ by the formula
\begin{equation}\label{eq:def-psi-tilde}
\opsi^j(\cf) := 
\psi^j(\cf)\cdot\theta^j({\cs}^*)
\end{equation}
for all $\cf$ in $K_{G}(\IGX)$. (Here $\cdot$ is the ordinary tensor product on
$K_G(\IGX)$.)
\end{df}
We show in Theorem~\ref{thm:OrbifoldPsiRing} that, in many cases, these inertial Adams operations define a $\psi$-ring structure on $(K_G(\IGX), \star_{\R})$.

\begin{rem} \label{rem.psiGor} If $({\R}, {\cs})$ is Gorenstein, then $\cs$ has 
  integral rank and
 $\theta^j({\cs}^*)$ may be
defined as an element of the completion $\Kh_G(\IGX)_\nq$ (see Remark \ref{rem.bottclass}).
Thus we can still define inertial Adams
  operations as maps $\opsi^j \colon K_G(\IGX) \to
  \Kh_G(\IGX)_{\nq}$.
\end{rem} 

\begin{df}\label{df.olam}
 Let $({\R}, {\cs})$ be a strongly Gorenstein inertial pair. 
We define   
$\olam_t:K_G(\IGX) \to K_G(\IGX)_\nq[[t]]$ 
by Equation
  (\ref{eq:LambdaFromPsi}) after replacing 
$\psi$, $\lam$, and $\exp$ by their 
             respective inertial analogs $\opsi$, $\olam$, and $\oexp$:
\begin{equation}\label{eq:oLambdaFromoPsi}
\olam_t = \oexp\left(\sum_{r\geq 1} (-1)^{r-1} \opsi^r \frac{t^r}{r}\right).
\end{equation} 
Define $\olam^i$ to be the coefficient of $t^i$ in $\olam_t$.
We call $\olam^i$ the \emph{$i$th inertial $\lam$ operation}. 
\end{df}

We now prove a relation between  
inertial  Chern classes,
the inertial  Chern character, and  
inertial  Adams operations, but first we need two lemmas connecting the classical Chern
character, Adams operations, Bott classes, and Todd classes.

\begin{lm}
Let $\cf \in K_G(\IGX)$
be the class of a $G$-equivariant vector bundle on $\IGX$.
For all integers $n\geq 1$, we have the equality
in $A^*_G(\IGX)$:
\begin{equation}\label{eq:ChThetaToddIdentity}
\Ch(\theta^n(\cf^*)) \Td(-\cf) = n^{\Ch^0(\cf)} \Td(-\psi^n(\cf)).
\end{equation}
More generally, if $\cf \in K_G(\IGX)_\nq$ is such that $\cf = \sum_{i=1}^k
\alpha_i \cv_i$, where $\cv_i$ is a vector bundle, $\alpha_i\in\nq$ with
$\alpha_i > 0$ for all $i=1,\ldots,k$, and 
$\Ch^0(\cf) \in \nz^\ell \subset A_G^0(\IGX)_\nq$, 
($\ell$ is the number of connected components of $[\IGX/G]$),
then Equation~(\ref{eq:ChThetaToddIdentity}) still holds in 
$A_G^*(\IGX)_\nq$, where $\theta^n(\cf^*)$ is an element in the
completion $\Kh_G(\IGX)_\nq$. 
\end{lm}
\begin{proof}
Let $\cl$ in $K_G(\IGX)$ be a line bundle with 
ordinary
first Chern class $c := c^1(\cl)$.
For all $n\geq 1$ we have
\begin{eqnarray*}
\Ch(\theta^n(\cl^*)) \Td(-\cl) 
&=&  \Ch\left(\frac{1-(\cl^*)^n}{1-\cl^*}\right) \left(\Td(\cl)\right)^{-1} 
=  \left(\frac{1-e^{-n c}}{1-e^{-c}}\right)
\left(\frac{c}{1-e^{-c}}\right)^{-1} \\ 
&=&  n \left(\frac{n c}{1-e^{-n c}} \right)^{-1} 
=  n \Td(\cl^n)^{-1}, \\ 
\end{eqnarray*}
and we conclude that
$\Ch(\theta^n(\cl^*)) \Td(-\cl) = n \Td(-\psi^n(\cl))$.
Equation (\ref{eq:ChThetaToddIdentity}) now follows from the splitting principle, the
multiplicativity of $\theta^n$ and $\Td$, and the fact that $\Ch$ is a ring
homomorphism.

The more general statement follows the fact that $\Ch$ and $\Td$ factor
through $\Kh_G(\IGX)_\nq$, 
together with the fact that 
$\Ch^0(\theta^j(\cf)  - j^{\aug(\cf)}) = 0$.
\end{proof}

This lemma yields the following useful theorem.
\begin{thm}
Let $({\R}, {\cs})$ be a strongly 
Gorenstein inertial pair. For any
$\alpha \in {\mathbb N}$ 
and integer $n\geq 1$, we have 
\begin{equation}\label{eq:orbChernCharAndPsi}
\ovCh^{\alpha}(\opsi^n(\cf)) = n^\alpha \ovCh^\alpha(\cf)
\end{equation}
in 
$A^\bra{\alpha}_G(\IGX)_\nq$,
where the grading is the ${\cs}$-age grading. 
\end{thm}

\begin{proof}
\begin{eqnarray*}
\ovCh(\opsi^n(\cf)) 
&=& \Ch(\psi^n(\cf)\theta^n({\cs^*}))\Td(-{\cs}) 
= \Ch(\psi^n(\cf))\Ch(\theta^n({\cs}^*))\Td(-{\cs}) \\
&=& \Ch(\psi^n(\cf))\Td(-\psi^n({\cs})) n^{\age} 
= \sum_{\alpha \in {\mathbb N}} n^{\alpha}  \ovCh^\alpha(\cf),
\end{eqnarray*}
where the third equality follows from Equation
(\ref{eq:ChThetaToddIdentity}),
and the final equality follows from 
the definition of $\ovCh^\alpha$, from Equation (\ref{eq:ChPsi}), and
from the fact that for all $j\geq 0$ and $k\geq 1$  
\begin{equation}\label{eq:ToddPsi}
\Td^j \circ\psi^k = k^j \Td^j,
\end{equation}
where $\Td = \sum_{j\geq 0} \Td^j$ such that $\Td^j$ belongs to $A_G^j(\IGX)_\nq$.
Equation (\ref{eq:ToddPsi}) is proved in the same fashion as Equation (\ref{eq:ChPsi}).
\end{proof}

\begin{rem}
If $({\R}, {\cs})$ is
 a Gorenstein inertial pair, then
Equation~(\ref{eq:orbChernCharAndPsi}) 
also holds in 
$A^\bra{\alpha}_G(\IGX)_\nq$,
where $\opsi^n$ is interpreted as a map $\opsi^n:K_G(\IGX)\to \Kh_G(\IGX)_\nq$
(cf. Remark \ref{rem.psiGor}). This follows because 
$\ovCh$ factors
through the completion $\Kh_G(\IGX)_\nq$.
\end{rem}

\begin{df}
Let $({\R}, {\cs})$ be a strongly 
Gorenstein inertial pair.
We define the 
inertial  operations $\oga_t$ on inertial K-theory as in
Equation~(\ref{eq:gamma}), that is 
\begin{equation}\label{eq:ogamma}
\oga_t := \sum_{i = 0}^\infty \oga^i t^i := \olam_{t/(1-t)}.
\end{equation}
\end{df}
\begin{rem}
If $({\R}, {\cs})$ is only Gorenstein, then we may still define
$\gamma_t$ as a map $K_G(\IGX) \to \Kh_G(\IGX)_\nq[[t]]$.
\end{rem}

\begin{thm}\label{thm:oChClassProps}
Let $({\R}, {\cs})$ be a Gorenstein inertial pair. 
The $\cs$-inertial Chern series $\ocv_{t}:K_G(\IGX)\to A_G^*(\IGX)_{\nq}[[t]]$ 
satisfies the following properties:
\begin{description}[labelindent=2\parindent]
\item[Consistency with $\oga$] For all integers $n\geq 1$ and for all $\cf$
  in $K_G(\IGX)_\nq$, we have the following equality in $A^*_G(\IGX)_\nq$:
\begin{equation}\label{eq:oChernGammaConsistency}
\ocv^{n}(\cf) = \ovCh^n(\oga^n(\cf - \ovaug(\cf))),
\end{equation} 
where $\oga_t$ is interpreted as a map $K_G(\IGX)_\nq\to \Kh_G(\IGX)_\nq[[t]]$.
\item[Multiplicativity] For all $\cv$ and $\cw$ in $K_G(\IGX)_\nq$, 
\[ \ocv_{t}(\cv + \cw) = \ocv_{t}(\cv) \star_{\R} \ocv_{t}(\cw). \]
\item[Zeroth Chern class] For all $\cv$ in $K_G(\IGX)_\nq$, 
we have $\ocv^{0}(\cv) = \one$.
\item[Untwisted sector] For all $\cf \in K_G(X^1)\subseteq K_G(\IGX)$ (i.e., supported only on the untwisted sector), the inertial Chern classes
  agree with the ordinary Chern classes, i.e.,
$  \ocv_{t}(\cf) = c_t(\cf)$.
 \item[Classes of Unity] 
All the 
inertial  Chern classes of unity vanish, except for $\ocv^{0}(\one)$, so we have $\ocv_{t}(\one) =   \one$.
\end{description}
\end{thm}

\begin{rem}
The 
theorem shows 
that Equation~(\ref{eq:oChernGammaConsistency}) yields an alternative, but
equivalent, definition of inertial Chern classes.
\end{rem}

\begin{proof}
Multiplicativity and $\ocv^{0}(\cv) = \one$ follow immediately from the
exponential form of Equation (\ref{eq:OChernSeries}) and the fact that
$\ovCh$ is a homomorphism. 

On the untwisted sector, inertial products
reduce to the ordinary products, and the inertial Chern character reduces to the classical Chern character,   
and this shows that Equation (\ref{eq:OChernSeries}) 
agrees with Equation~\eqref{eq:cfromch}, which implies the untwisted sector agrees with ordinary Chern classes.
The Classes of Unity condition will follow immediately from Equation~(\ref{eq:oChernGammaConsistency}).

The hard part of this proof is the consistency of the 
inertial  Chern classes
with $\oga$ (Equation~(\ref{eq:oChernGammaConsistency})).  To prove this, it
will be useful to first introduce the ring homomorphism $\ovCh_t:K_G(\IGX)\to
A^*_G(\IGX)_\nq[[t]]$ via 
$
\ovCh_t(\cf) := \sum_{n\geq 0} \ovCh^n(\cf) t^n.
$
For the remainder of the proof, all products are understood to be
inertial products. We have the following equality in 
$A^*_G(\IGX)_\nq[[t]]$, 
\begin{eqnarray*}
\ovCh_t(\olam_u(\cf))) &=&\ovexp\left( \sum_{k\geq 1} \frac{(-1)^{k-1}}{k} \ovCh_t(\opsi^k(\cf)) u^k \right) \\
&=& \ovexp\left( \sum_{k\geq 1} \frac{(-1)^{k-1}}{k} \ovCh_{k t}(\cf) u^k \right) \\
&=&\ovexp\left(\sum_{k \geq 1} \frac{(-1)^{k-1}}{k}
\sum_{\alpha\geq 0}\ovCh^\alpha(\cf) (k t)^\alpha u^k \right)
\\ 
&=&\ovexp\left(\sum_{\alpha\geq 0} 
\ovCh^\alpha(\cf) t^\alpha \sum_{k\geq 1} (-1)^{k-1} k^{\alpha-1} u^k\right),  
\end{eqnarray*}
where the first equality follows from the definition of $\olam$ and the fact that $\ovCh_t$ is a ring homomorphism, and the second equality follows from Equation (\ref{eq:orbChernCharAndPsi}).
From the definition of $\oga_t$, it follows that
\begin{eqnarray*}
\ovCh_t(\oga_u(\cf - \ovaug(\cf)))
& =& \ovexp\left(\sum_{\alpha\geq 0} \ovCh^\alpha(\cf - \ovaug(\cf)) t^\alpha 
\sum_{k\geq 1} (-1)^{k-1} k^{\alpha-1} \left(\frac{u}{1-u}\right)^k\right)\\
& =& \ovexp\left(\sum_{\alpha\geq 0}
\sum_{k\geq 1} (-1)^{k-1} k^{\alpha-1} \ovCh^\alpha(\cf) t^\alpha
\sum_{n\ge k} u^n{{n-1}\choose{k-1}}\right)\\
& =& \ovexp\left(\sum_{\alpha\geq 0}
\ovCh^\alpha(\cf) t^\alpha
\sum_{n\geq 1}u^n \sum_{k=1}^{n} (-1)^{k-1} k^{\alpha-1}   {{n-1}\choose{k-1}}\right)\\
& =& \ovexp\left(\sum_{\alpha\geq 0}
\ovCh^\alpha(\cf) t^\alpha
\sum_{n\geq 1}u^n (-1)^{n-1}(n-1)! S(\alpha,n)\right),
\end{eqnarray*}
where 
\[
S(\alpha,n) = \frac{1}{n!}\sum_{j=0}^n (-1)^{n-j} {n \choose j} j^\alpha
\]
are the Stirling numbers of the second kind. Projecting out those terms which
are 
not 
powers of 
$z := ut$ 
yields the equality
\[
\sum_{\ell\geq 0} \ovCh^\ell(\ovga^\ell(\cf-\ovaug(\cf))) z^\ell =
\ovexp\left(\sum_{s\geq 0} z^s \ovCh^s(\cf) (-1)^{n-1} (n-1)! S(n,n)\right).
\]
The identity $S(n,n) = 1$ and Equation~(\ref{eq:OChernSeries}) yield
Equation~(\ref{eq:oChernGammaConsistency}).
\end{proof}

Even when an inertial pair $({\R},\cs)$ is not Gorenstein, there are natural subrings of $K_G(\IGX)$ and $A^*_G(\IGX)$ where things behave well (as if $({\R},\cs)$ were Gorenstein). 

\begin{df} 
Let $({\R},\cs)$ be an inertial pair, and let $\ell$ be the number of connected components of $I\ix = [\IGX/G]$. 
The subring of
  $K_G(\IGX)$  consisting of all elements of 
  $\cs$-grading $0 \in (\nq/\nz)^\ell$
  is called the \emph{Gorenstein subring $\KGor_G(\IGX)$} of  $K_G(\IGX)$, and the 
  subring of $A_G^*(\IGX)$ consisting of all elements of 
$\cs$-degree in $\nz^\ell \subseteq \nq^\ell$
is called the  \emph{Gorenstein subring $\AGor_G(\IGX)$} of $A_G^*(\IGX)$. 
\end{df}

\begin{rem}
The previous theorem holds for a general inertial pair of a $G$-space $X$, 
provided that $K_G(\IGX)$ and $A_G^*(\IGX)$ are replaced by their Gorenstein subrings
$\KGor_G(\IGX)$ and $\AGor_G^*(\IGX)$, respectively.
\end{rem}

\subsection{$\psi$-ring and $\lambda$-ring structures on inertial K-theory}
The main result of this section is the following:
\begin{thm}\label{thm:OrbifoldPsiRing}
If $G$ is a diagonalizable group and 
$({\R}, {\cs})$
is a strongly Gorenstein inertial pair on $\IGX$, then 
$(K_G(\IGX),\star_{\R},\one,\ovaug,\opsi)$ is an augmented $\psi$-ring.

Moreover, for general (possibly non-diagonalizable) $G$
and any inertial pair $(\R,\cs)$, 
the augmentation completion of
the Gorenstein subring $\KGor_G(\IGX)_\nq$ of $K_G(\IGX)_\nq$ is an augmented $\psi$-ring.
\end{thm} 

\begin{rem}
The hypothesis that $G$ is diagonalizable is necessary, as is demonstrated later in this section (see Example~\ref{ex:BS3}).
\end{rem}

With a little work we get the following corollary.

\begin{crl}\label{crl.OrbLambdaRing}
Let $({\R}, {\cs})$ be 
a strongly Gorenstein
inertial pair with $G$ diagonalizable.
Then 
$(K_G(\IGX)_\nq,\star_{\R},\one,\olam)$ 
is 
an augmented $\lam$-algebra over $\nq$.

Moreover, for general (possibly non-diagonalizable) $G$
and any inertial pair $(\R,\cs)$, 
the augmentation completion of the Gorenstein subring $\KGor_G(\IGX)_\nq$ of
$K_G(\IGX)_\nq$  is an augmented 
$\lam$-algebra over $\nq$.
\end{crl}
\begin{proof}[Proof of Corollary \ref{crl.OrbLambdaRing}]
Combining Theorem~\ref{thm:OrbifoldPsiRing} with Theorem~\ref{thm:LambdaFromPsi}, all 
that we must prove is that 
\begin{equation}\label{eq:OvaugOlam}
\ovaug(\olam_t(\cf)) = \olam_t(\ovaug(\cf)) = (1+t)^{\ovaug(\cf)}.
\end{equation}
Here we have omitted the $\pd$ from the notation, but all products are the
inertial product $\pd$, and exponentiation is also with respect to the
product $\pd$.  

For all $\cf \in K_G(\IGX)$, we have 
\[
\ovaug(\olam_t(\cf)) = \sum_{i\geq 0} t^i \ovaug(\olam^i(\cf)),
\]
but 
\begin{eqnarray*}
\ovaug(\olam_t(\cf)) 
&=& \ovaug(\oexp(\sum_{n\geq 1} \frac{(-1)^{n-1}}{n} t^n \opsi^n(\cf))) 
= \oexp(\sum_{n\geq 1} \frac{(-1)^{n-1}}{n} t^n \ovaug(\opsi^n(\cf))) \\
&=& \oexp(\sum_{n\geq 1} \frac{(-1)^{n-1}}{n} t^n \ovaug(\cf)) 
= (1+t)^{\ovaug(\cf)},
\end{eqnarray*}
where the third line follows from  
$\ovaug\circ\opsi^n = \ovaug$ (by Theorem~\ref{thm:OrbifoldPsiRing}). Finally, we have that $\olam_t(\ovaug(\cf))=(1+t)^{\ovaug(\cf)}$, since $\ovaug$ commutes
with $\opsi$ by Theorem~\ref{thm:OrbifoldPsiRing}.
\end{proof}

\begin{proof}[Proof of Theorem \ref{thm:OrbifoldPsiRing}]
First, it is straightforward from the definition that $\opsi^n(\cf+\cg) = \opsi^n(\cf)+\opsi^n(\cg)$, and also $\opsi^1(\cf) = \cf$, since $\theta^1(\cg) = \one$ for any $\cg$.   
Second, we have $\opsi^n(\one) = \one$, since $\one$ is supported only on 
$K_G(X^1)$, and ${\cs}_{X^1} = 0$ (because $({\R}, {\cs})$
is an inertial pair).
Now, to show for all $\cf$ in $K_G(\IGX)$ that
\[
\opsi^n(\opsi^\ell(\cf)) = \opsi^{n\ell}(\cf),
\]
we observe that
\[
\opsi^n(\opsi^\ell(\cf)) 
= \opsi^n(\psi^\ell(\cf)\theta^\ell({\cs}^*)) 
= \psi^{n\ell}(\cf)\psi^n(\theta^\ell({\cs}^*)\theta^n({\cs}^*)).
\]
Hence, we need to show that
\[
\psi^n(\theta^\ell({\cs}^*))\theta^n({\cs}^*) = \theta^{n\ell}({\cs}^*).
\]
This follows from the splitting principle in ordinary K-theory, from the fact that the Bott classes
are multiplicative, and from the fact that for any line bundle $\cl$ we have
\begin{equation}
\psi^n(\theta^\ell(\cl))\theta^n(\cl) 
= \psi^n\left(\frac{1-\cl^\ell}{1-\cl}\right) \frac{1-\cl^n}{1-\cl} 
= \frac{1-\cl^{n \ell}}{1-\cl^n} \frac{1-\cl^n}{1-\cl} 
= \theta^{n\ell}(\cl).
\end{equation}

It remains to show that $\opsi$ preserves the inertial product defined
by ${\R}$, i.e.,
\begin{equation}\label{eq:OrbifoldPsiHomo}
\opsi^n(\cf
\star \cg) = \opsi^n(\cf)\star\opsi^n(\cg),
\end{equation}
where $\star$ is understood to refer to the $\star_{\R}$-product. It is at this point in the proof that we need to use the hypothesis that $G$ is diagonalizable.
\begin{lm} \label{lm.IIembed}
Let $G$ be a diagonalizable group. For each $(m_1,m_2) \in G \times G$
 let $X^{m_1,m_2} = \{(m_1,m_2,x)|m_1x= m_2 x=x\} \subset \II_G X$.
Then $X^{m_1,m_2}$ is open and closed (but possibly empty)
and the restriction of $\mu$ to $X^{m_1,m_2}$ is a regular embedding.
\end{lm}
\begin{proof}
There is a decomposition of $\II_G X$ into closed and open components indexed by conjugacy classes of pairs in $G \times G$. However, since
$G$ is diagonalizable, each conjugacy class consists of a single pair. 
If $\Psi = \{(m_1,m_2)\}$, then $\II(\Psi) = X^{m_1,m_2}$ 
and the multiplication map
restricts to the closed embedding $\mu \colon X^{m_1,m_2} \to X^{m_1m_2}$,
where $X^{m_1m_2} = \{(m_1m_2,x)| m_1m_2x =x\} \subset I_GX$. Since
$X$ is smooth, the fixed loci $X^{m_1,m_2}$ and $X^{m_1m_2}$ are also smooth, so the map is a regular embedding.
\end{proof}

Let us prove that $\opsi$ is compatible with the inertial product.
\begin{equation}
\begin{array}{ccl} \label{eq.prod1stpart}
\opsi^n(\cv\star\cw) 
&=& \theta^{n}({\cs}^*) \cdot \psi^n(\cv\star \cw) \\
&=& \theta^{n}({\cs}^*) \cdot  \psi^n\left(\mu_*( e_1^*\cv 
\cdot e_2^*\cw  \cdot \lam_{-1}({\R}^*))\right) 
\end{array}
\end{equation}
By our lemma  $\II_GX$ decomposes as a disjoint sum  $\coprod X^{m_1,m_2}$
with $\mu|_{X^{m_1,m_2}}$ a closed regular embedding.
Since an element $\alpha \in K_G(\II_GX)$ decomposes as sum
$\alpha = \sum_{(m_1,m_2)}\alpha_{m_1,m_2}$ with $\alpha_{m_1,m_2} \in 
K_G(X^{m_1,m_2})$,
we may invoke the equivariant
Adams-Riemann-Roch for closed embeddings (Theorem \ref{thm:ARR}) on each 
of $\alpha_{m_1,m_2}$ to conclude that 
$\psi^n\mu_*\alpha = \mu_*(\theta^n(N^*_\mu) \psi^n\alpha)$,
where $N^*_\mu$ is the conormal bundle of $\mu$. Writing
$N^*_\mu = -T^*_\mu$ (see Definition~\ref{df.inertialpair}) we obtain the equalities
\begin{equation}
\begin{array}{ccl} \label{eq.prod2stpart}
&=& \theta^{n}({\cs}^*) \cdot  \mu_*\left[\theta^n(-T^*_\mu) 
\cdot \psi^n\left( e_1^*\cv \cdot e_2^*\cw
  \cdot \lam_{-1}({\R})^*\right) \right] \\
&=& \theta^{n}({\cs}^*) \cdot  \mu_*\left[\theta^n(-T^*_\mu)  
\cdot e_1^*\psi^n(\cv)
  \cdot e_2^*\psi^n(\cw) \cdot  \psi^n(\lam_{-1}({\R}^*)) \right] \\
&=& \theta^{n}({\cs}^*) \cdot  \mu_*\left
[\theta^n(-T^*_\mu)  \cdot e_1^*\psi^n(\cv)
  \cdot e_2^*\psi^n(\cw)  \cdot \lam_{-1}(\psi^n({\R}^*)) 
\right] \\
&=& \theta^{n}({\cs}^*) \cdot  \mu_*
\left[\theta^n(-T^*_\mu)  \cdot e_1^*\psi^n(\cv)
  \cdot e_2^*\psi^n(\cw)  \cdot \lam_{-1}({\R}^*)  
\cdot \theta^n({\R}^*)  \right]  \\
&=& \theta^{n}({\cs}^*) \cdot  \mu_*\left[ e_1^*\psi^n(\cv)
  \cdot e_2^*\psi^n(\cw)  \cdot \lam_{-1}({\R}^*)  \cdot 
\theta^n({\R}^*-T^*_\mu) \right], 
\end{array}
\end{equation}
where the second equality follows from the fact that $\psi^n$ respects
the ordinary ($\cdot$) multiplication, the third from the definition
of the Euler class and the fact that \cite[p. 48]{Knu:73} for all $i,n$,
\[
\psi^n\circ \lam^i = \lam^i\circ \psi^n,
\]
the fourth from the fact that for any nonnegative element $\cf$ in 
$K_G(\IGX)$ we have 
\[
\theta^n(\cf){\lam_{-1}(\cf)} = {\lam_{-1}(\psi^n(\cf))},
\]
and the fifth from the multiplicativity of $\theta^n$.
Since $\opsi^n(\cf) = \psi^n(\cf)\theta^n({\cs}^*)$,
we may express the last line of 
\eqref{eq.prod2stpart} 
as 
\begin{equation}\label{eq.psidef}
\theta^n({\cs}^*)\mu_*\left[ e_1^*\opsi^n(\cv)
  \cdot e_2^*\opsi^n(\cw)  \cdot \lam_{-1}({\R}^*) 
\cdot \theta^n({\R}^*-T^*_\mu-e_1^*{\cs}^* -e_2^*{\cs}^*)
\right]. 
\end{equation}
Applying the projection formula to \eqref{eq.psidef} yields
\begin{equation*} \label{eq.almostthere}
\opsi^n(\cv \star \cw) = \mu_*\left[ e_1^*\opsi^n(\cv) \cdot e_2^*\opsi^n(\cw) \cdot \lam_{-1}({\R}^*)
\cdot \theta^n({\R}^* - T^*_\mu -e_1^*{\cs}^* -
e_2^*{\cs}^* +\mu^*{\cs}^*)\right].
\end{equation*}
Now because $({\R}, {\cs})$ is an inertial pair, we have
$${\mathscr R} = e_1^* {\cs} + e_2^*{\cs} -\mu^*{\cs} 
+ T\mu,$$
so 
\[
\opsi^n(\cv \star \cw)  = \mu_*\left[e_1^*\opsi^n(\cv) \cdot e_2^*\opsi^n(\cw)
\cdot \lambda_{-1}({\R}^*) \right]
 =   \opsi(\cv) \star \opsi(\cw),
\]
as claimed.

Finally, from the definition of $\opsi$ and the fact that the ordinary
augmentation in ordinary equivariant K-theory is preserved 
by 
and commutes with
the ordinary $\psi$-operations, we have 
\begin{equation}\label{eq:OvaugOpsi}
\ovaug(\opsi^n(\cv)) = \opsi^n(\ovaug(\cv)) = \ovaug(\cv).
\end{equation}

When $G$ is not diagonalizable, $\mu$ is a finite l.c.i.~morphism, 
but in general it does not restrict to a closed embedding on each
component of $\II_GX$. In this case the
equivariant Adams-Riemann-Roch theorem holds \cite[Theorem
4.5]{Koc:98} after completing $K_G(I_GX)_{\nc}$ and
$K_G(\II_GX)_{\nc}$ at the augmentation ideal. Restricting to the
augmentation completion of the Gorenstein subring insures that the
Bott class $\theta^k(\cs^*)$ takes values in that subring (which has
$\nq$ coefficients), whereas the Bott class in general would take
values in the augmentation completion of
$K_G(\IGX)\otimes\overline{\nq}$. The rest of the above argument goes
through verbatim.
\end{proof}
\begin{rem} 
  Suppose $G$ is not Abelian, but the fixed locus $X^g$ is empty if
  $g$ is not in the center of $G$. Then, since the conjugacy classes of
  central elements are singletons, the argument of Lemma
  \ref{lm.IIembed} shows that $\II_GX$ is a disjoint sum of components such that
  the restriction of $\mu$ to each of them is a regular
  embedding. Arguing as in the proof of Theorem
  \ref{thm:OrbifoldPsiRing} shows that in this case the inertial
  product would also commute with the inertial Adams
  operations.
\end{rem}
\begin{rem}
  If $G$ is finite, then for each conjugacy class $\Phi \subset
  G\times G$ and $\Psi \subset G$ such that $\mu( \II(\Phi)) \subset
  I(\Psi)$, the pushforward map $\mu_* \colon K_G(\II(\Phi)) \to
  K_G(\I(\Psi))$ can be identified as a combination of pushforward
  along a regular embedding with an induction functor.  Precisely, if
  $(m_1, m_2) \in \Phi$ is any element, then $K_G(\II(\Phi))$ can be
  identified with $K_{Z_{1,2}}(X^{m_1,m_2})$, where $Z_{1,2}$ is the
  centralizer of $m_1$ and $m_2$ in $G$. Likewise $K_G(I(\Psi))$ can be
  identified with $K_{Z_{12}}(X^{m_1m_2})$, where $Z_{12}$ is the
  centralizer of the element $m_1m_2$.  Let $i \colon X^{m_1,m_2}
  \hookrightarrow X^{m_1m_2}$ be the inclusion.  Via these
  identifications the pushforward $\mu_*$ is the composition of the
  pushforward $i_* \colon K_{Z_{1,2}}(X^{m_1,m_2}) \to
  K_{Z_{1,2}}(X^{m_1m_2})$, with the induction functor
  $\Ind_{Z_{1,2}}^{Z_{12}} :K_{Z_{12}}(X^{m_1m_2}) \to
  K_{Z_{1,2}}(X^{m_1m_2})$.  
In this case, determining whether the equality $\opsi^j(\alpha
  \star \beta) = \opsi^j(\alpha) \star \opsi^j (\beta)$ holds in $K_G(I_GX)_\nq$
  boils down to the question of whether the classical Adams operations $\psi^j$
 commute with
  induction. This question has been studied in Section 6 of
  \cite{Koc:98},
where it is proved that Adams operations commute with induction after completion at the augmentation ideal.
\end{rem}

\begin{rem} 
Let $(\R,\cs)$ be a Gorenstein inertial pair on $\IGX$. For
each integer $k\geq 1$, let $\psit^k:A^*_G(\IGX)\to
  A^*_G(\IGX)$ be defined by Equation (\ref{eq:ChowPsi}).
If  $\ovaug:A^*_ G(\IGX)\to
  A^{\{0\}}_G(\IGX)$ is the canonical projection, then the inertial
  Chow theory $(A^*_G(\IGX),\pd,1,\psit, \ovaug)$ is an augmented
  $\psi$-ring. 

Moreover, if $G$ is a diagonalizable group and $(\R,\cs)$ is a strongly
Gorenstein inertial pair on $\IGX$, then the
summand $\Kh_G(\IGX)_{\nq}$ inherits an augmented $\psi$-ring
  structure from $K_G(\IGX)_\nq$. In addition, 
  Equation~(\ref{eq:orbChernCharAndPsi}) means that the inertial Chern
  character
homomorphism
$\ovCh:K_G(\IGX)_\nq\to A_G^*(\IGX)_\nq$ 
preserves the augmented
  $\psi$-ring structures and factors through an isomorphism of
 augmented $\psi$-rings $\Kh_G(\IGX)_\nq \to A_G^*(\IGX)_{\nq}$.
  In particular, if $G$ acts freely on $X$, then the inertial Chern
  character is an isomorphism of augmented $\psi$-rings.
\end{rem}

\begin{ex}\label{ex:BS3} 
  The hypothesis of Theorem~\ref{thm:OrbifoldPsiRing} 
  that $G$ is diagonalizable is necessary, as demonstrated by the following example. 
  
  Let $G= S_3$ be the symmetric group $S_3$ on three letters, and
  consider the classifying stack $\B S_3 = [\text{pt}/S_3].$ The
  inertia stack $IBS_3$ is the disjoint union of three components,
  corresponding to the conjugacy classes of $(1)$, $(12)$, and
  $(123)$ in $S_3$.  The component corresponding to class
  $\Psi$ is the stack $[\Psi/S_3]$, which is isomorphic to the
  classifying stack $\B Z$, where $Z$ is the centralizer of any element
  of $\Psi$. So the components of the inertia stack are isomorphic
  to $\B S_3$, $\B \mu_2$ and $\B \mu_3$.  

The double inertia $\II BS_3$ is the disjoint union of
  eleven components; three are 
  isomorphic to a point ($\B\{e\}$), corresponding to
  the conjugacy classes of the pairs 
  $((12),(13))$, 
  $((12),(123))$ and $((123),(12))$ respectively; three isomorphic to $\B\mu_2$, 
  corresponding to the conjugacy classes of the pairs
 $((1),(12))$, $((12),(1))$, and $((12),(12))$; 
 four 
 isomorphic to $\B\mu_3$, corresponding to the conjugacy classes of $((1),(123))$, $((123),(1))$, $((123),(123))$, and $((123),(132))$;
  and the identity component is isomorphic to $\B S_3$. 
Consider the inertial
  product with ${\R} = 0$ and ${\cs} = 0$. (This is just the usual
  orbifold product on $\B S_3$.)

  Let $\chi \in R(\mu_2)=K(\B\mu_2)$ be the defining character.
  Denote by $\chi |_{\B\mu_2} \in K(I\B S_3)$ the class which is
  $\chi$ on the sector isomorphic to $\B\mu_2$ (corresponding to the
  conjugacy class of a transposition in $S_3$) and 0 on all other
  sectors. Likewise, let $1|_{\B\mu_2} \in K(I\B S_3)$ be the class
  which is the trivial representation on the sector isomorphic to
  $\B\mu_2$ and 0 on all other sectors.  We will compare $\psi^2\left(
    \chi|_{\B\mu_2} \star 1|_{\B\mu_2}\right)$ and $\psi^2\chi
  |_{\B\mu_2} \star \psi^2\chi |_{\B\mu_2}$ and show that they are not
  equal in $K(I\B S_3)$.

Since $\R= 0$, the orbifold product is given by the formula
$$\alpha \star \beta  = \mu_*(e_1^*\alpha \cdot e_2^*\beta).$$
To compute the product, we note if $\alpha$ is supported on the sector
corresponding to the conjugacy class of $(12)$, then $e_1^*\alpha$ is supported
on the components of $\II\B S_3$ corresponding to the conjugacy classes of pairs
$$\left( (12), (1)\right), \left( (12), (13)\right), \left( (12), (12)\right),
\left( (12), (123)\right).$$
Similarly, $e_2^*\alpha$ is supported on the components corresponding to
the conjugacy classes of the pairs
$$\left((1), (12)\right), 
\left( (12), (13)\right), 
\left( (12), (12)\right),
\left( (123), (12)\right).$$ So if $\alpha, \beta$ are both
supported on the sector corresponding to 
$(12)$, then the classical product $e_1^*\alpha \cdot e_2^*\alpha$ is
supported on components of $\II \B S_3$, corresponding to the conjugacy
classes of the pairs $\left((12), (13)\right)$ and $\left(
  (12),(12)\right)$. The multiplication map $\mu$ takes the 
component corresponding to the conjugacy class of $\left((12),
  (13)\right)$ to the twisted sector isomorphic to $\B\mu_3$
corresponding to the conjugacy class of 3-cycles. Likewise, $\mu$ maps
the component corresponding to the conjugacy class of $\left( (12), (12)\right)$
to the untwisted sector $\B S_3$, which corresponds to the conjugacy class of the identity.

Identifying $K(\B G) = R(G)$, we see that $K(I\B S_3) = R(S_3) \oplus
R(\mu_2) \oplus R(\mu_3)$, while $K(\II\B S_3) = R(S_3) \oplus
R(\{e\})^3 \oplus R(\mu_2)^3 \oplus R(\mu_3)^4$.  Under this
identification the pullbacks $e_i^* \colon K(I\B S_3) \to K(\II\B S_3)$
correspond to restriction functors between the various representation
rings. Likewise, the pushforward $\mu_* \colon
K(\II\B S_3) \to K(I\B S_3)$ corresponds to the induced representation
functor.
Hence, 
\[
\chi|_{\B\mu_2} \star 1|_{\B\mu_2}  =  (\Ind_{\mu_2}^{S_3}
  \chi)|_{\B S_3} + (\Ind_{\{e\}}^{\mu_3} \Res_{\{e\}}^{\mu_2}\chi)|_{\B\mu_3}
 =  ({\sgn} + V_2)|_{\B S_3} + V_3|_{B\mu_3},
\]
where ${\sgn}$ is the sign representation on $S_3$, and $V_2$ is the
two-dimensional irreducible representation, and $V_3$ is the regular representation
of $\mu_3$.  
The character of
$\psi^2({\sgn} + V_2)$ has value $3$ at the
identity and at the conjugacy class of a 2-cycle, and it has value 0 on
3-cycles. On the other hand, $\psi^2(\chi) = \psi^2(1) = 1$ in
$R(\mu_2)$, so 
\[
\psi^2(\chi|_{\B\mu_2}) \star \psi^2(1|_{\B\mu_2})  =  1|_{B\mu_2} \star 1|_{B\mu_2}
 =  (1 + V_2)|_{\B S_3} + V_3|_{\B\mu_3}.
\]
The character of $1+ V_2$ has value 1 on
2-cycles, so $\psi^2({\sgn} + V_2) \neq (1 + V_2)$. Therefore, 
$$\psi^2(\chi|_{\B\mu_2} \star 1_{\B\mu_2}) \neq \psi^2 \chi|_{\B\mu_2} \star
\psi^2 1|_{\B\mu_2}.$$
\end{ex}

\section{$\lam$-positive elements,
the inertial dual, and inertial Euler classes}\label{sec:positive}

Every $\lam$-ring contains the semigroup of \emph{$\lam$-positive
elements}, which is an invariant of the $\lam$-ring structure. 
In the case of ordinary equivariant K-theory, every class of a rank $d$
vector bundle is a $\lam$-positive element, although the converse need not be true.
Nevertheless, $\lam$-positive elements of degree $d$ share many of the same
properties as classes of rank $d$-vector bundles; for example, they have a top Chern
class in Chow theory and an Euler class in K-theory. This is because the ordinary Chern character and Chern classes are compatible
with the $\lam$-ring and $\psi$-ring structures. 

In this section, we will introduce the framework to investigate the
$\lam$-positive elements of inertial K-theory for strongly Gorenstein
inertial pairs.  We will see that the $\lam$-positive elements of degree $d$
in inertial K-theory satisfy the inertial versions of these properties. 
We will also introduce a notion of duality for inertial K-theory, which is
necessary to define the inertial Euler class in inertial K-theory.

For the examples $\pro(1,2)$ and $\pro(1,3)$, we will see that the
set of $\lam$-positive elements yield 
integral structures
on inertial K-theory and inertial Chow theory  which will correspond, under a kind of mirror
symmetry,  
to the usual integral structures on ordinary K-theory and Chow theory of an associated crepant resolution of the 
orbifold cotangent bundle.

\begin{rem}
All results in this section hold for possibly non-diagonalizable $G$, provided
that $K_G(\IGX)$ is replaced by the augmentation completion of its
Gorenstein subring $\KGor_G(\IGX)$.
\end{rem}

We begin by defining the appropriate notion of duality for inertial K-theory.
\begin{df}  Consider the inertial K-theory
  $(K_G(\IGX),\star,1,\ovaug,\psit)$ of a strongly Gorenstein pair
  $(\R,\cs)$ associated to 
  a proper action of a diagonalizable group $G$ on $X$. 
  The \emph{inertial dual}  is the map $\Dt:K_G(\IGX) \to K_G(\IGX)$  defined by
\[
\Dt(\cv) := \cv^\dagger := \cv^*\cdot\rho(\cs^*),
\]
where 
\begin{equation}\label{eq.rhodef}
\rho(\cf) := (-1)^{\aug(\cf)} \det(\cf^*)
\end{equation} 
for all classes of locally free sheaves $\cf$ in $K_G(\IGX)$, and $\det(\cf) 
= \lambda^{\aug(\cf)} \cf $ 
is  the class of the usual determinant line bundle of $\cf$.  
Note that in this definition the dual ($^*$), as well as both $\aug$ and $\det$, are the usual, non-inertial forms.
\end{df}
\begin{thm}\label{prop.inertialdual} Consider the inertial K-theory
  $(K_G(\IGX),\star,1,\ovaug,\psit)$ of a strongly Gorenstein pair
  $(\R,\cs)$ for a diagonalizable group $G$ with a proper action on $X$.
\begin{enumerate}
\item \label{dual.one} $\Dt^2$ is the identity map, i.e., $\cf^{\dagger\dagger} = \cf$ for all $\cf\in K_G(\IGX)$.
\item \label{dual.two} For all $\ell\geq 1$, the inertial dual satisfies
\begin{equation}
\Dt\circ\ovaug = 
\ovaug\circ\Dt = \ovaug \quad  \mathrm{and} \quad \psit^\ell\circ\Dt  =
\Dt\circ\psit^\ell. 
\end{equation}
\item \label{dual.three} The inertial dual is a homomorphism of unital rings.
\end{enumerate}
\end{thm}
Before we give the proof of the theorem, we need to recall one fact from \cite{FuLa:85} about the ordinary dual in K-theory, and we need to prove a Riemann-Roch type of result for the ordinary dual.
\begin{lm}[\protect{\cite[I Lemma 5.1]{FuLa:85}}]\label{lm.FLlamwithdual}
Let $\cf$ be any locally free sheaf of rank $d$. Then for all $i$ with $0\le i \le d$ we have
\begin{equation}
\lambda^i(\cf) = \lambda^{d-i}(\cf^*)\lambda^d(\cf).
\end{equation}
\end{lm}
\begin{thm}[Riemann-Roch for the Ordinary Dual]\label{thm.ARRdual}
Using the hypotheses and notation from Theorem
\ref{thm:ARR}, and using the definition of $\rho$ given in Equation~\eqref{eq.rhodef}, for all $\cf$ in $K_G(Y)$ we have
\begin{equation}\label{eq:ARRdual} 
(\iota_*(\cf))^* = \iota_*(\rho(N_\iota^*)\cdot \cf^*).
\end{equation}

\end{thm}
\begin{proof}
We first observe, using Lemma~\ref{lm.FLlamwithdual}, that for any locally free sheaf $\cf \in K_G(Y)$ we have
\begin{equation}\label{eq.lambdadual}
\lam_{-1}(\cf)^* = \lam_{-1}(\cf)  \rho(\cf).
\end{equation}
We also observe that ordinary dualization commutes with pullback and is a ring homomorphism.  Because of these properties, the ordinary dual is a so-called \emph{natural operation}, and the desired result follows immediately from K\"ock's ``Riemann-Roch theorem without denominators'' \cite[Satz 5.1]{Koc:91}.
\end{proof}

\begin{proof}[Proof (of Theorem~\ref{prop.inertialdual})]
Part \ref{dual.one}  follows from the identity $\rho(\cf^*) = (\rho(\cf))^{-1}$. 

The first Equation of Part \ref{dual.two} is follows from the definition of
$\ovaug$. The second Equation of Part \ref{dual.two} follows from the
identity $\theta^n(\cs) = \theta^n(\cs^*) (\det(\cs))^{\aug(\cs)-1}$, which follows
from the splitting principle in ordinary K-theory. 

The proof of Part \ref{dual.three} 
is  identical to the proof 
that $\psit^n$ is a homomorphism for all $n\geq 1$, but where the Bott class
$\theta^n$ is replaced by the class $\rho$, and the Adams-Riemann-Roch Theorem
\ref{thm:ARR} is replaced by Theorem~\ref{thm.ARRdual}.  

\end{proof}

\begin{df}\label{df:LamPositive}
Let $(K,\cdot,1,\lam)$ be a $\lam$-ring.
For any integer $d \geq0$, an element  $\cv \in K$ is said to have \emph{$\lam$-degree $d$} if $\lam_t(\cv)$ is a degree-$d$ polynomial in $t$. The element $\cv$ is said to be a
\emph{$\lam$-positive element of degree $d$ of $K$} if it has $\lam$-degree
$d$ for $d\geq 1$ and  $\lam^d(\cv)$ is a unit of $K$. A $\lam$-positive element of degree
$1$ is said to be a \emph{$\lam$-line element of $K$}.  Let 
$\cP_d := \cP_d(K)$
be the set of $\lam$-positive elements of degree $d$ in $K$, 
and let $\cP = \sum_d \cP_d \subset K$ be the semigroup of positive elements.
\end{df}
\begin{rem}
If the $\lam$-ring $(K,\cdot,1,\lam)$ has an involutive homomorphism
$K\to K$ taking $\cf$ to $\cf^\triangledown$ that commutes with $\lam^i$ for all
$i\geq 0$, then it may be useful 
in the definition of a $\lam$-positive
element of degree $1$ to assume, in addition, that 
$\cv^{-1} = \cv^{\triangledown}$.
However, we will later see that this condition  
automatically holds 
for the virtual K-theory of 
$\B\mu_2$ (Proposition \ref{prop.bmutwo}) , $\pro(1,2)$ (Proposition
\ref{prop.p12line}), and $\pro(1,3)$ (Proposition \ref{prop:POneThreeVirtualK}).
\end{rem}

\begin{prop}\label{prop:lambda}
 Let $(K,\cdot,1,\lam)$ be a $\lam$-ring. 
\begin{enumerate}
\item \label{lambda.one} Addition in $K$ induces a map $\cP_{d_1}\times\cP_{d_2}\to
  \cP_{d_1+d_2}$ for all integers $d_1,d_2\geq 1$.
\item \label{lambda.two} Multiplication in $K$ induces a map $\cP_{d_1}\times\cP_{d_2}\to
  \cP_{d_1 d_2}$ for all integers $d_1,d_2\geq 1$. In particular, the set
  $\cP_1$ of $\lam$-line elements of $K$ forms a group. 
\item \label{lambda.three} 
If $K$ is torsion free, then an element $\cl$ in $K$ 
has $\lam$-degree $1$
if and only if
\begin{equation}\label{eq:psilineelt}
\psi^\ell(\cl) = \cl^\ell
\end{equation}
for all integers $\ell\geq 1$.
\item For all  $\cv$ in $\cP_d$,
\begin{equation}\label{eq:ovgavectorbundle}
\gamma_t(\cv-d) = \sum_{i=0}^d t^i (1-t)^{d-i} \lam^i(\cv).
\end{equation}
\item \label{lambda.four.a} 
For all integers
  $i\geq 0$ and $d\geq 1$, we have $\lam^i:\cP_d\to \cP_{{d\choose i}}$. 
Furthermore, 
if $K$ is an augmented
$\lam$-algebra over $\nq$
 with augmentation $\aug$, and
if $\cv$ belongs to $\cP_d$, then in $K$
\begin{equation}\label{eq:lambda.four.b}
\aug(\cv) = d,
\end{equation}
 and thus, 
\begin{equation}\label{eq:OAugOlami}
\aug(\lam^i(\cv)) = {d\choose i}. 
\end{equation}

\end{enumerate}
\end{prop}
\begin{proof}
Part \ref{lambda.one} follows from the fact that the product of invertible elements is
invertible. Part  \ref{lambda.two} follows from properties of the universal polynomials
$\PP_n$ appearing in Equation \eqref{lamring.Pn} of the definition of a $\lam$-ring. Part  \ref{lambda.three}
follows immediately from Equation~(\ref{eq:PsiFromLambda})
and the fact that $K$ is torsion free.

Equation~(\ref{eq:ovgavectorbundle}) holds since for all $\cv$ in $\cP_d$, we have 
\begin{equation*}
\gamma_t(\cv - d) = \frac{\lam_{\frac{t}{1-t}}(\cv)}{(1-t)^{-d}} 
= (1-t)^d \sum_{i=0}^d \left(\frac{t}{1-t}\right)^i \lam^i(\cv)
= \sum_{i=0}^d t^i (1-t)^{d-i} \lam^i(\cv).
\end{equation*}

To prove Part  \ref{lambda.four.a},
the properties of the universal polynomials $\PP_{m,n}$ 
(see Remark~(\ref{rem:UniversalPoly})) imply that 
$\lam^i:\cP_d\to\cP_{{d \choose i}}$ for all $i\geq 0$. 
Hence, if $\cv$ has $\lam$-degree $d$ where $d,i\geq 1$, then since $\lam^d\cv$ is invertible, so is $\lam^{{d\choose i}}(\lam^i(\cv))=(\lam^d\cv)^{{d-1\choose i-1}}$.

To prove Equation~(\ref{eq:lambda.four.b}) let us first suppose that $\cf := \cl$ belongs to $\cP_1$.
Applying $\aug$ to Equation (\ref{eq:psilineelt}) for $\ell=2$, we obtain
$\aug(\psi^2(\cl)) = \aug(\cl^2) = \aug(\cl)^2$, but $\aug(\psi^2(\cl)) =
\aug(\cl)$.  Thus, $\aug(\cl)^2 = \aug(\cl)$, but since $\cl$ is invertible and
$\aug$ is a homomorphism of unital rings, $\aug(\cl)$ is invertible. Therefore,
$\aug(\cl) = 1$. 
More generally, if $\cf$ belongs to $\cP_d$ for some integer
$d\geq 1$, then Equation~(\ref{eq:AugLambda}) implies that ${\aug(\cf) \choose
  d } = 1$, and 
\[
0 = {\aug(\cf)\choose d+1} = {\aug(\cf)\choose d} \frac{\aug(\cf)-d}{d+1} = \frac{\aug(\cf)-d}{d+1}.
\]
Therefore, $\aug(\cf)=d$.  

Finally, Equation (\ref{eq:OAugOlami})
follows from  Equations~(\ref{eq:AugLambda}) and (\ref{eq:lambda.four.b}).
\end{proof}

In ordinary equivariant K-theory $(K_G(X),\otimes,1,\aug)$, it is often useful
to assume that $[X/G]$ is connected. This is not an actual restriction, since
$K_G(X)$ can be expressed as the direct sum of $\lam$-rings or $\psi$-rings of
the form  $K_G(U)$, where $[U/G]$ is a connected component of $[X/G]$. The
condition that $[X/G]$ is connected is equivalent to the condition that the
image of the 
augmentation
is $\nz$ times the unit element $1$, i.e., one may
interpret the augmentation as 
a map
$\aug:K_G(X)\to\nz$. 

For an inertial K-theory $(K_G(\IGX),\pd ,1,\ovaug)$, an additional condition must
be imposed in order for the inertial augmentation to have 
image equal to $\nz$.
\begin{df}
Let $X$ be an algebraic space with an action of $G$, 
and let $(\R,\cs)$ be an inertial pair.  For each $m\in G$ the restriction of $\cs$ to $X^m$ is denoted by $\cs_m$.

We say that the action of $G$ on $X$ 
is \emph{reduced} 
with respect to the inertial pair $(\R,\cs)$ 
if $\cs_m = 0$ implies $m = 1$.
\end{df}
The following Proposition is immediate.
\begin{prop}
Consider the inertial K-theory $(K_G(\IGX),\pd,1,\ovaug)$
(respectively the rational inertial K-theory $(K_G(\IGX)_\nq,\pd ,1,\ovaug)$) 
for some inertial pair $(\R,\cs)$.
The image of the inertial augmentation $\ovaug$ is equal to $\nz$ (respectively $\nq$)
times the unit element $1$ of $K_G(\IGX)$ if and only if $[X/G]$ is connected
and the action of $G$ on $X$ is reduced
with respect to $(\R,\cs)$.
\end{prop}
In ordinary equivariant K-theory 
any vector bundle
of rank $d$ has $\lam$-degree $d$.
Thus,
if $[X/G]$ is connected, then by definition,
$(K_G(X),\cdot,1,\aug,\lam)$ (respectively $(K_G(X)_\nq,\cdot,1,\aug,\lam)$) is
generated as a group (respectively $\nq$-vector space) by the classes of vector 
bundles and hence by elements of $\cP$. 

In inertial K-theory $(K_G(\IGX)_\nq,\pd,1,\ovaug,\olam)$, the situation is more complicated.
Equation~(\ref{eq:OAugOlami}) implies that if $\cv$ is in $\cP_d$, then
for any any connected component $U$ of 
$\IGX \smallsetminus X^1$
which has $\cs$-age equal to
$0$, the restriction $\res{\cv}{U}$, must have ordinary rank equal to $0$ on $U$. 
Therefore, the $\nq$-linear span of $\cP_d$ cannot be equal to $K_G(\IGX)_\nq$.
Furthermore, even if $[X/G]$ is connected and the action of $G$ on $X$ is
reduced
with respect to the inertial pair $(\crr,\cs)$,
there is no \emph{a priori}  reason   
that 
$(K_G(\IGX)_\nk,\cdot,1,\ovaug,\olam)$
is generated as a 
$\nk$-vector space by its $\lam$-positive elements for any field
$\nk$ containing $\nq$.

\begin{crl}
The Gorenstein subring
$(\Ks_G(\IGX)_\nq,\star,1,\olam)$ is a $\lam$-subring of the inertial K-theory
which is preserved by the inertial dual.
\end{crl}
\begin{proof}
The proof follows from Part 2 and Part 4 of Proposition \ref{prop:lambda}
and the fact that the inertial dual maps $\cP_d\mapsto\cP_d$ for all $d$.
\end{proof}

One thing that makes the elements $\cP_d$ in
$(K_G(\IGX)_\nq,\cdot,1,\ovaug,\olam)$ interesting is that in many ways they
behave as though they were rank-$d$ vector bundles. In particular, they have
inertial Euler classes in both K-theory and Chow rings.

\begin{prop}
Let $(K_G(\IGX)_\nq,\pd,1,\ovaug,\olam)$ be the inertial K-theory of a
strongly Gorenstein pair $(\R,\cs)$ associated to a diagonalizable group $G$ with a proper action on $X$.
\begin{enumerate} 
\item The inertial Chern class $\ocv^1:\cP_1\to A^\bra{1}_G(\IGX)_\nq$ is
  a group homomorphism. 
\item For all  $\cv$ in $\cP_d$ and $\cl$ in $\cP_1$,
\begin{equation}\label{eq:ovChlineovexp}
\ovCh(\cl) = \ovexp(\ocv^1(\cl)),
\end{equation}
and
\begin{equation}\label{eq:ovcvectorbundle}
\ocv_t(\cv) = \sum_{i=0}^d \ocv^i(\cv) t^i,
\end{equation}
so $\ocv^i(\cv) = 0$ for all $i > d$.
\end{enumerate}
\end{prop}
\begin{proof}
Part 1 follows from the fact that for all $\cl_1$ and $\cl_2$ in $\cP_1$, 
$\ovCh(\cl_1\star \cl_2) = \ovCh(\cl_1)\star \ovCh(\cl_2)$. 
Picking off terms in
$A^\bra{1}_G(\IGX)_\nq$ and using $\ovCh^1 = \ovc^1$ and Equation
(\ref{eq:OAugOlami}) yields the desired result.

Equation~(\ref{eq:ovChlineovexp}) follows from
Equations~(\ref{eq:OChernSeries}) and (\ref{eq:ovcvectorbundle}), which yields
\[
1+t \ovc^1(\cl) = \ovexp\left(\sum_{n\geq 1} (-1)^{n-1} (n-1)! t^n \ovCh^n(\cl)\right),
\]
which implies that $\ovCh^n(\cl) = \ovc^1(\cl)^n/n!$, as desired.
Equation (\ref{eq:ovcvectorbundle}) follows from
Equations~(\ref{eq:oChernGammaConsistency}) and (\ref{eq:ovgavectorbundle}).
\end{proof}

The inertial dual allows us to introduce a generalization of the Euler class.
\begin{df}
\label{df:InertialEuleK} 
Let 
  $(K_G(\IGX)_\nq,\pd,1,\ovaug,\olam)$ be the inertial K-theory associated to $(\R,\cs)$. Let $\cv$
  belong to $\cP_d$.
The \emph{inertial Euler class in $K_G(\IGX)_\nq$} of $\cv$ is  
  $$\olam_{-1}(\cv^\dagger) = \sum_{i=0}^d (-1)^i \olam^i(\cv^\dagger).$$ 
 The \emph{inertial Euler class
of $\cv$ in $A^{\bra{d}}_G(\IGX)_\nq$} is defined to be $\ovc^d(\cf)$. 
\end{df}
The inertial Euler classes are multiplicative by Part \ref{lambda.one} of
Proposition \ref{prop:lambda} and the multiplicativity of $\ovc_t$ and $\olam_t$.

Finally, we observe that $\cP_1$ is preserved by the action of certain
groups. This will be useful in our analysis of the virtual K-theory of $\pro(1,n)$.
\begin{df}
Let $(K,\cdot,1,\psi,\aug)$ be 
a torsion-free,
augmented $\psi$-ring. A \emph{translation group of $K$} is 
an additive subgroup $J$ of $K$ such that for all $n\geq 1$, $j\in J$, and
$x \in K$, the following identities hold:
\begin{enumerate}
\item \label{eq.TransGroupOne} $\psi^n(j) = n j$, 
\item \label{eq.TransGroupTwo} $x \cdot j = \aug(x) j$,
\item \label{eq.TransGroupThree} 
$\aug(x) j \in J.$ 
\end{enumerate}
\end{df}

\begin{prop}\label{prop:JActsOnK}
Let $(K,\cdot,1,\psi,\aug)$ be 
a torsion-free 
augmented $\psi$-ring.
If $J$ is a translation 
subgroup of $K$, then $\aug(J) = 0$, $J^2 = 0$, and $J$ is an ideal of the ring
$K$. Furthermore, $J$ acts freely on $\cP_1$, where $J\times\cP_1\to\cP_1$
is $(j,\cl)\mapsto j+\cl$.  
\end{prop}
\begin{proof} For all $j$ in $J$ and integers $n\geq 1$, $\aug(\psi^n(j)) = \aug(j)$ by the
  definition of an augmented $\psi$-ring. On the other hand, $\aug(\psi^n(j)) =
  \aug(n j) = n \aug(j)$ for all integers $n\geq 1$ by 
  Condition~(\ref{eq.TransGroupOne}) 
  in
  the definition of a translation group. Therefore, $\aug(j) = 0$
  since $K$ is torsion free.
  The fact that $J^2 = 0$
  and $J$ is  an ideal of $K$ follows from 
 Conditions~(\ref{eq.TransGroupTwo}) 
  and (\ref{eq.TransGroupThree})
 in the definition of a translation group.

Consider $\cl$ in $\cP_1$ and $j$ in $J$.  We have 
\[
\psi^n(\cl+j) = \psi^n(\cl)+\psi^n(j) = \cl^n + n j =(\cl+j)^n,
\]
where the second equality is by Equation (\ref{eq:psilineelt}) and 
  Condition~(\ref{eq.TransGroupOne}) 
in the
definition of a translation group, and the last is from the binomial theorem 
and the fact that $J^2 = 0$
since $\aug(\cl) = 1$. 
Hence by Equation \eqref{eq:psilineelt},
$\cl + j$ has $\lam$-degree 1.
Also, notice that $(\cl^{-1}-j)(\cl+j) = 1$, so
$\cl+j$ is invertible 
and thus an element of $\cP_1$.
\end{proof}

\section{Examples}\label{sec:examples}

In this section, we work out some examples of 
inertial 
$\psi$-rings and
$\lam$-rings. 

\subsection{The classifying stack of a finite Abelian group}
In this section we discuss the case where $X$ is a point with a trivial action by a finite group $G$ 
and the trivial inertial pair $\R=0, \cs=0$. 
Since $G$ is zero-dimensional, its tangent bundle is $0$, so the orbifold and virtual inertial pairs (Definitions \ref{df:orbifold-in-pair}, \ref{df:virtual-in-pair})
are both trivial.   
We begin with some general results and conclude with explicit computations for the special case of the cyclic group $G= \mu_2$ of order 2.

\subsubsection{General results}
Let $X$ be a point with the trivial action of a finite Abelian group $G$.  The inertia
scheme is $\IGX = G$, which also has a trivial $G$ action. 
The orbifold
K-theory of $\B G:=[X/G]$ is additively  the Grothendieck group $K_G(\IGX) = K_G(G)$ of
$G$-equivariant vector bundles over $G$; however, the orbifold product on
$K_G(G)$ differs from the ordinary one, as we now describe. 

The double inertia manifold is $\II_G X =
G\times G$ with the diagonal conjugation action of $G$ (again, trivial);  the evaluation
maps $e_i:G\times G\to G$ are the projection maps onto the $i$th factor for
$i=1,2$; and $\mu:G\times G\to G$ is the multiplication map. Let $\cf$ and
$\cg$ be $G$-equivariant vector bundles on $G$, then $\cf\ti\cg :=
\mu_*(\cf\boxtimes\cg)$ is the $G$-equivariant vector bundle over $G$ whose
fiber over the point $m$ in $G$ is
\begin{equation}\label{eq:BGmult}
(\cf\star\cg)_m = \bigoplus_{m_1 m_2 = m}\cf_{m_1}\otimes\cg_{m_2},
\end{equation}
where the sum is over all pairs $(m_1,m_2) \in G^2$ such that $m_1 m_2 = m$.

The orbifold K-theory $(K_G(G),\star,\one)$ of $\B G$ can naturally be
identified with 
two better-known rings:  
first, the group ring $R(G)[G]$ of $G$ with coefficients in the representation ring $R(G)$ of $G$, and second, the representation ring $\Rep(D(G))$ of the Drinfeld double 
$D(G)$ of the group $G$ (see \cite[Thm 4.13]{KaPh:09}). 
The ring $\Rep(D(G))$  has been studied in some detail in \cite{DPR:90,KaPh:09,Wit:96}.

In this case the orbifold Chern classes are all trivial,
i.e., $\oc_t(\cf) = \one$ for all $\cf$.  
This follows from two facts. First,  $\cs=0$, 
so  $\ovCh_t(\cf) = \Ch_t(\cf)$ is the classical Chern character. Second,
$A^i(BG)_\nq = 0$ for $i > 0$ because $BG$ is a $0$-dimensional Deligne-Mumford stack. Thus,
$\Ch_t(\cF) = \rk(\cf)$ 
for every $\cf \in K_G(\IGX)$.

Since $\Sm= 0$ on $\IGX$, the orbifold Adams operations in $K_G(G)$ agree with the 
ordinary ones,
i.e.,  $\opsi^i :=
\psi^i$ for all $i\geq 1$. 

\subsubsection{The classifying stack $\B\mu_2$}
We now consider the special case where $G = \mu_2$ is the cyclic group of
order 2.
For  each $m\in G$ and each irreducible
representation $\alpha\in \Irrep(\mu_2)= \{\pm 1\}$,
let $V_m^\alpha$ denote the bundle on 
$G$ which is $0$ away from the one-point set $\{\,m\,\}\in\IGX = \mu_2$ and which is equal to $\alpha$ on $\{m\}$.
In this case the free Abelian group  $K_{\mu_2}(\mu_2)$ decomposes as
$K(I\B\mu_2) = K_{\mu_2}(\mu_2) = K_{\mu_2}(\{\,1\,\}) \oplus K_{\mu_2}(\{\,-1\,\})$
and has a basis consisting of the four elements $V_1^1,V_1^{-1}, V_{-1}^1, V_{-1}^{-1}$.

\begin{prop} \label{prop.bmutwo}
The orbifold $\lam$-ring 
$(K(IB\mu_2)_\nq,\star,\one,\olam)$
 satisfies the following:
\begin{equation}\label{eq:V11}
\olam_t(V_1^1) = \one + t
\end{equation}
\begin{equation}\label{eq:V1m1}
\olam_t(V_1^{-1}) = \one + t V_1^{-1}
\end{equation}
\begin{equation}\label{eq:Vm11}
\olam_t(V_{-1}^1) = \one + t V_{-1}^1 + \frac{t^2}{2 (1+t)} \left(1 - V_{-1}^1
\right)
\end{equation}
\begin{equation}\label{eq:Vm1m1}
\olam_t(V_{-1}^{-1}) = 1 + V_{-1}^{-1} t + \frac{t^2}{2 (1-t^2)} \left( 1 - t
V_1^{-1} - V_{-1}^1 + t V_{-1}^{-1}\right).
\end{equation}
There are four elements in $\cP_1$, namely, $V_1^{\pm 1}$ and
\begin{equation*}
\sigma_\pm := \frac{1}{2}\left(V_1^1 + V_1^{-1} \pm (V_{-1}^1 - V_{-1}^{-1})\right),
\end{equation*}
with multiplication given by
$\sigma_\pm\star\sigma_\pm = V_1^1$, and $V_1^{-1}\star\sigma_\pm = \sigma_\mp$, and
$\sigma_+\star\sigma_- = V_1^{-1}$.
\end{prop}
\begin{proof}
Equations (\ref{eq:V11}) and (\ref{eq:V1m1}) hold since $\{V_{1}^{1},V_1^{-1}\}$ generates a subring of $(K_{\mu_2}(\mu_2)_\nq,\star)$ isomorphic as a $\lam$-ring to the ordinary representation ring $K(\B\mu_2)$.

Let us introduce some notation. If $f(t)$ is a formal power series in $t$, let 
\[f_\pm(t) := \frac{1}{2}(f(t)\pm f(-t)).
\]
In order to prove Equation (\ref{eq:Vm11}), we observe that $\psit^k = \psi^k =\psi^{k+2}$ for all $k\geq 1$.  
This can be seen from Equation~\eqref{eq:AdamsOp} and the fact that any irreducible representation $V$ of $G$  is a line element satisfying $V^2 = 1$. 

Let $\lambda_t := \exp(\sum_{k=1}^\infty \frac{(-1)^{k-1}}{k} t^k \psi^k)$. Since 
\begin{equation}\label{eq:bmutwo-psi-minus-plus}
\psi^k(V_{-1}^1) = V_{-1}^1 \qquad\text{ for all $k\geq 1$,}
\end{equation} we obtain
\[
\olam_t(V_{-1}^1) 
= \exp\left(\sum_{k=1}^\infty \frac{(-1)^{k-1}}{k} t^k V_{-1}^1\right) = \exp(V_{-1}^1 \log(1+t)).
\]

Since we have
\[\left(V_{-1}^1\right)^k = 
\begin{cases} 
V_{-1}^1 &\mbox{if $k$ is odd} \\ 
V_1^1 = 1 & \mbox{if  $k$ is even},  
\end{cases} \]
we obtain
\begin{eqnarray*}
\exp(V_{-1}^1 \log(1+t)) &=& \exp_+(V_{-1}^1 \log(1+t)) + \exp_-(V_{-1}^1 \log(1+t)) \\
&=& \exp_+(\log(1+t)) +  V_{-1}^1 \exp_-(\log(1+t)) \\
&=&\frac{1+t+(1+t)^{-1}}{2} + V_{-1}^1 \frac{1+t-(1+t)^{-1}}{2}\\
&=&\frac{1+t}{2}(1+V_{-1}^1) + \frac{1}{2(1+t)}(1-V_{-1}^1),
\end{eqnarray*}
which agrees with Equation (\ref{eq:Vm11}).

The proof of Equation (\ref{eq:Vm1m1}) is similar. Since for all $k\geq 1$
\begin{equation}\label{eq:bmutwo-psi-minus-minus}
 \psi^k(V_{-1}^{-1}) = 
\begin{cases}
V_{-1}^{-1}&\mbox{if $k$ is odd} \\ 
V_{-1}^1  & \mbox{if  $k$ is even},
\end{cases}
\end{equation}
we obtain
$
\olam_t(V_{-1}^{-1}) =\exp(\psi_t(V_{-1}^1)) = \exp(V_{-1}^{-1}\log_-(1+t) + V_{-1}^1 \log_+(1+t))
$
and
\begin{equation}\label{eq:midway}
\olam_t(V_{-1}^{-1})=\exp(V_{-1}^{-1}\log_-(1+t)) \exp(V_{-1}^1 \log_+(1+t)).
\end{equation}
Since $\exp(V_{-1}^{-1}\log_-(1+t))= \exp_+(V_{-1}^{-1}\log_-(1+t))+\exp_-(V_{-1}^{-1}\log_-(1+t))$
\begin{align*}
{\ }&=\exp_+(\log_-(1+t))+V_{-1}^{-1}\exp_-(\log_-(1+t))\\
&=\frac{1}{2}\left(\exp\left(\frac{1}{2}(\log(1+t)-\log(1-t))\right)\right) +\left.\exp\left(-\frac{1}{2}(\log(1+t)-\log(1-t))\right)\right)\\
 &+ \frac{V_{-1}^{-1}}{2}\left(\exp\left(\frac{1}{2}(\log(1+t)-\log(1-t))\right)
\right) -\left.\exp\left(-\frac{1}{2}(\log(1+t)-\log(1-t))\right)\right)\\
&=\frac{1}{2}\left(\left(\frac{1+t}{1-t}\right)^\frac{1}{2}+\left(\frac{1-t}{1+t}\right)^\frac{1}{2}\right) +  \frac{V_{-1}^{-1}}{2}\left(\left(\frac{1+t}{1-t}\right)^\frac{1}{2}-\left(\frac{1-t}{1+t}\right)^\frac{1}{2}\right),
\end{align*}
we obtain
\begin{equation}\label{eq:midwayone}
\exp(V_{-1}^{-1}\log_-(1+t))=\frac{1 + t V_{-1}^{-1}}{(1-t^2)^{\frac{1}{2}}}.
\end{equation}

And since $\exp(V_{-1}^{1}\log_+(1+t))
=\exp_+(V_{-1}^{1}\log_+(1+t))+\exp_-(V_{-1}^{1}\log_+(1+t))$
\begin{align*}
{\ }&=\exp_+(\log_+(1+t))+V_{-1}^{1}\exp_-(\log_+(1+t))\\
&=\frac{1}{2}\left(\exp\left(\frac{1}{2}(\log(1+t)+\log(1-t))\right) + \exp\left(-\frac{1}{2}(\log(1+t)+\log(1-t))\right)\right)\\
&+\frac{V_{-1}^{1} }{2}\left(\exp\left(\frac{1}{2}(\log(1+t)+\log(1-t))\right)-\exp\left(-\frac{1}{2}(\log(1+t)+\log(1-t))\right)\right)\\
&=\frac{1}{2}\left(\left(1-t^2\right)^\frac{1}{2}+\left(1-t^2\right)^{-\frac{1}{2}}\right) + \frac{V_{-1}^{1} }{2}\left(\left(1-t^2\right)^\frac{1}{2}-\left(1-t^2\right)^{-\frac{1}{2}}\right),
\end{align*}
we obtain
\begin{equation}\label{eq:midwaytwo}
\exp(V_{-1}^{1}\log_-(1+t))=
\frac{2-t^2 - V_{-1}^1t^2}{2 (1-t^2)^\frac{1}{2}}.
\end{equation}
Plugging  Equations (\ref{eq:midwayone}) and (\ref{eq:midwaytwo}) into Equation  (\ref{eq:midway}) and then expanding using Equation (\ref{eq:BGmult}) yields Equation (\ref{eq:Vm1m1}).

The fact that $V_1^{\pm1}$ is in $\mathcal{P}_1$ is immediate, since  the
orbifold $\lam$-ring structure reduces to the ordinary $\lam$-ring
structure on the untwisted sector. 
The fact that $\sigma_\pm$  is in $\mathcal{P}_1$
follows from Equation~(\ref{eq:psilineelt}) as follows.  Since $\psit^k = \psi^k = \psi^{k+2}$ for all $k\ge1$, it suffices to check that $\psi^2(\sigma_\pm) = \sigma_\pm\star\sigma_\pm = V_1^1$.  But this is immediate from Equations~(\ref{eq:bmutwo-psi-minus-plus}) and (\ref{eq:bmutwo-psi-minus-minus}):
\[
\psi^2(\sigma_{\pm})  =  \frac12 ((V_1^1)^2 + (V_1^{-1})^2 \pm (V_{-1}^{1} -  V_{-1}^{1})) =V_1^1.
\] 
\end{proof}

\subsection{The virtual K-theory and virtual Chow ring of $\pro(1,n)$}
Let $X := \nc^2
\smallsetminus \{0\}$ and $G := \nc^\times$, with the action 
$\nc^\times\times X\to X$ defined by taking $(t,(a,b))$ to  $(t a, t^n b)$. 
In this section, we first develop some general results about 
the virtual K-theory and virtual Chow theory of
the weighted projective line $\pro(1,n):= [X/\nc^\times]$.
Recall (see Definition~\ref{df:virtual-in-pair}) that the inertial pair associated to the \emph{virtual product} is given by 
$\cs = \N$,
where $\N$ is the normal bundle of the 
projection
morphism $I_{\nc^\times}X \to X$, and 
$\R$ is given by Equation~(\ref{eq:VirtualR}).
We work out the full inertial K-theory and Chow theory
for the weighted projective spaces $\pro(1,2)$ and $\pro(1,3)$, and we compare our results with the usual K-theory and Chow theory of the resolution of singularities of the coarse moduli spaces of the cotangent bundles to these orbifolds.

\subsubsection{General results on the K-theory of $\pro(1,n)$ and its inertia}
Since the action of $\nc^\times$ on $\nc \smallsetminus \{0\}$ has
weights $(1,n)$, the only elements of $\nc^\times$ with nonempty fixed
loci are the $n$th roots of unity.  For 
$m \in \{0, \ldots , n-1\},$ 
let
$X^m$ denote the fixed locus of the element $e^{2\pi i m/n}$
in $X$.

With this notation $X^0 = X$, so $[X^0/\nc^\times] = \pro(1,n)$. For $m
> 0$, $X^m =\{(0,b) | b \neq 0\} = \nc^\times$. For each $m>0$, the
action of $\nc^\times$ on $X^m$ has weight $n$, so the quotient $[X^m/\nc^\times]$
is the classifying stack $\B\mu_n$. The
inertia variety is $I_{\nc^\times}X = \coprod_{m=0}^{n-1} X^m$, so the inertia
stack $I\pro(1,n)$ decomposes as $\pro(1,n) \sqcup \coprod_{m=1}^{n-1} \B\mu_n$.

We 
now compute
the classical equivariant Grothendieck and Chow rings of the inertia
variety, or equivalently the Grothendieck and Chow rings of the inertia stack.
\begin{nota}
  Let $\chi$ be the defining character of $\nc^\times$. 
We can associate to $\chi$ a $\nc^\times$-equivariant line bundle
on $X$. It is the trivial bundle $X \times \nc$ with $\nc^\times$-action
given by 
$\beta(a,b,v) = (\beta a, \beta^n b, \beta v)$.
For each $m$,
denote by $\chi_m$ the class in $K_{\nc^\times}(X^m)$ corresponding to
the pullback of this $\nc^\times$-equivariant line bundle to $X^m$.

The character $\chi$ has a first Chern class $c_1(\chi) \in
A^1_{\nc^\times}(pt)$, and we denote by $c_m$ the pullback of $c_1(\chi)$
to $A^1_{\nc^\times}(X^m)$ under the projection $X^m \to pt$. With
this notation $c_1(\chi_m) = c_m$.
\end{nota}
\begin{prop} \label{prop.KP1ncalculation} 
We have the following isomorphisms for all $m\in\{ 1,\ldots,n-1\}$:
\begin{equation} \label{eq.grothx0}
K_{\nc^\times}(X^0)= K(\pro(1,n)) \cong \frac{\nz[\chi_0]}{\langle (\chi_0-1)(\chi_0^n-1)\rangle}
\end{equation}
\begin{equation} \label{eq.grothxm}
K_{\nc^\times}(X^m)  = 
K(\B\mu_n) 
\cong \frac{\nz[\chi_{m}]}{\langle 
\chi_{m}^n -1\rangle}
\end{equation}
\begin{equation} \label{eq.chowx0}
A^*_{\nc^\times}(X^0)  = A^*(\pro(1,n)) \cong \frac{\nz[c_0]}{\langle n c_0^2 \rangle}
\end{equation} 
\begin{equation}\label{eq.chowxm}
A^*_{\nc^\times}(X^m)  = 
A^*(\B\mu_n) 
\cong \frac{\nz[c_{m}]}{\langle 
n c_{m}\rangle}.
\end{equation}
\end{prop}
\begin{proof}
Since $\nc^2$ is smooth, Thomason's equivariant resolution theorem \cite{Tho:87a}
identifies the equivariant K-theory of vector bundles
with the equivariant K-theory of coherent sheaves. It follows that 
there is a five-term localization exact sequence for equivariant K-theory \cite{Tho:87}
\begin{equation}\label{eq.exact5}
K_{\nc^\times}(\{0\}) \stackrel{i_*} \rTo K_{\nc^\times}(\nc^2) \stackrel{j^*}\rTo K_{\nc^\times}(X^0) \rTo 0,
\end{equation}
where $i \colon \{0\} \hookrightarrow \nc^2$ is a closed embedding and $j \colon X^0 \to \nc^2$ is an open immersion.
Equation (\ref{eq.exact5}) implies that $K_{\nc^\times}(X^0)$ is the quotient of
$K_{\nc^\times}(\nc^2)$ by the image of $K_{\nc^\times}(\{0\})$ under
the pushforward induced by the inclusion 
$i$.
Since $\nc^2$ is a representation of $\nc^\times$,
the homotopy property of equivariant K-theory implies that
$K_{\nc^\times}(\nc^2) = \Rep(\nc^\times) = \nz[\chi, \chi^{-1}]$.
The projection formula implies that
$i_*K_{\nc^\times}(\{0\})$ is an ideal in $\nz[\chi, \chi^{-1}]$, and
$K_{\nc^\times}(X^0)$ is the quotient of $\nz[\chi,\chi^{-1}]$ 
by this ideal.  By the self
intersection formula in equivariant K-theory \cite[Corollary 3.9]{Koc:98},
$i^*i_*K_{\nc^\times}(\{0\}) = \euler(N_{\{0\}})
K_{\nc^\times}(\{0\})$, where $N_{\{0\}}$ is the normal
bundle to the origin in $\nc^2$.  Since $\nc^\times$ acts with weights
$(1,n)$, the class of the normal bundle is $\chi + \chi^{n}$ and
$\euler(N_{\{0\}}) = (1-\chi^{-1})(1-\chi^{-n})$.
Since the pullback
$i^*\colon K_{\nc^\times}(\nc^2) \to K_{\nc^\times}(\{0\})$ is an
isomorphism, $i_*(K_{\nc^\times}(\{0\})$ is the ideal 
generated by $(1-\chi^{-1})(1-\chi^{-n})$. Thus, $K_{\nc^\times}(X^0) = \nz[\chi,
\chi^{-1}]/\langle(1-\chi^{-1})(1 - \chi^{-n})\rangle$.  Clearing denominators and
observing that the relation already implies that $\chi$ is a unit, we
have the presentation $K_{\nc^\times}(X^0) = \nz[\chi]/\langle(\chi-1)(\chi^n
-1)\rangle$. Since $\chi_0$ is our notation for the pullback of $\chi$
to $X^0$, we obtain the presentation $\nz[\chi_0]/\langle (\chi_0-1)(\chi^n_0 -1)\rangle$.

For $m>0$ observe that if $\nc^\times$ acts on $\nc^\times = \nc \smallsetminus
\{0\}$ by $\lambda \cdot v = \lambda^n v$, then the
$\nc^\times$-equivariant normal bundle to $\{0\}$ in $\nc$ is
$\chi^{n}$. The same argument as above implies that 
$K_{\nc^\times}(\nc^\times) = \Z[\chi,
\chi^{-1}]/\langle 1- \chi^{-n}\rangle$. Clearing denominators and using the
notation $\chi_m$ for $\chi$ on $X^m$ gives the desired presentation.

The proof in Chow theory is similar. We again use the 5-term
localization sequence for equivariant Chow groups \cite{EdGr:98} to see that
$A^*_{\nc^\times}(X^m)$ is a quotient of $A^*_{\nc^\times}(pt) =
\nz[c_1(\chi)]$. We can again apply the self-intersection formula. In
Chow theory $\euler(\chi) = c_1(\chi)$, while $\euler(\chi + \chi^n) =
c_2(\chi + \chi^n) = n(c_1(\chi))^2$, which gives the relations in
\eqref{eq.chowxm} and \eqref{eq.chowx0}.
\end{proof}

\begin{rem}
 As a consequence of the relations in Proposition \ref{prop.KP1ncalculation}, an
additive basis for $K(I\pro(1,n))$ is given by $n^2+1$ classes of
the form $\chi_m^k$, where the subscript 
  refers to the sector, while the superscript is an exponent.
Including the untwisted sector $X^0$ there are $n$ sectors, 
so $0 \leq m \leq n-1$. If $m > 0$, 
then the exponent $k$ is in $[0,n-1]$, while
if $m =0$, then exponent $k$ is in $[0,n]$.

Similarly, the classes $\{c_m^k\}_{k \in {\mathbb N}}$ 
for $0 \leq m  \leq n-1$
generate
$A^*_{\nc^\times}(I_{\nc^\times}(X))=A^*(I\pro(1,n))$. Again, in the notation $c_m^k$
the subscript refers to the sector and the superscript 
to the exponent.
Note the relations in the
presentation imply that only $c_0$ and the fundamental classes $c_m^0
= [X^m]$ are non-torsion.
\end{rem}
\begin{rem}
  If $f \colon X \to Y$ is any morphism of $G$-varieties, then the
  pullback $f^* \colon K_G(Y) \to K_G(X)$ is a homomorphism of
  $\lambda$-rings, since for any $G$-equivariant vector bundle
  $\Lambda^k(f^*V) = f^*(\Lambda^k V)$. Applying this observation to
  the pullbacks $K_{\nc^\times}(\nc^2) \to K_{\nc^\times}(X^0)$ and
  $K_{\nc^\times}(\nc) \to K_{\nc^\times}(X^m)$, this means that for all $m\geq 0$ the
  classical $\lambda$-ring structure on $K_{\nc^\times}(X^m)$ is
  induced from the
usual $\lambda$-ring structure on $\Z[\chi_m,\chi_m^{-1}]$ defined by
setting $\lam_{t}(\chi_m^k) = 1 + t \chi_{m}^k$.
\end{rem}
\begin{rem}
For any $m > 0$ the map $X^m \to X$ is an embedding of codimension $1$,
so the $\cs$-age of $X^{m}$ is $1$, and the age of $X^0$ is $0$.
Hence the virtual degree of $c_0$ is $1$, as is the virtual degree of the fundamental class $c_m^0=[X^m] $
for $m>0$. 

The virtual 
augmentation $\ovaug:K(I\pro(1,n))\to K(I\pro(1,n))$ satisfies 
$\ovaug(\chi_0^a) = \chi_0^0$, 
and for $m\in\{1,\ldots,n-1\}$ we have
$
\ovaug(\chi_{m}^a) = 0
$
for all $a$ in $\nz$.

The virtual 
Chern character homomorphism is very simple, since in $A^*(I\Pro(1,n))_\Q$, 
$c_0^k =0$ for $k>1$, 
and if $m>1$, then $c_m^l=0$ for $l \neq 0$. Stated more precisely, 
$\ovCh:K(I\pro(1,n))\to
A^*(I\pro(1,n))_\nq$ satisfies 
\begin{equation}
\ovCh(\chi_0^a) = c_0^0 + a c_0^1,
\end{equation}
for all $a \in \Z$.  And, for $m\in\{1,\ldots,n-1\}$ we have $
\ovCh(\chi_{m}^a) = c_{m}^0$.
\end{rem}

We now compute the virtual product.
\begin{thm} \label{thm.P1nvirtualprod}
The virtual product on $K(I\pro(1,n))$ satisfies
\[
\chi_{m_1}^{a_1} \star \chi_{m_2}^{a_2} = 
\begin{cases}
\chi_{m_1+m_2}^{a_1+a_2}
&\mbox{if $m_1=0$ or $m_2 = 0$,}\\
\chi_{0}^{a_1+a_2}\left(1 - 2\chi_{0}^{-1}+\chi_{0}^{-2}\right)&\mbox{if $m_1+m_2 = n$,}\\
\chi_{m_1+m_2}^{a_1+a_2}\left(1-\chi_{m_1 + m_2}^{-1}\right)&\mbox{otherwise},
\end{cases}
\]
and the virtual product in $A^*(I\pro(1,n))$ satisfies
\[
c_{m_1}^{a_1} \star c_{m_2}^{a_2} = 
\begin{cases}
c_{m_1+m_2}^{a_1+a_2}
&\mbox{if 
$m_1=0$ or $m_2 = 0$,}\\
c_{0}^{a_1+a_2+2}&\mbox{if $m_1+m_2 = n$,}\\
c_{m_1+m_2}^{a_1+a_2+1} &\mbox{otherwise}.
\end{cases}
\]
Here the sum $m_1 + m_2$ is understood to be reduced modulo $n$ and
all products on the right-hand side are the classical product in
$K_{\nc^\times}(X^{m_1+m_2})$ (or $A_{\nc^\times}^*(X^{m_1+m_2})$). In particular, the classes $\chi_m^{-1}$ are defined via \eqref{eq.grothx0} and \eqref{eq.grothxm}.
\end{thm}

\begin{rem}
  Since $c_0^2 = 0$ in $A^*(I\Pro(1,n))_\Q$, and since for all $m > 0$, we have $c_m = 0$
  in $A^*(I\Pro(1,n))_\Q$, 
   Theorem \ref{thm.P1nvirtualprod} implies that all products
$c_{m_1}^{a_0} \star c_{m_2}^{a_1}$ are equal to 0
unless  one of 
the classes
is the identity $c_0^0$. It follows that the rational virtual Chow ring is isomorphic to the graded ring $\Q[t_0, t_1, \ldots t_{n-1}]/\langle t_0,\ldots , t_{n-1} \rangle^2$, where $t_0$ corresponds to $c_0^1$,  and $t_m$ corresponds to $c_m^0$ for 
all $m\in\{ 1, \ldots , n-1\}$. 
\end{rem}
Before proving
  Theorem \ref{thm.P1nvirtualprod}, we need some notation for
  $K_{\nc^\times}(I^2_{\nc^\times}X)$ 
  and
  $A^*_{\nc^\times}(I^2_{\nc^\times}X)$.
\begin{nota}
Given a pair $(m_1,m_2) \in (\Z_n)^2$
let $X^{m_1,m_2} = X^{m_1} \cap X^{m_2}$. Unless $m_1 = m_2 =0$, 
$X^{m_1,m_2}= \{(0,b)|b \neq 0\} \subset X$ and $X^{0,0} = X$. 
The double inertia 
decomposes as 
$I^2_{\nc^\times}X = \coprod_{(m_1,m_2) \in (\nz_n)^2} 
X^{m_1,m_2}$. For each pair $(m_1,m_2)$, let $\chi_{m_1,m_2} \in K_{\nc^\times}(X^{m_1,m_2})$ be the class corresponding to the character $\chi \in \Rep(\nc^\times)$. 
With this notation, 
Proposition \ref{prop.KP1ncalculation} implies that
$K_{\nc^\times}(X^{m_1,m_2}) = \Z[\chi_{m_1,m_2}]/\langle\chi_{m_1,m_2}^n -1\rangle$ when
$(m_1,m_2) \neq (0,0)$ and $K_{\nc^\times}(X^{0,0}) = \Z[\chi_{0,0}]/\langle(\chi_{0,0}-1)(\chi_{0,0}^n-1)\rangle$.

Similarly, we let $c_{m_1,m_2}$ be the class in $A^1_{\nc^\times}(X^{m_1,m_2})$
corresponding to $c_1(\chi)$.
\end{nota}

\begin{proof}[Proof of Theorem \ref{thm.P1nvirtualprod}]
We first use the 
Equation (\ref{eq:DefOfR})
with $\cs= \N$
and compute the restriction of $\R$ to $X^{m_1,m_2}$. With our additive notation,
the multiplication map 
$\mu \colon I^2_{\nc^\times} X \to I_{\nc^\times}X$
 maps
$X^{m_1,m_2} \to X^{m_1 + m_2}$, so in $K_{\nc^\times}(X^{m_1,m_2})$ we have:
\begin{equation} \label{eq.RP1nvirt}
\R|_{X^{m_1,m_2}} = 
(
e_1^*N_{m_1} + e_2^*N_{m_2} - \mu^*N_{m_1+m_2} + T_\mu,
)|_{X^{m_1,m_2}},
\end{equation}
where $N_{m}$ denotes the normal bundle to $X^m$ in $X$.

First suppose that $m_1=0$. Then $X^{m_1,m_2} = X^{m_2} =
X^{m_1+m_2}$. It follows that $\mu \colon X^{m_1,m_2} \to X^{m_1 +
  m_2}$ is the identity map, so 
$(T_\mu)|_{X^{m_1,m_2}} = 0$.  
Also, $N_{m_1} = 0$ and
$N_{m_1+m_2} = N_{m_2}$, so plugging into Equation \eqref{eq.RP1nvirt} gives
$\R|_{X^{m_1,m_2}} = 0$.
In this case, $\chi_{m_1}^{\alpha_1} \star
\chi_{m_2}^{\alpha_2}$ corresponds to the usual product $\chi^{\alpha_1}
\chi^{\alpha_2} = \chi^{\alpha_1+\alpha_2}$, but viewed as an element
of $K_{\nc^\times}(X^{m_1+m_2})$. In our notation, this class is
$\chi_{m_1+m_2}^{\alpha_1 + \alpha_2}$.

Next suppose that $m_1,m_2$ are nonzero, but $m_1+m_2 = n$. In this case
$X^{m_1,m_2} = X^{m_1} = X^{m_2} = \{(0,b)|b \neq 0\}$, while
$X^{m_1 + m_2} = X^0 = \nc^2 \smallsetminus \{0\}$. 
Since $\nc^\times$ acts with weights $(1,n)$, the normal bundle to $\{(0,b)|b\neq 0\} \subset \nc^2 \smallsetminus \{0\}$ is the bundle determined by the character $\chi$, so in our notation $N_{m_1} =\chi_{m_1}$ and $N_{m_2} = \chi_{m_2}$, and $N_{m_1+m_2} = 0$.  
The map $\mu \colon X^{m_1, m_2}
\to X^{m_1 + m_2}$ is the inclusion and $(T_\mu)|_{X^{m_1,m_2}} = -(N_\mu|_{X^{m_1,m_2}})$ corresponds to
the class $-\chi$, which on $X^{m_1,m_2}$ we denote by $-\chi_{m_1,m_2}$.
Since
\begin{eqnarray*}
\R|_{X^{m_1, m_2}} &  = &  
e_1^*\chi_{m_1}|_{X^{m_1,m_2}} + e_2^*\chi_{m_2}|_{X^{m_1,m_2}}  - \chi_{m_1,m_2}
\\
& = & \chi_{m_1,m_2} + \chi_{m_1,m_2} - \chi_{m_1,m_2}
=  \chi_{m_1,m_2}
,
\end{eqnarray*}
it follows that
\[\chi_{m_1}^{\alpha_1}\star \chi_{m_2}^{\alpha_2} = \mu_*\left(\chi_{m_1,m_2}^{\alpha_1}  \cdot \chi_{m_1,m_2}^{ \alpha_2} \cdot \euler(\chi_{m_1,m_2})\right) =  \mu_*\left(\chi_{m_1,m_2}^{\alpha_1 + \alpha_2} (1 -\chi_{m_1,m_2}^{-1})\right).
\]
Since the class $\chi_{m_1,m_2}$ is pulled back from the character $\chi \in \Rep(\nc^\times)$, the projection formula 
yields the further simplification
$\chi_{m_1}^{\alpha_1} \star \chi_{m_2}^{\alpha_2} =\chi_{m_1+m_2}^{\alpha_1 + \alpha_2}
(1  - \chi_{m_1}^{-1})\mu_*(1)$. To compute $\mu_*(1)$ consider the diagram of inclusions
\[
\begin{diagram}&\nc & \rInto^{j}  & \nc^2 &\\
&\uInto & & \uInto &\\
X^{m_1,m_2} = &\nc \smallsetminus \{0\} & \rInto^{\mu} & X^{m_1+m_2} = &\nc^2 \smallsetminus \{0\}.
\end{diagram}
\]

Then $\mu_*(1)$ is the restriction to $K_{\nc^\times}(X^{m_1+m_2})$
 of the image of $j_*(1)$. By the self-intersection formula,
$j^*j_*(1)  = \euler(N_j) = (1 - \chi^{-1})$ under the identification
of $K_{\nc^\times}(\nc) = \Rep(\nc^\times)$.
Since $j^*$ is an isomorphism, we conclude that $j_*(1) = (1-\chi^{-1})$, and then
restricting to $K_{\nc^\times}(X^{m_1+m_2})$, we obtain $\mu_*(1) = 
(1-\chi_{m_1 + m_2}^{-1})$.
Hence 
$$\chi_{m_1}^{\alpha_1} \star \chi_{m_2}^{\alpha_2} = \chi_{m_1+m_2}^{\alpha_1 + \alpha_2}(1-\chi_{m_1 + m_2}^{-1})^2.$$

If $m_1, m_2 \neq 0$ and $m_1 + m_2 \neq 0$, $X^{m_1,m_2} = X^{m_1} = X^{m_2} = X^{m_1 + m_2}$ so $e_1, e_2, \mu$ are all identity maps. In this case,
$$\R_{|X^{m_1,m_2}} = e_{1}^*\chi_{m_1}|_{X^{m_1,m_2}} + e_{2}^*\chi_{m_2}|_{X^{m_1,m_2}} - \mu^*\chi_{m_1 + m_2}|_{X^{m_1,m_2}} =\chi_{m_1,m_2},$$
and
$$\chi_{m_1}^{\alpha_1}\star \chi_{m_2}^{\alpha_2} = \chi_{m_1+ m_2}^{\alpha_1 + \alpha_2} (1 - \chi_{m_1 + m_2}^{-1}).$$

The proof in Chow theory is similar. When $m_1,m_2 \neq 0$, then $\euler(\R) = 
c_{m_1,m_2} \in A^1_{\nc^\times}(X^{m_1,m_2})$, and when $m_1 + m_2 = n$,
then $\mu_*(1) = c_{m_1+m_2}$, which gives the factors of $c_{m_1+ m_2}^2$
and $c_{m_1 + m_2}$ appearing above.
\end{proof}

In order to calculate the virtual 
$\psi$-operations, for all $m\in\{1,\ldots,n-1\}$ we need the
$\ell$th Bott class $\theta^\ell(\cs_{m}^*)$ in $K_{\nc^\times}(X^{m})$, which satisfies 
\[
\theta^\ell(\cs_{m}^*) = \theta^\ell(\chi_{m}^{-1}) = 
\sum_{i=0}^{\ell-1} \chi_{m}^{-i}.
\]
Applying Equation~(\ref{eq:def-psi-tilde}) gives the virtual 
$\psi$-operations $\opsi^k:K(I\pro(1,n))\to K(I\pro(1,n))$.

\begin{df} Let $\nk$ be $\nq$ or $\nc$. For all $m\in\{1,\ldots,n-1\}$,
  let  $\Delta_m =   \sum_{i=0}^{n-1} \chi_m^i$ in $K_{\nc^\times}(X^m)$
  (respectively $K_{\nc^\times}(X^m)_\nk$)  and 
  $\Delta_0 = -\chi_0^0+\chi_0^n$
   in
  $K_{\nc^\times}(X^0)$ (respectively $K_{\nc^\times}(X^0)_\nk$). Let $J$ (respectively $J_\nk$)
be the additive group (respectively $\nk$-vector space) generated by
  $\{\Delta_i\}_{i=0}^n$. 
Let
$\psit^0$ be the inertial augmentation $\ovaug$.
\end{df}

\begin{lm} \label{lm:JJAndJK} Let $(K(I\pro(1,n)),\star,\bone,\ovaug,\psit)$ be the virtual K-theory ring.
\begin{enumerate}
\item For all $m\in\{0,\ldots,n-1\}$ and $\cf_m$ in $K_{\nc^\times}(X^m)$, we have the identity with respect to the ordinary product
\begin{equation}\label{eq:DeltamProduct}
\Delta_m\cdot \cf_m = \aug_m(\cf_m)\Delta_m.
\end{equation}
\item For all $j$ in $J$ and $\cf$ in the virtual K-theory ring $K(I\pro(1,n))$, 
\begin{equation}\label{eq:JInertialProduct}
\cf \star j = \ovaug(\cf) j, \qquad J\star J = 0, \quad \text{ and } \quad \ovaug(J) = 0. 
\end{equation}
\item For all $\ell\geq 1$ and $j\in J$, we have the identity
\begin{equation}\label{eq:PsitJ}
\psit^\ell(j) = \ell j.
\end{equation}
\end{enumerate}
In particular, $J$ is a translation group of the virtual K-theory $K(I\pro(1,n))$.
\end{lm} 
\begin{proof}
Equation (\ref{eq:DeltamProduct}) follows from the identity $(\chi_0^n-1)(\chi_0^1-1)=0$ in $K_{\nc^\times}(X^0)$, and $\chi_m^n-1=0$ in $K_{\nc^\times}(X^m)$ for all $m\not= 0$.

Equation (\ref{eq:JInertialProduct}) follows from Theorem~\ref{thm.P1nvirtualprod} and Equation (\ref{eq:DeltamProduct}). The fact that $J\star J=0$ follows from Equation (\ref{eq:JInertialProduct}) and the fact that $\ovaug(\Delta_m) = 0$ for all $m$.

To prove equation (\ref{eq:PsitJ}), we first consider
\begin{eqnarray*}
\psit^\ell(\Delta_0) &=& \psi^\ell(-1+\chi_0^n) = -1+\chi_0^{n\ell} 
=-1+(1+(\chi_0^n-1))^\ell \\ 
&=& -1+(1+\ell(\chi_0^n-1)) = \ell \Delta_0,
\end{eqnarray*}
where we have used the binomial series and the relation $(\chi_0^n-1)(\chi_0^1-1)=0$ in the fourth equality.
Let $m\not=0$, $\zeta_n := e^\frac{2\pi i}{n}$, and $x = \chi_m^1$, and assume in the following that all products are ordinary products. By definition, 
\begin{eqnarray*}
\psit^\ell(\Delta_m) &=& \psi^\ell(\Delta_m)\cdot\theta^\ell(x^{-1}) 
=\psi^\ell(\sum_{i=0}^{n-1} x^i)\sum_{j=0}^{\ell-1}(x^{-j}) 
=\sum_{i=0}^{n-1} (x^\ell)^ i \sum_{j=0}^{\ell-1} x^{-j}.
\end{eqnarray*}
To prove Equation (\ref{eq:PsitJ}), consider the algebra isomorphism
\[
K_{\nc^\times}(X^m)_\nq = \frac{\nq[x]}{\langle x^n-1 \rangle}\rTo^\Upsilon \nq\times \nq[t]/(1+t+\cdots+t^{n-1})
\]
defined by $\Upsilon( f ) := (f(1),f(\zeta_n))$.  Then $\Upsilon(\psit^\ell(\Delta_m)) = (n\ell, 0)=\ell\Upsilon(\Delta_m)$.  
\end{proof}

\begin{prop}\label{PsiPeriodic} 
Let $\varphi_0:K(I\pro(1,n))\to\nz$ be the additive map that is supported on $K_{\nc^\times}(X^0)$ such that $\varphi_0(\chi_0^s) = s$ for all $s\in\{0,\ldots,n\}$.

For all $k\geq 0$ and $a\in\{0,\ldots,n-1\}$, we have the identity in virtual K-theory $(K(I\pro(1,n)),\star,1,\ovaug,\psit)$
\begin{equation}\label{eq:periodic}
\psit^{n k + a} = \psit^a + k \Delta_0 \varphi_0 + \sum_{m=1}^n k \Delta_m \aug_m,
\end{equation}
where $\aug_m(\cf)$ denotes the ordinary augmentation of $\cf_m$ in $K_{\nc^\times}(X^m)$ of $\cf$.
\end{prop}
\begin{proof}
For all $k\geq 1$, let $\psit_m^k(\cf) := \psit^k(\cf_m)$ for all $\cf=\sum_{m=0}^n 
\cf_m$, where $\cf_m$ belongs to $K_{\nc^\times}(X^m)$. 

If $a\in\{0,\ldots,n-1\}$, $k\geq 0$ , $s\in\{0,\ldots,n\}$ and $x=\chi_0^1$, then
\[
\psit_0^{nk+a}(x^s) = (x^n)^{k s} x^{a s} = (1+(x^{n}-1))^{k s} x^{a s} 
=(1+k s (x^n-1)) x^{a s} = x^{s a} + k s \Delta_0,
\]
where we have used the relation $(x^n-1)(x-1) = 0$ in $K_{\nc^\times}(X^0)$ in the third and 
fourth equalities. Therefore, for all $n,k\geq 0$ and $a\in\{0,\ldots,n-1\}$, we have 
\begin{equation}\label{eq:untwistedperiodic}
\psit_0^{n k+a} = \psit_0^a +k  \Delta_0 \varphi_0.
\end{equation}
If $m\in\{1,\ldots,n-1\}$, then, adopting the convention that $\theta^0(0) = 1$ and $\theta^0(\chi_m^s) = 0$ for all $s$, we obtain
\begin{eqnarray*}
\psit_m^{n k+a}(\chi_m^s) &=&\psi_m^{n k+a}(\chi_m^s)\theta^{n k+a}(\cs_m^*)\\
&=&\psi_m^{a}(\chi_m^s)(k \Delta_m+\theta^a(\cs_m^*))
=k \psi_m^a(\chi_m^s)\Delta_m + \psi_m^a(\chi_m^s)\theta^a(\cs_m^*) \\
&=&k\aug_m(\psi_m^a(\chi_m^s))\Delta_m + \psit^a_m(\chi_m^s) = k \Delta_m+\psit_m^a(\chi_m^s),
\end{eqnarray*}
where we have used periodicity of $\psi$, the fact that $\cs_m = \chi_m^1$ for all $m\in\{1,\ldots,n-1\}$, the relation $(\chi_m^1)^n-1=0$ in $K_G(X^m)$ (with respect to the ordinary multiplication), Equation (\ref{eq:DeltamProduct}), 
and the fact that $\aug_m\psi^a_m=\aug_m$. Consequently, we have
\begin{equation}\label{eq:twistedperiodic}
\psit_m^{n k + a} = \psit_m^a + k \Delta_m \aug_m
\end{equation}
for all $n,k\geq 0$, $a\in\{0,\ldots,n-1\}$, and $m\in\{1,\ldots,n-1\}$.

Equations (\ref{eq:untwistedperiodic}) and (\ref{eq:twistedperiodic}) yield Equation (\ref{eq:periodic}).
\end{proof}

\begin{prop}\label{prop:FiniteEqsForLineElts} An 
invertible
element $\cl$ in virtual
  K-theory 
$(K(I\pro(1,n))_\nq,\star,1,\ovaug,\psit)$
 is a $\lambda$-line element
  with respect to its inertial $\lambda$-ring structure if and only if
  $\ovaug(\cl) = \bone$ and Equation (\ref{eq:psilineelt}) holds for all
  $\ell\in\{1,\ldots,n\}$. 
\end{prop}
\begin{proof} 
First, Equation~(\ref{eq:psilineelt}) holds for
$\ell=1$ by definition of a $\psi$-ring.
Suppose that $\cl$ in $K(I\pro(1,n))_\nq$ satisfies Equation
(\ref{eq:psilineelt}) for 
all $\ell\in \{1,\ldots,n\}$.  
We now prove that Equation
(\ref{eq:psilineelt}) holds for all $\ell$.  We do this by induction on $k$ in the expression $nk+a$, as follows.  Suppose for each $a\in\{1,\ldots,n\}$ there exists $k\ge 0$ such that Equation (\ref{eq:psilineelt}) holds for all $\ell \in\{ a, n+a, \ldots, nk+a\}$. Equation
(\ref{eq:periodic}) implies that 
\begin{equation}\label{eq:shortperiod}
\psit^{n(k+1)+a}(\cl) = \psit^a(\cl) + (k+1) j(\cl),
\end{equation}
where $j(\cl) := \varphi_0(\cl)\Delta_0+\sum_{m=1}^n \Delta_m\aug_m(\cl)$ belongs to  $J$. However,
\begin{eqnarray*}
\cl^{n (k+1) + a} &=& \cl^{n k + a} \cl^n =(\psit^a(\cl)+k j(\cl))(\psit^0(\cl)+j(\cl))\\
&=&(\psit^a(\cl)+k j(\cl))(1+j(\cl))\\
&=&\psit^a(\cl) + k j(\cl) +  \psit^a(\cl) j(\cl) + k j(\cl)^2 \\
&=&\psit^a(\cl)+(k+1) j(\cl) = \psit^{n(k+1)+a}(\cl),
\end{eqnarray*}
where we have used the induction hypothesis and Equation (\ref{eq:shortperiod}) in the second equality, the definition $\psit^0=\ovaug$ in the third equality, Lemma (\ref{lm:JJAndJK}) in the fifth,  the fact that $\ovaug\circ\psit^q = \ovaug$ 
in the fifth,
and Equation (\ref{eq:shortperiod}) in the 
sixth.
\end{proof}

\begin{rem}\label{rem:LineEltCalc}
Proposition (\ref{prop:FiniteEqsForLineElts}) reduces
the problem of finding $\lam$-line elements of $K(I\pro(1,n))_\nq$ to solving a
finite number of equations for $n^2+1$ (the rank of $K(I\pro(1,n))$) unknowns.
Furthermore, since the action of the translation group $J$, which is rank
$n$, respects $\cP_1$ by Proposition (\ref{prop:JActsOnK}), it is enough to
solve for only $n^2-n+1$ variables satisfying Equation (\ref{eq:psilineelt}) for all
$\ell\in\{0,\ldots,n-1\}$, as all other $\lam$-line elements will be their $J$
translates.  
\end{rem}

\begin{crl}\label{crl:CanonicalReps}
 Let $\cP_1$ be the semigroup of $\lam$-line elements of the virtual K-theory
  $(K(I\pro(1,n))_\nq,\star,1,\ovaug,\olam)$. Each $J_\nq$-orbit in
  $\cP_1$ contains a unique representative $\cl$ such that $\cl^{\star n} = 1$.
\end{crl}
\begin{proof}
Given $\cf$ in $\cP_1$, we have $\cf^{\star n} = \psit^n(\cf) = 1 + j$ for
some $j$ in $J_\nq$ by  Proposition (\ref{PsiPeriodic}). If $\cl = \cf -
\frac{j}{n}$, then by Equation~\eqref{eq:JInertialProduct} we have $\cl^{\star n} = (\cf - \frac{j}{n})^{\star n} =
\cf^{\star n} - j = 1+j-j = 1$.
\end{proof}

\subsubsection{The Virtual K-theory and 
virtual Chow ring of $\pro(1,2)$} 
We now study the virtual 
K-theory and virtual 
Chow theory 
(with either $\nq$ or $\nc$ coefficients)
of the weighted projective line $\pro(1,2):= [X/{\nc^\times}]$.
By \cite[Theorem 4.2.2]{EJK:12a}
they are isomorphic to the orbifold K-theory and
orbifold Chow theory, respectively, of 
the cotangent bundle $T^*\pro(1,2)$.  

\begin{rem}
For the remainder of this section, unless otherwise specified, all products
are the virtual 
products.
\end{rem}

Let $\olam:K(I\pro(1,2))_\nq\to K(I\pro(1,2))_\nq$ denote the induced virtual 
$\lam$-ring structure. In order to describe the group of $\lam$-line
elements $\cP_1$ of $(K(I\pro(1,2))_\nq,\cdot,1,\olam)$, it will be useful to
introduce the injective map $f:\nq^2\to K(I\pro(1,2))_\nq$ defined by
\begin{equation}
f(\alpha,\beta) := 
\alpha\Delta_0
 + \beta \Delta_1, 
\end{equation}
whose image 
is 
the translation group $J_\nq$ of $K(I\pro(1,2))_\nq$. 

Consider the following injective maps $\nq^2\to K(I\pro(1,2))_\nq$:
\begin{equation}
\rho_0(\alpha,\beta) := \chi_0^0 + f(\alpha,\beta),
\end{equation}
\begin{equation}
\rho_1(\alpha,\beta) := \chi_0^1 + f(\alpha,\beta),
\end{equation}
and
\begin{equation}
\rho_{\pm}(\alpha,\beta) := \frac{1}{2} (\chi_0^0+\chi_0^1 \pm \chi_{1}^0) + f(\alpha,\beta).
\end{equation}
\begin{prop} \label{prop.p12line}
The group of $\lam$-line elements $\cP_1$ of the virtual K-theory\\
$(K(I\pro(1,2))_\nq,\star,1,\olam)$ is the disjoint union of
the images of the four maps $\rho_0$, $\rho_1$, $\rho_{\pm}$, 
and the restriction of the inertial dual $\cP_1\to\cP_1$ agrees with the
operation of taking the inverse.
In particular,
$K(I\pro(1,2))_\nq$ is spanned as a $\nq$-vector space by $\cP_1$.
The multiplication in $\cP_1$ is 
given by the following equations:
\begin{equation}\label{eq:RhoMultOne}
\rho_0(\alpha,\beta)\rho_0(\alpha',\beta') = \rho_0(\alpha+\alpha',\beta+\beta')
\end{equation}
\begin{equation}
\rho_0(\alpha,\beta)\rho_1(\alpha',\beta') = \rho_1(\alpha+\alpha',\beta+\beta')
\end{equation}
\begin{equation}
\rho_0(\alpha,\beta)\rho_\pm(\alpha',\beta') =
\rho_\pm(\alpha+\alpha',\beta+\beta')
\end{equation}
\begin{equation}
\rho_1(\alpha,\beta)\rho_1(\alpha',\beta') = \rho_0(\alpha+\alpha'+1,\beta+\beta')
\end{equation}
\begin{equation}
\rho_1(\alpha,\beta)\rho_\pm(\alpha',\beta') =
\rho_\mp(\alpha+\alpha'+\frac{1}{2},\beta+\beta'\pm\frac{1}{2})
\end{equation}
\begin{equation}\label{eq:RhoMultSix}
\rho_\pm(\alpha,\beta)\rho_\pm(\alpha',\beta') =
\rho_0(\alpha+\alpha'+\frac{1}{2},\beta+\beta'\pm\frac{1}{2})
\end{equation}
\begin{equation}\label{eq:RhoMultSeven}
\rho_+(\alpha,\beta)\rho_-(\alpha',\beta') =
\rho_1(\alpha+\alpha',\beta+\beta').
\end{equation}
The inverses are given by the following equations:
\begin{equation}\label{eq:RhoMultEight}
\rho_0(\alpha,\beta)^{-1} = 
\rho_0(-\alpha,-\beta)
\end{equation}
\begin{equation}
\rho_1(\alpha,\beta)^{-1} = \rho_1(-(1+\alpha),-\beta)
\end{equation}
\begin{equation}\label{eq:RhoMultTen}
\rho_\pm(\alpha,\beta)^{-1} = \rho_\pm(-(\alpha+\frac{1}{2}),-\beta\mp\frac{1}{2}).
\end{equation}
\end{prop}
\begin{proof}

We first show that the set of line elements $\cP_1$ in the virtual K-theory $K:=K(I\pro(1,2))_\nq$ is the union of the images of the maps $\rho_0, \rho_1,\rho_\pm$. Since $\{ \chi_0^0,\chi_0^1, \chi_1^1, \Delta_0, \Delta_1 \}$ is a $\nq$-basis for $K$, it follows from Proposition \ref{prop:JActsOnK} that every element of $\cP_1$ can be uniquely written as $L + f(\alpha,\beta)$, where $L$ is an element in $\cP_1$ of the form
$
L = c_0^0 \chi_0^0 + c_0^1 \chi_0^1 + c_1^1 \chi_1^1
$
for some $c_0^0, c_0^1, c_1^1, \alpha,\beta$ in $\nq$.  We will now find all such elements $L$ in $\cP_1$. By Proposition \ref{prop:FiniteEqsForLineElts}, $L$ belongs to $\cP_1$ if and only if it is invertible, $\ovaug(L) = 1$ and $\psit^2(L) = L^2$. Using the definition of $\psit^2$, we obtain
\[
\psit^2(L) = c_0^0 \chi_0^0 + c_0^1 \chi_0^2 + c_1^1 (\chi_1^0+\chi_1^1),
\]
and the virtual multiplication yields
\begin{eqnarray*}
L^2 &=& (c_0^0 \chi_0^0 + c_0^1 \chi_0^1 + c_1^1 \chi_1^1)^2 = 
(c_0^0)^2 \chi_0^0 + (c_0^1)^2 \chi_0^2 + (c_1^1)^2 (\chi_1^1)^2 + 2 c_0^0 c_0^1 \chi_0^1 + 2 c_0^0 c_1^1 \chi_1^1 + 2 c_0^1 c_1^1 \chi_0^1\chi_1^1\\
&=& (c_0^0)^2 \chi_0^0 + (c_0^1)^2 \chi_0^2 + (c_1^1)^2 (\chi_0^0 -2\chi_0^1 + \chi_0^2) + 2 c_0^0 c_0^1 \chi_0^1 + 2 c_0^0 c_1^1 \chi_1^1 + 2 c_0^1 c_1^1 \chi_1^0 \\
&=& ((c_0^0)^2 + (c_1^1)^2 )\chi_0^0 + 2 (c_0^0 c_0^1 - (c_1^1)^2)\chi_0^1  + ((c_0^1)^2 + (c_1^1)^2)\chi_0^2 + 2 c_0^1 c_1^1 \chi_1^0 + 2 c_0^0 c_1^1 \chi_1^1,
\end{eqnarray*}
and $\psit^2(L) - L^2=0$ is equivalent to the following simultaneous equations:
\begin{equation*}
0 = c_0^0(1-c_0^0) - (c_1^1)^2 = -c_0^0 c_0^1 + (c_1^1)^2 = c_0^1(1- c_0^1) - (c_1^1)^2 = c_1^1(1-2 c_0^1) = c_1^1(1-2 c_0^0).
\end{equation*}
It follows that $\psit^2(L) = L^2$ if and only if $L = 0, \rho_0(0,0), \rho_1(0,0), \rho_\pm(0,0)$. However, the virtual augmentation $\ovaug(0) = 0$, while $\ovaug(\rho_0(0,0)) = \ovaug(\rho_1(0,0))=\ovaug(\rho_\pm(0,0)) = 1$.  Finally, $\rho_0(0,0)$, $\rho_1(0,0)$ are invertible, being classes of ordinary line bundles on the untwisted sector $\pro(1,2)$, while a calculation shows that $\rho_\pm(0,0)^{-1} = \rho_\pm(-\frac{1}{2},\mp\frac{1}{2})$.  

Therefore, by Proposition \ref{prop:JActsOnK}, $\cP_1$ is the union of images of the maps $\rho_0, \rho_1, \rho_\pm$.  It is easy to see that these images are disjoint. Furthermore, $K$ is spanned by $\cP_1$, since $\{ \rho_0(0,0), \rho_0(1,0), \rho_1(0,0), \rho_\pm(0,1)\} $ is a $\nq$-basis. Also, Equations (\ref{eq:RhoMultEight}) to (\ref{eq:RhoMultTen}) follow from Equations (\ref{eq:RhoMultOne}) to (\ref{eq:RhoMultSeven}).

We will now write out a detailed proof of Equation (\ref{eq:RhoMultSix}) to give the reader a feel for the calculation, noting that the proofs for Equations (\ref{eq:RhoMultOne}) to (\ref{eq:RhoMultSeven}) are similar. We first show that Equation (\ref{eq:RhoMultSix}) holds when $\alpha=\alpha'=\beta=\beta' = 0$ since
\begin{eqnarray*}
(\rho_\pm(0,0))^2
&=& \left(\frac{1}{2}(\chi_0^0+\chi_0^1\pm \chi_1^0)\right)^2\\
&=&\frac{1}{4}\left( (\chi_0^0)^2 + (\chi_0^1)^2 + (\chi_1^0)^2 + 2 \chi_0^0\chi_0^1 \pm 2 \chi_0^0\chi_1^0\pm 2\chi_0^1\chi_1^0\right) \\
&=& \frac{1}{4}\left( \chi_0^0 + \chi_0^2 + (\chi_0^0-2\chi_0^{-1}+\chi_0^{-2})+ 2\chi_0^1\pm 2\chi_1^0\pm 2\chi_1^1\right) \\
&=& \frac{1}{4}\left( \chi_0^0 + \chi_0^2 + (\chi_0^0+\chi_0^2 - 2 \chi_0^1)+ 2\chi_0^1\pm 2\chi_1^0\pm 2\chi_1^1\right) \\
&=& \frac{1}{2}\left( \chi_0^0 + \chi_0^2 \pm (\chi_1^0+\chi_1^1)\right)
= \chi_0^0 + \frac{1}{2}\Delta_0 \pm \frac{1}{2}\Delta_1 = \rho_0(\frac{1}{2},\pm\frac{1}{2}),
\end{eqnarray*}
where the third equality follows from Theorem \ref{thm.P1nvirtualprod} while the fourth is from the relations 
\begin{equation}\label{eq:InvChi}
\chi_0^{-1} = \chi_0^0 + \chi_0^1 - \chi_0^2 \qquad \mathrm{and}\qquad
\chi_0^{-2} = 2\chi_0^0 - \chi_0^2.
\end{equation}

Now, Equation (\ref{eq:RhoMultSix}) follows for all $\alpha$, $\beta$, $\alpha'$, and $\beta'$ since
\begin{eqnarray*}
\rho_\pm(\alpha,\beta) \rho_\pm(\alpha',\beta') &=& (\rho_\pm(0,0)+f(\alpha,\beta)) (\rho_\pm(0,0)+f(\alpha',\beta')) \\ &=& \rho_\pm(0,0)\rho_\pm(0,0) + (f(\alpha,\beta)+f(\alpha',\beta'))\rho_\pm(0,0) + f(\alpha,\beta) f(\alpha',\beta') \\
&=& \rho_\pm(\frac{1}{2},\pm\frac{1}{2}) + f(\alpha+\alpha',\beta+\beta')\rho_\pm(0,0)\\
&=&\rho_\pm(\frac{1}{2},\pm\frac{1}{2}) + f(\alpha+\alpha',\beta+\beta')\ovaug(\rho_\pm(0,0))
\\
&=& \rho_\pm(\frac{1}{2},\pm\frac{1}{2}) + f(\alpha+\alpha',\beta+\beta')
=  \rho_\pm(\alpha+\alpha'+\frac{1}{2},\beta+\beta'\pm\frac{1}{2}).\\
\end{eqnarray*}
Here, the third equality follows from the fact that $J^2 = 0$ in Lemma \ref{lm:JJAndJK}(2), from Equation (\ref{eq:RhoMultSix}) when $\alpha=\beta = \alpha'=\beta'= 0$, and from the definition of $f$. The fourth equality is from Equation (\ref{eq:JInertialProduct}), the fifth is from Proposition \ref{prop:FiniteEqsForLineElts}, and the sixth is from the definition of $\rho_\pm$. This finishes the proof of Equation (\ref{eq:RhoMultSix}).

Finally, we write details of the proof that  $\rho_0(\alpha,\beta)^\dagger = \rho_0^{-1}(\alpha,\beta)$. The proof of the analogous statements for $\rho_1(\alpha,\beta), \rho_\pm(\alpha,\beta)$ and, hence, for all elements in $\cP_1$ is similar. The definition of the inertial dual together with the fact that $\cs_0 = 0$ and $\cs_1 = \chi_1^1$ yields the following identities for all $a,b\in\nz$:
\begin{equation}\label{eq:ChiDaggers}
(\chi_0^a)^\dagger = \chi_0^{-a} \qquad \mathrm{and} \qquad (\chi_1^b)^\dagger = - \chi_1^{-b-1}.
\end{equation}
It follows that
\begin{eqnarray*}
\rho_0(\alpha,\beta)^\dagger 
&=& (\chi_0^0)^\dagger + \alpha (\Delta_0)^\dagger + \beta(\Delta_1)^\dagger 
= (\chi_0^0)^\dagger + \alpha ((\chi_0^2)^\dagger - (\chi_0^0)^\dagger) + \beta ((\chi_1^0)^\dagger + (\chi_1^1)^\dagger )\\
&=& \chi_0^0 + \alpha(\chi_0^{-2}-\chi_0^0) - \beta(\chi_1^{-1}+\chi_1^{-2}) 
= \chi_0^0 + \alpha((2\chi_0^0 - \chi_0^2) - \chi_0^0) -\beta(\chi_1^0+\chi_1^1) \\ &=& \chi_0^0 -\alpha\Delta_0 - \beta\Delta_1 
= \rho_0(-\alpha,-\beta) = \rho_0(\alpha,\beta)^{-1},
\end{eqnarray*}
where the third equality follows from Equation (\ref{eq:ChiDaggers}), the fourth from Equation (\ref{eq:InvChi}), and the last from Equation (\ref{eq:RhoMultEight}).
\end{proof}

A direct calculation yields the following.
\begin{prop}\label{eq:FirstChernOfLines}The 
inertial first 
Chern class 
for virtual K-theory
is a homomorphism of
  groups $\ovc^1:\cP_1\to A^\bra{1}(I\pro(1,2))_\nq$, where
\[
\ovc^1(\rho_0(\alpha,\beta)) = 2 \alpha c_0^1 + 2 \beta c_{1}^0,
\]
\[
\ovc^1(\rho_1(\alpha,\beta)) = (2 \alpha +1) c_0^1 + 2 \beta c_{1}^0,
\]
\[
\ovc^1(\rho_\pm(\alpha,\beta)) = (2 \alpha+\frac{1}{2}) c_0^1 + (2 \beta \pm
\frac{1}{2}) c_{1}^0.
\]
\end{prop}

The virtual 
K-theory ring has a simple form in
terms of these $\lam$-line elements.
\begin{prop} 
\label{prop:POneTwoVirtualK}
Let $(K(I\pro(1,2))_\nq,\star,1 := \chi_0^0)$ be the virtual
K-theory ring. We have two isomorphisms of $\nq$-algebras (and $\psi$-rings)
\begin{equation}\label{eq:VirtualKIsom}
\Phi_\pm:  \frac{\nq[\sigma,\tau]}{\langle 
  (\tau-1)(\tau^2-1),(\sigma-1)(\sigma^2-1),(\sigma-\tau)(\tau-1)
    \rangle}\to K(I\pro(1,2))_\nq,
\end{equation}
where 
$\Phi_\pm(\sigma) := \rho_1(0,0) = \chi_0^1$, and $\Phi_\pm(\tau) := \rho_{\pm}(0,0) =\frac{1}{2} (\chi_0^0 +\chi_0^1 \pm \chi_{1}^0)$.
Here, the $\psi$-ring structure of the domain of $\Phi_\pm$ is given by $\psi^\ell(\sig^{\pm 1}) = \sig^{\pm\ell}$ and $\psi^\ell(\tau^{\pm 1}) = \tau^{\pm \ell}$ for all $\ell\geq 1$.
Similarly, we have two isomorphisms of graded $\nq$-algebras
\begin{equation}
\Psi_\pm:  \frac{\nq[\mu,\nu]}{\langle \mu, \nu \rangle^2}\to A^*(I\pro(1,2))_\nq,
\end{equation}
where $\mu,\nu$ in $A^{\bra{1}}(I\pro(1,2))_\nq$ and $\Psi_\pm(\nu) :=
\ovc^1(\rho_{\pm}(0,0)) =\frac{1}{2} (c_0^1\pm c_{1}^0)$ 
and $\Psi_\pm(\mu) := \ovc^1(\rho_1(0,0)) = c_0^1$.  Under the
identifications $\Phi_\pm$ and $\Psi_\pm$, the inertial Chern character
$\ovCh:K(I\pro(1,2))\to A^*(I\pro(1,2))_\nq$ corresponds to the map
$\sigma\mapsto\exp(\mu) = 1 + \mu$ and
$\tau\mapsto\exp(\nu) = 1 + \nu$.

\end{prop}
\begin{proof}
Since 
$(\chi_0^1)^2 = \chi_0^2$ and $\chi_0^0 = 1$
and 
$
\rho_{\pm}(0,0)^2 = \frac{1}{2}((\chi_0^0+\chi_0^2) \pm (\chi_{1}^0 + \chi_{1}^1)),
$
the set $\{\,\chi_0^0, \chi_0^1, \chi_0^2,\rho_+(0,0),\rho_+(0,0)^2\}$ 
is a basis for the $\nq$-vector space $K(I\pro(1,2))_\nq$. Thus, $K(I\pro(1,2))_\nq$ is
generated as a 
$\nq$-algebra
by 
$\chi_0^1$
and $\rho_+(0,0)$. A calculation 
shows that the following three polynomials are zero:
\[
(\chi_0^1-1)((\chi_0^1)^2-1) = (\rho_+(0,0)-1)(\rho_+(0,0)^2-1) =
  (\chi_0^1-\rho_+(0,0))(\rho_+(0,0)-1) = 0.
\]
 A dimension count shows that these are the only relations. Therefore,
$\Phi_+$ is an isomorphism of $\nq$-algebras. 
The previous analysis holds verbatim if $\rho_+(0,0)$ is replaced by
$\rho_-(0,0)$ everywhere.

A similar analysis holds for the Chow theory.
\end{proof}

\begin{rem}
The 
presentation in the previous proposition yields an
exotic integral structure in virtual
K-theory and Chow theory as we now explain.

Consider the subring $\oKZ(I\pro(1,2))$ (not sub-$\nq$-algebra) of $K(I\pro(1,2))_\nq$ generated by
$\{\,\rho_1(0,0),\rho_+(0,0)\,\}$.  Under the isomorphism $\Phi_+$ in Proposition \ref{prop:POneTwoVirtualK}, the ring $\oKZ(I\pro(1,2))$ is isomorphic to
\[
\frac{\nz[\sigma,\tau]}{\langle 
  (\tau-1)(\tau^2-1),(\sigma-1)(\sigma^2-1),(\sigma-\tau)(\tau-1)
    \rangle}
\]
under the identification $\sigma = \rho_1(0,0)$ and $\tau=\rho_+(0,0)$.  

We will now show that the group of $\lam$-line elements of $\oKZ(I\pro(1,2))$, $\oPZ_1$, is equal to $\cP_1\cap \oKZ(I\pro(1,2))$. To see this, notice that since $\Delta_0 = \sigma^2 - 1$ and $\Delta_1 = 2 \tau^2 - \sigma^2 - 1$, 
\[
f(\alpha,\beta) = 2\beta\tau^2 + (\alpha-\beta)\sigma^2 - (\alpha+\beta). 
\]
Hence,  $\rho_s(\alpha,\beta)$ belongs to $\oKZ(I\pro(1,2))$ if and only if $(\alpha,\beta)$ belongs to  
\[
D := \{ (p+\frac{q}{2},\frac{q}{2}) \, | \, p,q\in\nz \}, 
\]
where $s=0,1,\pm$ , noting that $\rho_-(0,0) = \sigma\tau^{-1}$.  Thus, by Proposition \ref{prop:POneTwoVirtualK},
\[
\cP_1\cap\oKZ(I\pro(1,2)) = \rho_0(D)\cup\rho_1(D)\cup\rho_+(D)\cup\rho_-(D),
\]
but Equations (\ref{eq:RhoMultEight}) to (\ref{eq:RhoMultTen}) imply that $\cP_1\cap\oKZ(I\pro(1,2))$ is closed under inversion. It follows that $\oPZ_1 = \cP_1\cap\oKZ(I\pro(1,2))$.

We will now show that $\oPZ_1$ is the subgroup generated by $\sigma$ and $\tau$.  Notice that since $\sigma^2 = \rho_0(1,0)$ and $\tau^2 = \rho_0(\frac{1}{2},\frac{1}{2})$, $\sigma^{2k} \tau^{2\ell} = \rho_0(k+\frac{\ell}{2},\frac{\ell}{2})$ belongs to $\langle \sigma, \tau \rangle$ for all $k,\ell\in\nz$, i.e., $\rho_0(D)\subseteq \langle \sigma,\tau\rangle$. Similarly, $\rho_1(0,0)\rho_0(D) = \rho_1(D)$, $\rho_+(0,0)\rho_0(D)= \rho_+(D)$, and $\rho_-(0,0)\rho_0(D)= \rho_-(D)$ are all subsets of $\langle\sigma,\tau\rangle$. It follows that $\langle\sigma,\tau\rangle =\oPZ_1$.

Consider the subring $\oAZ^*(I\pro(1,2)) := \ovCh(\oKZ(I\pro(1,2)))$ of the virtual Chow ring of $A^*(I\pro(1,2))_\nq$. From this we obtain (see Proposition \ref{eq:FirstChernOfLines})
\[
\oAZ^{\{0\}}(I\pro(1,2)) = \nz c_0^0 \qquad \mathrm{and}\qquad \oAZ^{\{1\}}(I\pro(1,2)) = \{ v c_0^1 + w c_1^0 \,|\,(v,w)\in D\}.
\]

It follows that the first virtual Chern class $\ovc^1:\cPZ_1\to\oAZ^\bra{1}(I\pro(1,2))$  is a group isomorphism by Proposition \ref{eq:FirstChernOfLines},  since for all $p,q$ in $\nz$,
\[
\ovc^1(\sigma^p\tau^q) = p \ovc^1(\sigma) + q\ovc^1(\tau) = (p+\frac{q}{2})c_0^1 + \frac{q}{2} c_1^0.
\]

\end{rem}

\subsubsection{The virtual K-theory and virtual Chow ring of $\pro(1,3)$} 
We now study the virtual K-theory and virtual Chow ring of
$\pro(1,3)$. Unlike the case of $\pro(1,2)$, the 
formula of \cite[Theorem 4.2.2]{EJK:12a} implies that
the rational virtual K-theory and rational virtual
Chow rings of $\pro(1,3)$ differ from the orbifold K-theory and the orbifold
Chow rings of the cotangent bundle $T^*\pro(1,3)$, respectively. 
Indeed
the formula of \cite[Definition 4.0.11]{EJK:12a} shows that
the class $\cs^+T^*\Pro(1,3)$ is not integral, so the inertial pair 
from the orbifold theory of $T^*\pro(1,3)$ is Gorenstein but not strongly Gorenstein. 
We will now describe the $\lambda$-positive elements of virtual K-theory of $\pro(1,3)$. Unlike the case of $\pro(1,2)$, we need to work with
$\nc$-coefficients, so that the set of $\lambda$-line elements generate the entire virtual
K-theory group.  

\begin{rem}
For the remainder of this section, unless otherwise specified, all products
are the virtual 
products.
\end{rem}

\begin{prop} \label{prop:POneThreeVirtualK}
 Let $(K(I\pro(1,3))_\nc,\star,1 :=
  \chi_0^0,\psit)$ be the virtual K-theory ring with its virtual $\lam$-ring
  structure. The set of its $\lambda$-line elements $\cP_1$ spans the
  $\nc$-vector space $K(I\pro(1,3))_\nc$.
The restriction of the inertial dual $\cP_1\to\cP_1$ agrees with the
operation of taking the inverse. The space 
$\cP_1$ consists of
  $27$ orbits of the action of the translation group $J_\nc$, where each orbit has a unique
  representative\footnote{This representative need not be the same as the one defined
  in Corollary (\ref{crl:CanonicalReps}).}
 in the set \[\{\Sigma_i\}_{i=1}^3 \sqcup \coprod_{\substack{i=1,2,3\\ j = 1,2}}
  \cD_{i,j} \sqcup \coprod_{\substack{i=1,\dots ,6\\ k = 0,1,2}}
  \cT_{i,k}\] given by the
  following (where $\zeta_3 = \exp(2\pi i/3)$), $j\in \{1,2\}$ and $k\in\{0,1,2\}$:
\[ \Sigma_{1}=\chi_0^{0}, \quad \Sigma_{2}=\chi_0^{1}, \quad \Sigma_{3}=\chi_0^2, \]
\[ 
\cD_{1,j} =\frac{1}{3} \chi_0^0 + \frac{1}{3} \chi_0^1 +\frac{1}{3} \chi_0^2
-\frac{1}{3} \zeta_3^j \chi_1^0 + \frac{1}{3} \chi_1^1 -\frac{1}{3} \zeta_3^{2j}
\chi_2^0 + \frac{1}{3} \chi_2^1, 
\]
\[
\cD_{2,j} = \frac{1}{3} \chi_0^0 + \frac{1}{3} \chi_0^1 + \frac{1}{3} \chi_0^2
-\frac{1}{3} \chi_1^0 + \frac{1}{3} \zeta_3^j \chi_1^1 -\frac{1}{3} \chi_2^0 +
\frac{1}{3} \zeta_3^{2j} \chi_2^1, 
\]
\[
\cD_{3,j} = \frac{1}{3} \chi_0^0 + \frac{1}{3} \chi_0^1 + \frac{1}{3} \chi_0^2 -\frac{1}{3} \zeta_3^{2j} \chi_1^0 + \frac{1}{3} \zeta_3^j \chi_1^1 -\frac{1}{3} \zeta_3^j \chi_2^0 + \frac{1}{3} \zeta_3^{2j} \chi_2^1,
\]
\[
\cT_{1,k} = \frac{1}{3} \chi_0^0 + \frac{2}{3} \chi_0^2 + \frac{1}{3} \zeta_3^k \chi_1^0 + \frac{1}{3} \zeta_3^{2k} \chi_2^0, 
\]

\[
\cT_{2,k} =  \frac{2}{3} \chi_0^0 + \frac{1}{3} \chi_0^2 -\frac{1}{3} \zeta_3^k \chi_1^0 -\frac{1}{3} \zeta_3^{2k} \chi_2^0, 
\]

\[
\cT_{3,k} = \frac{2}{3} \chi_0^0 + \frac{1}{3} \chi_0^1 + \frac{1}{3} \zeta_3^k \chi_1^1 + \frac{1}{3} \zeta_3^{2k} \chi_2^1,
\]
\[
\cT_{4,k} = \frac{1}{3} \chi_0^0 + \frac{2}{3} \chi_0^1 -\frac{1}{3} \zeta_3^k \chi_1^1 -\frac{1}{3} \zeta_3^{2k} \chi_2^1,
\]
\[
\cT_{5,k} = \frac{1}{3} \chi_0^1 + \frac{2}{3} \chi_0^2 + \frac{1}{3} \zeta_3^k \chi_1^0 + \frac{1}{3} \zeta_3^k \chi_1^1 + \frac{1}{3} \zeta_3^{2k} \chi_2^0 + \frac{1}{3} \zeta_3^{2k} \chi_2^1,
\]
\[
\cT_{6,k} = \frac{2}{3} \chi_0^1 + \frac{1}{3} \chi_0^2 -\frac{1}{3} \zeta_3^k \chi_1^0 -\frac{1}{3} \zeta_3^k \chi_1^1 -\frac{1}{3} \zeta_3^{2k} \chi_2^0 -\frac{1}{3} \zeta_3^{2k} \chi_2^1.
\]
\end{prop}

\begin{proof}
The $\lam$-line elements in $\cP_1$ are calculated by applying the algorithm in Remark (\ref{rem:LineEltCalc}) and by showing that these $\lam$-line elements are invertible. The fact that the elements of $\cP_1$ span $K(\pro(1,3))_\nc$ is also a calculation. We omit the details to all of these calculations which are straightforward but lengthy.
\end{proof}

\begin{prop}
Let $K(I\pro(1,3))_\nc$ be the virtual K-theory with its virtual
$\lambda$-ring structure. We have an isomorphism of $\nc$-algebras
$\Psi:\nc[\sigma^{\pm 1},\tau^{\pm 1},\taub^{\pm 1}]/\I \to
K(I\pro(1,3))_\nc$, where 
$\Psi(\sigma) = \Sigma_2$, 
$\Psi(\tau) = \cT_{1,1}$,
and $\Psi(\taub)=\cT_{1,2}$, where the ideal $\I$ is generated by the
following ten relations: 
\[\rl_1 := \sigma ^3-2 \sigma ^2+\sigma -\tau ^2+\tau  \taub +\tau -\taub ^2+\taub -1,\]
\[\rl_2 :=(\tau -1) \left(\tau ^2-\sigma\right),\qquad \rlb_2:=(\taub -1) \left(\taub ^2-\sigma\right),\] 
\[\rl_3:=(\tau -1) \left(\sigma ^2-\tau \right),\qquad\,\rlb_3:=(\taub -1) \left(\sigma ^2-\taub \right),\]
\[\rl_4 := \sigma ^2-\sigma  \tau -\sigma  \taub +\tau ^2 \taub -\tau  \taub +\taub ^2-\taub +1,\]
\[\rlb_4 := \sigma ^2-\sigma  \tau -\sigma  \taub +\tau ^2+\tau  \taub ^2-\tau  \taub -\tau +1,\]
\[\rl_5 := (\tau -1) (\sigma  \tau -1),\qquad\rlb_5 := (\taub -1) (\sigma  \taub -1),\]
\[\rl_6 :=-\sigma ^2+\sigma  \tau  \taub +\sigma -\tau ^2+\tau  \taub -\taub ^2.\]
It follows that $(\sigma-1)(\sigma^3-1)$ belongs to $\I$, which is the relation on the untwisted sector.
Furthermore, every element $K(I\pro(1,3))_\nc$ can be uniquely presented as a polynomial $\{\,\sigma,\tau,\taub\,\}$ of degree less than or equal to $2$. In particular, we have 
\[
\sigma^{-1} =-\sigma ^2+\sigma -\tau ^2+\tau  \taub +\tau -\taub ^2+\taub,
\]
\[
\tau^{-1} =- \sigma \tau +\sigma +1, \dsand \taub^{-1} = -\sigma\taub+\sigma+1.
\]
\end{prop}
\begin{proof}
$K(I\pro(1,3))_\nc$ is a ten-dimensional $\nc$-vector space. A 
calculation shows that the set of all monomials in $\{\sigma,\tau,\taub\,\}$
of degree less than or equal to $2$ is a basis of this vector space. The ten
relations correspond to the ten cubic monomials in
$\{\,\sigma,\tau,\taub\,\}$. The expression for the inverses can be verified
by computation.
We omit the details of these straightforward but lengthy calculations.
\end{proof}

\begin{rem} 
Restricting $\Psi$ to $\nz[\sigma^{\pm 1},\tau^{\pm 1},\taub^{\pm 1}]/\I$ yields
an exotic integral structure on the virtual K-theory $K(I\pro(1,3))_\nc$. The inertial
Chern character 
homomorphism 
$\ovCh:K(I\pro(1,3))_\nc\to
A^*(I\pro(1,3))_\nc$ induces an exotic integral structure on virtual Chow theory.
\end{rem}
\subsection{The resolution of singularities of ${\mathbb T}^*\pro(1,n)$
and the HKRC}
We now connect the virtual $\lambda$-ring to the usual $\lambda$-ring structure on a crepant resolution of singularities of the coarse moduli space of the cotangent bundle stack ${\mathbb T}^*\pro(1,n)$.
\begin{prop} \label{prop.cotanp1n}
The cotangent bundle ${\mathbb T}^*\pro(1,n)$ of $\pro(1,n)$ is the quotient stack $[(X
\times \A^1)/\nc^\times]$, where $\nc^\times$ acts with weights $(1,n,-(n+1))$. 
\end{prop}
\begin{proof}
Since $\dim \pro(1,n) = 1$ the cotangent bundle stack is a line bundle. 
Consider the quotient map $\pi \colon X^0 \to \pro(1,n)=[X^0/\nc^\times]$. We begin by determining $\pi^*{\mathbb T}^*\pro(1,n)$ as an $\nc^\times$-equivariant bundle $L$ on $X^0$. Once we do this, we can identify ${\mathbb T}^*\pro(1,n)$
with the quotient stack $[L/\nc^\times]$.

The restriction map $\Pic_{\nc^\times} (\nc^2) \to \Pic_{\nc^\times}(X^0) = \Pic(\pro(1,n))$ is surjective, so any $\nc^\times$-equivariant line bundle on $X^0$
is determined by a character $\xi$ of $\nc^\times$, so $L = X^0 \times \A^1$
and $\nc^\times$ acts on $L$ by $\lambda(a,b,v) = (\lambda a, \lambda^n b, \xi(\lambda)v)$.

To find the character $\xi$, note that for any algebraic group
$G$ and any $G$-torsor $\pi \colon P \to X$, there is an exact sequence of $G$-equivariant vector bundles on $P$
$$0 \rTo P \times \Lie(G) \rTo TP \rTo \pi^*TX \to 0,$$ 
where $TP$ is the tangent bundle to $P$ \cite[Lemma A.1]{EdGr:05}. Applying this fact to the $\nc^\times$ torsor $\pi \colon X^0 \to \pro(1,n)$, we obtain an exact sequence of vector bundles
$$X^0 \times \nc \rTo TX^0 \rTo \pi^*{\mathbb T}\pro(1,n).$$
The action of $\nc^\times$ is as follows: Since $\nc^\times$ is Abelian,
the Lie algebra is the trivial representation, 
while $TX^0 = X^0 \times \nc^2$, where $\nc^\times$ acts on the $\nc^2$ factor
with weights $(1,n)$. 
Taking the determinant of this sequence 
shows 
$\pi^*{\mathbb T}\pro(1,n)$ is the $\nc^\times$-equivariant line bundle
$X^0 \times \nc$, where $\nc^\times$ acts on $\nc$ with weight $(n+1)$.
Hence, $\pi^*{\mathbb T^*}\pro(1,n)$ is the $\nc^\times$-equivariant 
bundle $X^0 \times \nc$, where $\nc^\times$ acts on $\nc$ with weight $-(n+1)$.
\end{proof}

By Proposition \ref{prop.cotanp1n},
the coarse moduli space of ${\mathbb T}^*\pro(1,n)$ is the geometric quotient $\left((\nc^2 \smallsetminus \{0\}) \times \nc\right)/\nc^\times$, where $\nc^\times$ acts by $\lambda(a,b,v)= (\lambda a, \lambda^n b, \lambda^{-n-1}v)$. By the Cox construction \cite[Section 5.1]{CLS:11},
this
quotient is the toric surface associated to the 
simplicial fan $\Sigma_n$ with two maximal cones $\sigma_{n+1,n-1}$
and $\sigma_{n,n+1}$. The cone $\sigma_{n+1,n-1}$  has rays
$\rho_{n-1}$ generated by $(-n,n+1)$ and $\rho_{n+1}$ generated by $(0,1)$.
The cone $\sigma_{n,n+1}$ has rays $\rho_{n+1}$ and $\rho_{n}$
spanned by $(1,0)$.
The fan is as follows.
$$
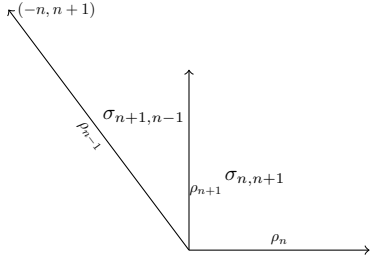
\begin{tikzpicture}[scale=0.8]
\draw [->] (0,0)--(3,0) node[pos=.5,sloped,above,scale=0.6]  {$\rho_n$};
\draw[->] (0,0) -- (0,3);
\node[scale=0.6] at (.28,1) {$\rho_{n+1}$};
\draw[->] (0,0)--(-3, 4)  node[pos=.5,sloped,below,scale=0.6]  {$\rho_{n-1}$}; 
\node[scale = 0.6] at (-2.2,4) {$(-n, n+1)$};
\node[scale = 0.8] at (-.75,2.2) {$\sigma_{ n+1,n-1}$};
\node[scale = 0.8] at (1.1,1.2) {$\sigma_{ n,n+1}$};
\end{tikzpicture}
$$
The cone $\sigma_{n+1,n-1}$ has multiplicity $n+1$ and by the method of Hirzebruch-Jung continued fractions \cite[Section 10.2]{CLS:11}, 
the nonsingular
toric surface determined by the fan $\Sigma'_n$, where $\sigma_{n-1,n+1}$ is subdivided along the rays $\rho_0, \rho_1, \ldots, \rho_{n-2}$ 
where $\rho_i$ is generated by $(-(i+1), i+2)$, is a toric resolution of singularities of $X(\Sigma_n)$. 
$$\begin{tikzpicture}[scale=0.9]
\draw [->] (0,0)--(3,0) node[pos=.5,sloped,above,scale=0.6]  {$\rho_n$};
\draw[->] (0,0) -- (0,3);
\node[scale=0.6] at (.28,1) {$\rho_{n+1}$};
\draw[->] (0,0)--(-3, 4)  node[pos=.5,sloped,below,scale=0.6]  {$\rho_{n-1}$}; 
\draw[->] (0,0)--(-1.8, 3.6)  node[pos=.5,sloped,below,scale=0.6]  {$\rho_{n-2}$}; 
\draw[->] (0,0)--(-0.5, 3.2)  node[pos=.5,sloped,below,scale=0.6]  {$\rho_{0}$}; 
\node[scale = 0.6] at (-2.4,4) {$(-n, n+1)$};
\node[scale = 0.6] at (-1.2,3.6) {$(-n+1,n)$};
\node[scale = 0.6] at (-0.1,3.2) {$(-1,2)$};
\node at (-.7,2) {$\ldots$};
\node[scale = 0.8] at (1.1,1.2) {$\sigma_{ n,n+1}$};
\end{tikzpicture}
$$
By \cite[Exercise 8.2.13]{CLS:11}, $X(\Sigma_n)$ is Gorenstein, so by \cite[Proposition 11.28]{CLS:11} the resolution of singularities $X(\Sigma'_n) \to X(\Sigma_n)$
is crepant. 

By the Cox construction, we can realize
the smooth toric variety $X(\Sigma'_n)$ as the the quotient of
$\A^{n+2}\smallsetminus 
Z(\Sigma'_n)$ 
with coordinates $(x_0, \ldots x_{n+1})$
by the free action of $(\nc^\times)^n$ with weights
$$\left(\chi_0, \ldots , \chi_{n-1},\chi_0\chi_1^2
\ldots \chi_{n-1}^n,\chi_0^{-2} \chi_1^{-3} \ldots
\chi_{n-1}^{-(n+1)}\right),$$ 
where $\chi_i$ is the character of $(\nc^\times)^n$
corresponding to the $i$th standard basis vector of ${\mathbb Z}^n$
and 
$Z(\Sigma'_n)
= V(x_2x_3\ldots x_{n+1}, x_0x_3\ldots x_{n+1}, x_0x_1x_4\ldots x_{n+1}, \ldots , 
x_0 \ldots x_{n-3} x_n x_{n+1}, \linebreak[4] x_0 \ldots  x_{n-1},x_1x_2 \ldots x_n)$.
\begin{prop} 
The following isomorphisms hold
where $t_i = c_1(\chi_i)$:
\[
K(X(\Sigma_n')) = 
\frac{\Z[\chi_0,\chi_0^{-1},  \ldots , \chi_{n-1}, \chi_{n-1}^{-1}]}{\left
<\euler(\chi_0), \ldots \euler(\chi_{n-1})\right>^2}
\]
and
\[
A^*(X(\Sigma_n'))= 
\frac{\Z[t_0,t_1,\ldots t_{n-1}]}{\left<t_0,t_1,\ldots t_{n-1}\right>^2}.
\]
\end{prop}
\begin{proof}
  The action of the torus is free, so $K(X(\Sigma_n')) =
  K_{(\nc^\times)^n}(\nc^{n+2} \smallsetminus Z(\Sigma_n'))$ and
  $A^*(X(\Sigma_n')) = A^*_{(\nc^\times)^n}(\nc^{n+2} \smallsetminus
  Z(\Sigma_n'))$. As in the proof of Proposition 
  \ref{prop.KP1ncalculation}, the localization exact sequence in
  equivariant K-theory implies that $K_{(\nc^\times)^n}(\nc^{n+2}
  \smallsetminus Z(\Sigma_n'))$ is a quotient of 
$R((\nc^\times)^n)
= \Z[\chi_0,\chi_0^{-1}, \ldots , \chi_{n-1},
\chi_{n-1}^{-1}]$.  Because $Z(\Sigma_n')$ is the union of
intersecting linear subspaces, we use an inductive argument to establish the
relations.
The ideal 
$I= \langle x_2x_3\ldots x_{n+1}, x_0x_3\ldots x_{n+1},
x_0x_1x_4\ldots x_{n+1}, \ldots , x_0 \ldots x_{n-3} x_n x_{n+1},
\linebreak[4] 
x_0 \ldots  x_{n-1},x_1x_2 \ldots x_n\rangle$
has a primary decomposition as the intersection of the ideals of linear spaces
$\langle x_i, x_j\rangle $, where 
$i\in\{0, \ldots n-1\}$ and, for each $i$,
$i+2 \leq j \leq n+1$.
Thus $Z(\Sigma_n')$ is the union of the linear subspaces $L_{i,j}$,
where $L_{i,j}= Z(x_i,x_j)$.
Order
the pairs $(i,j)$ lexicographically and set $U_{i,j}
= \nc^2 \smallsetminus (\cup_{(k,l) \leq (i,j)} L_{k,l})$, so
that $\nc^{n+2} \smallsetminus Z(\Sigma_n') = U_{n-1,n+1}$. If
$j < n+1$, we have a localization sequence
$$K_{(\nc^\times)^n}(L_{i,j+1} \smallsetminus (\cup_{(j,k) < (i,j+1)}L_{j,k})) 
\rTo K_{(\nc^\times)^n}(U_{i,j+1}) \rTo K_{(\nc^\times)^n}(U_{i,j}) \rTo 0.$$
The same self-intersection argument used in the proof of Proposition
\ref{prop.KP1ncalculation} shows that 
$K_{(\nc^\times)^{n}}(U_{i,j}) =
K_{(\nc^\times)^n}(U_{i,j+1})/\langle \{\euler(N_{i,j+1})\}\rangle$,  
where
$N_{i,j+1}$ is the normal bundle to $L_{i,j+1}$ in $\nc^{n+2}$. Similarly, 
$K_{(\nc^\times)^n}(U_{i+1,i+2}) = K_{(\nc^\times)^n}(U_{i,i+2})/\langle \euler(L_{i+1,i+2})\rangle$. Hence, by induction we have that 
$$K_{(\nc^\times)^n}(U_{n-1,n+1}) = \Z[\chi_0, \chi_0^{-1}, \ldots , \chi_{n-1}, \chi_{n-1}^{-1}]/\langle \{ \euler(N_{i,j})\}\rangle.$$
The $K$-theoretic Euler class of the bundle $N_{i,j}$ can be read off from the weights of the $(\nc^\times)^n$ action. When $j < n$, $\euler(N_{i,j}) = (1 -\chi_i^{-1})(1 - \chi_j^{-1})$, while $\euler(N_{i,n}) = (1-\chi_{i}^{-1})(1- (\chi_0\chi_1^2 \ldots \chi_{n-1}^n)^{-1})$, and $\euler(N_{i,n+1}) = (1-\chi_i^{-1})(1 -\chi_0^2\chi_1^3 \ldots \chi_{n-1}^{n+1})$.

We wish to show that the ideal ${\mathfrak b}$ generated by these Euler classes
is the same as the ideal ${\mathfrak a} = \langle \euler(\chi_0), \ldots
, \euler(\chi_{n-1})\rangle^2$. If we set $e_i =\euler(\chi_i) = (1 - \chi^{-1}_i)$, then ${\mathfrak a} = \langle \{ e_ie_j\}_{0 \leq i\leq j \leq n-1} \rangle$. 
Note that the ideal $\langle e_1, \ldots , e_n \rangle$ is the ideal
of Laurent polynomials in $\chi_0, \ldots , \chi_n$ that vanish at
$(1,1,\ldots , 1)$. 
If $j < n$, then $\euler(N_{i,j}) = e_ie_j \in {\mathfrak a}$. Also note
that since the expression $(1-(\chi_0\chi_1^2 \ldots \chi_{n-1}^n)^{-1})$
vanishes when each $\chi_i$ is set to 1, it must be in the ideal generated
by $e_1, \ldots e_n$, so
$\euler(N_{i,n}) = (1- \chi_i^{-1})(1-(\chi_0\chi_1^2 \ldots \chi_{n-1}^n)^{-1})
\in \langle e_0, \ldots, e_n \rangle^2 = {\mathfrak a}$. Similarly,
$\euler(N_{i,n+1}) \in {\mathfrak a}$. 

If $i < n-1$ and $j \geq i+1$, then the generators $e_i e_j$ are the Euler classes
of the bundles $N_{i,j}$. The remaining generators of ${\mathfrak a}$
are of the form $e_i^2$ and $e_ie_{i+1}$. 
Since the $\chi_i$ are units, the fact that $e_ie_j$ is in ${\mathfrak b}$
implies that for all $k > 0$ and $|i -j| \geq 2$, 
all expressions of the form $e_i(1 -\chi_j^{-k})$
and $(1-\chi_{i}^{-1})(\chi_j^k -1)$ are in ${\mathfrak b}$. We can then perform repeated eliminations with the expression for $\euler(N_{i,n})$
to show that for any $i$, $e_i(1-(\chi_i^{-(i+1)}\chi_{i}^{-(i+2)}) \in {\mathfrak b}$.
A similar set of eliminations using the expression for $\euler(N_{i,n+1})$
shows that
$e_i(1 - \chi_i^{-(i+2)}\chi_{i+1}^{-(i+3)}) \in {\mathfrak b}$. Since this $\chi_i$
are units, 
$e_i(1 - \chi_i^{-(i+1)}\chi_{i+1}^{-(i+2)}) \in {\mathfrak b}$.
Hence $e_i(-\chi_i^{-(i+2)(i+1)}\chi_{i+1}^{-(i+2)^2} + \chi_i^{-(i+1)(i+2)}\chi_{i+1}^{(i+1)(i+3)}) = \chi_i^{-(i+2)(i+1)}\chi_{i+1}^{-(i+1)(i+3)} e_ie_{i+1}$. A similar
calculation shows that $e_i^2 \in {\mathfrak b}$.

The calculation for Chow groups is analogous where the Chow-theoretic 
Euler class of the bundles $N_{i,j}$ are expressed as $t_it_j$ when $j < n$
and $\euler(N_{i,n}) = t_i(t_0 + 2t_1 + \ldots n t_{n-1})$,
while $\euler(N_{i,n+1}) = t_i(-2t_0 -3t_1 - \ldots -(n+1)t_{n-1})$.
\end{proof}

\begin{thm} \label{prop:HKRC}
  Let $X(\Sigma'_n)$ be the crepant resolution of singularities of the moduli
  space of $T^*\Pro(1,n)$ indicated by the toric diagram above.  Then
for $n=2,3$  there are isomorphisms of augmented 
$\lam$-algebras over $\nc$.
  \[\Kh(I\Pro(1,n))_{\nc} \to 
K(X(\Sigma'_n))_\nc,\] 
where the 
the augmentation completion
$\Kh(I\Pro(1,n))_{\nc}$ has the inertial $\lambda$-ring structure described above.
\end{thm}
\begin{proof}
  We have calculated 
$K(I\Pro(1,2))_{\nc }$
 and 
$K(I\Pro(1,3))_{\nc}$,
 and in 
 both cases we obtain an Artin ring that is a quotient of a 
  coordinate ring of a torus of rank $2$ and $3$, respectively.  
The inertial augmentation ideal corresponds to the
  identity in the corresponding torus.  
Thus for $n=2,3$ the ring 
$\Kh(I\Pro(1,n))_\nc$
  is simply the localization of 
$K(I\Pro(1,n))_{\nc}$
 at the
 corresponding maximal ideal.  
 A calculation, which we omit as it is straightforward but lengthy, 
 shows that
 $
\Kh(I\Pro(1,2))_{\nc} 
= \nc[\sigma, \sigma^{-1}, \tau,\tau^{-1}]/\langle \sigma-1,\tau-1\rangle^2$ 
and 
$
\Kh(I\Pro(1,3))_{\nc} 
= \nc[\sigma, \sigma^{-1},\tau, \tau^{-1}, 
  \taub,\taub^{-1}]/\langle \sigma-1 ,\tau-1, \taub-1\rangle^2$, 
which are readily seen to be isomorphic as $\lambda$-rings to 
$K(X(\Sigma'_2))_\nc$ and $K(X(\Sigma'_3))_\nc$, 
respectively.
\end{proof}
\nocite{EdGr:00,ARZ:06}

\newcommand{\etalchar}[1]{$^{#1}$}
\def\cprime{$'$} \def\cprime{$'$} \def\cprime{$'$}
\providecommand{\bysame}{\leavevmode\hbox to3em{\hrulefill}\thinspace}
\providecommand{\MR}{\relax\ifhmode\unskip\space\fi MR }
\providecommand{\MRhref}[2]{%
  \href{http://www.ams.org/mathscinet-getitem?mr=#1}{#2}
}
\providecommand{\href}[2]{#2}

{}
\bibliographystyle{amsalpha}

\end{document}